\documentclass[12pt,reqno]{amsart}

\usepackage{amsfonts,amssymb}
\usepackage{mathdots}
\usepackage[all]{xy}
\usepackage{blkarray,longtable}
\usepackage{mathrsfs,calligra}
\usepackage[bbgreekl]{mathbbol}
  \DeclareSymbolFontAlphabet{\mathbb}{AMSb}
  \DeclareSymbolFontAlphabet{\mathbbl}{bbold}
\usepackage{xcolor}
\usepackage{graphicx}
\usepackage{here,float}

\usepackage{pgf}

\usepackage[vcentermath,enableskew]{youngtab}
\usepackage{ytableau}
\ytableausetup{centertableaux,nobaseline}


\newcommand{\version}{Ver.~0.0}
\newcommand{\setversion}[1]{\renewcommand{\version}{Ver.~{#1}}}
\setversion{0.0 [2022/12/28 12:16:43 JST]}
\setversion{0.2 [2023/03/01 00:43:23 JST]}
\setversion{0.9 [2023/09/21 00:34:01 JST]}
\setversion{0.98 [2023/09/25 10:59:16 JST]}
\setversion{0.99 [2023/09/29 18:05:02 JST]}

\title{Overview on the theory of double flag varieties for symmetric pairs}
\dedicatory{Dedicated to the memory of Joe Wolf}

\author{Lucas Fresse and Kyo Nishiyama}
\address{Universit\'e de Lorraine, CNRS, Institut \'Elie Cartan de Lorraine, UMR 7502, Vandoeu\-vre-l\`es-Nancy, F-54506, France}
\email{lucas.fresse@univ-lorraine.fr}
\address{Department of Mathematics, Aoyama Gakuin University, Fuchinobe 5-10-1, Chuo-ku, Sagamihara 252-5258, Japan}
\email{kyo@math.aoyama.ac.jp}
\thanks{K.~N.~is supported by JSPS KAKENHI Grant Number \#{21K03184}.}

\newcounter{mypart}  % new conter
\newcommand{\partlabel}[1]{%
\refstepcounter{mypart}% increase the counter
\label{#1}% defining label
}

\numberwithin{equation}{section}

\newtheorem{theorem}{Theorem}[section]
\newtheorem{lemma}[theorem]{Lemma}
\newtheorem{proposition}[theorem]{Proposition}
\newtheorem{corollary}[theorem]{Corollary}

\newtheorem{conjecture}[theorem]{Conjecture}
\theoremstyle{definition}
\newtheorem{example}[theorem]{Example}
\newtheorem{notation}[theorem]{Notation}
\newtheorem{definition}[theorem]{Definition}
\newtheorem{remark}[theorem]{Remark}

\newtheorem{problem}[theorem]{Problem}
\newtheorem{assumption}[theorem]{Assumption}

\newcounter{penum}
\newenvironment{penumerate}{%
\par\smallskip
\begin{list}{$\;\;(\thepenum)$}{%
\usecounter{penum}
\setlength{\topsep}{0pt}
\setlength{\partopsep}{0pt}
\setlength{\parsep}{1ex}
\setlength{\itemindent}{0pt}
\setlength{\labelsep}{.5em}
\setlength{\labelwidth}{0pt}
\setlength{\leftmargin}{0pt}
\setlength{\rightmargin}{0pt}
\setlength{\itemsep}{0pt}
}}
{\end{list}\par}
\newcommand{\skipover}[1]{}

\newcommand{\smallvstrut}{\vphantom{\Bigm|}}

\newcommand{\Z}{\mathbb{Z}}

\newcommand{\R}{\mathbb{R}}
\newcommand{\bbB}{\mathbb{B}}
\newcommand{\bbH}{\mathbb{H}}
\newcommand{\bbK}{\mathbb{K}}

\newcommand{\bbG}{\mathbb{G}}
\newcommand{\C}{\mathbb{C}}
\newcommand{\Cbatsu}{\mathbb{C}^{\times}}
\newcommand{\bbF}{\mathbb{F}}
\newcommand{\bbP}{\mathbb{P}}
\newcommand{\bbQ}{\mathbb{Q}}
\newcommand{\bbU}{\mathbb{U}}

\newcommand{\bbX}{\mathbb{X}}
\newcommand{\mbfa}{\mathbf{a}}
\newcommand{\mbfb}{\mathbf{b}}
\newcommand{\mbfc}{\mathbf{c}}

\newcommand{\I}{\mathcal{I}}

\newcommand{\Xfv}{\mathfrak{X}}
\newcommand{\Grass}{\mathrm{Gr}}
\newcommand{\Lie}{\mathrm{Lie}}
\newcommand{\SL}{\mathrm{SL}}
\newcommand{\GL}{\mathrm{GL}}

\newcommand{\Sp}{\mathrm{Sp}}
\newcommand{\SO}{\mathrm{SO}}
\newcommand{\OO}{\mathrm{O}}
\newcommand{\rk}{\mathrm{rk}\,}
\newcommand{\Mat}{\mathrm{M}}

\newcommand{\diag}{\qopname\relax o{diag}}

\newcommand{\Hom}{\qopname\relax o{Hom}}
\newcommand{\End}{\qopname\relax o{End}}

\newcommand{\Stab}{\qopname\relax o{Stab}}
\newcommand{\Rep}{\qopname\relax o{Rep}}

\newcommand{\Irr}{\qopname\relax o{Irr}}

\newcommand{\Aut}{\qopname\relax o{Aut}}
\newcommand{\Ad}{\qopname\relax o{Ad}}

\newcommand{\unitmatrix}{\mathbbl{1}}

\newcommand{\columnstrip}{\subset\!\!\!\!\!\cdot\,\,\,}

\renewcommand{\Im}{\qopname\relax o{Im}}

\newcommand{\vectwo}[2]{{\renewcommand{\arraystretch}{.85}\Bigl(\begin{array}{@{\,}c@{\,}}{#1}\\ {#2}\end{array}\Bigr)}}

\newcommand{\mattwo}[4]{\Bigl(\begin{array}{@{\,}c@{\;\;}c@{\,}}{#1} & {#2} \\ {#3} & {#4} \end{array}\Bigr)}

\newcommand{\dblFV}{\Xfv}
\newcommand{\bbdblFV}{\mathbb{X}}
\newcommand{\tripleFV}{\mathscr{X}}
\newcommand{\multipleFV}{\mathscr{X}}

\newcommand{\Flag}{\mathscr{F}}

\newcommand{\Flags}{\mathscr{F}\!\ell}
\newcommand{\FlagVar}{\mathscr{F}\!\ell}
\newcommand{\JVar}{\mathbf{JF}}
\newcommand{\MultQRep}{\mathscr{MF}}
\newcommand{\JQRep}{\mathscr{J}}

\newcommand{\FlP}{\mathscr{P}}
\newcommand{\FlB}{\mathscr{B}}

\newcommand{\lie}[1]{\mathfrak{#1}}
\newcommand{\nilradical}[1]{\lie{u}_{#1}}
\newcommand{\orbit}{\mathbb{O}}
\newcommand{\bborbit}{\mathbb{O}}
\newcommand{\calorbit}{\mathcal{O}}

\newcommand{\closure}[1]{\overline{#1}}

\newcommand{\transpose}[1]{\,{}^t{#1}}

\newcommand{\norm}[1]{\|{#1}\|}

\newcommand{\nilpotents}{\mathcal{N}}
\newcommand{\nilpotentsof}[1]{\mathcal{N}_{#1}}

\newcommand{\eb}{\boldsymbol{e}}

\newcommand{\signedyd}[1]{\mathcal{SYD}(#1)}

\newcommand{\genRS}{\qopname\relax o{gRS}}

\newcommand{\Rowinsert}{\mathrm{Rowinsert}}

\newcommand{\ST}[1]{\mathrm{STab}({#1})}

\newcommand{\Tquintuples}[1]{\mathcal{T}_{#1}}

\newcommand{\nopicture}[1]{}

\newcommand{\fk}{\mathfrak{k}}
\newcommand{\fs}{\mathfrak{s}}
\newcommand{\fg}{\mathfrak{g}}
\newcommand{\fp}{\mathfrak{p}}
\newcommand{\fn}{\mathfrak{n}}
\newcommand{\fq}{\mathfrak{q}}

\newcommand{\gl}{\mathfrak{gl}}
\newcommand{\fb}{\mathfrak{b}}

\newcommand{\RS}{\mathrm{RS}}
\newcommand{\shape}{\qopname\relax o{\mathbf{sh}}}

\newcommand{\Tr}{\qopname\relax o{Tr}}

\newcommand{\rank}{\qopname\relax o{rank}}

\newcommand{\nrboxes}[2]{\# #1_{\leq #2}}
\newcommand{\nrplus}[2]{\# #1_{\leq #2}(+)}
\newcommand{\nrminus}[2]{\# #1_{\leq #2}(-)}

\newcommand{\flag}{\mathcal{F}}

\newcommand{\conormal}{\mathcal{Y}}
\newcommand{\Xorbit}{\mathbb{O}}

\newcommand{\partitionsof}[1]{\mathcal{P}(#1)}
\newcommand{\permutationsof}[1]{\mathfrak{S}_{#1}}

\newcommand{\ppermutations}{\mathfrak{T}}
\newcommand{\regTnxTn}{\ppermutations_n^{\circ}}
\newcommand{\parameters}{\overline{\ppermutations}}

\newcommand{\graphic}{\Gamma}
\newcommand{\Indecomposable}{\mathrm{Indec}}
\newcommand{\ModQ}{\mathfrak{M}}

\newcommand{\bbd}{\boldsymbol{d}}
\newcommand{\bbg}{\boldsymbol{g}}
\newcommand{\bbk}{\boldsymbol{k}}
\newcommand{\bbp}{\boldsymbol{p}}
\newcommand{\bbu}{\boldsymbol{u}}
\newcommand{\bbx}{\boldsymbol{x}}

\newcommand{\bbPhi}{\mathbbl{\Phi}}
\newcommand{\bbBK}{\bbB_{\bbK}}
\newcommand{\restrict}{\big|}
\renewcommand{\bbphi}{\phi_{\bbdblFV}}
\newcommand{\LGrass}{\qopname\relax o{LGr}}

\newcommand{\regMat}{\mathrm{M}^{\circ}}
\newcommand{\regCMat}{\mathscr{S}^{\circ}}
\newcommand{\regCnxCn}{\mathscr{C}_n^{\circ}}

\newcommand{\Sym}{\mathop{\mathrm{Sym}}\nolimits}

\newcommand{\vdotsshiftedup}{\raisebox{1.2ex}{$\vdots$}}

\newcommand{\omegaI}{\mbox{\tiny
$\begin{picture}(85,20)(0,0)
\put(0,11){$\bullet$}
\put(2,13){\circle{6}}
\put(0,18){$1^+$}
\put(18,11){$\cdots$}
\put(40,11){$\bullet$}
\put(40,18){$i^+$}
\put(58,11){$\cdots$}
\put(80,11){$\bullet$}
\put(80,18){$p^+$}
\put(82,13){\circle{6}}
\put(0,-11){$\bullet$}
\put(0,-20){$1^-$}
\put(42,12){\line(-2,-1){38}}
\end{picture}$}}

\newcommand{\omegaII}{\mbox{\tiny
$\begin{picture}(85,20)(0,0)
\put(0,11){$\bullet$}
\put(2,13){\circle{6}}
\put(0,18){$1^+$}
\put(18,11){$\cdots$}
\put(40,11){$\bullet$}
\put(40,18){$i^+$}
\put(58,11){$\cdots$}
\put(80,11){$\bullet$}
\put(80,18){$p^+$}
\put(82,13){\circle{6}}
\put(0,-11){$\bullet$}
\put(0,-20){$1^-$}
\put(2,-9){\circle{6}}
\end{picture}$}}

\newcommand{\omegaIII}{\mbox{\tiny
$\begin{picture}(85,20)(0,0)
\put(0,11){$\bullet$}
\put(2,13){\circle{6}}
\put(0,18){$1^+$}
\put(38,11){$\cdots$}
\put(20,11){$\bullet$}
\put(20,18){$2^+$}
\put(22,13){\circle{6}}
\put(58,11){$\cdots$}
\put(80,11){$\bullet$}
\put(80,18){$p^+$}
\put(82,13){\circle{6}}
\put(0,-11){$\bullet$}
\put(0,-20){$1^-$}
\end{picture}$}}

\newcommand{\GIiplus}{\mbox{\tiny
$\begin{picture}(6,20)(0,0)
\put(0,0){$\bullet$}
\put(0,8){$i^+$}
\put(2,1){\circle{8}}
\end{picture}$}}
\newcommand{\GIiprimeplus}{\mbox{\tiny
$\begin{picture}(6,20)(0,0)
\put(0,0){$\bullet$}
\put(0,8){$i^+$}
\end{picture}$}}
\newcommand{\GIjminus}{\mbox{\tiny
$\begin{picture}(6,20)(0,0)
\put(0,0){$\bullet$}
\put(0,-10){$j^-$}
\put(2,1){\circle{8}}
\end{picture}$}}
\newcommand{\GIjprimeminus}{\mbox{\tiny
$\begin{picture}(6,20)(0,0)
\put(0,0){$\bullet$}
\put(0,-10){$j^-$}
\end{picture}$}}
\newcommand{\GIij}{\mbox{\tiny
$\begin{picture}(6,20)(0,0)
\put(1,11){$\bullet$}
\put(1,17){$i^+$}
\put(1,-11){$\bullet$}
\put(1,-19){$j^-$}
\put(3,12){\line(0,-1){20}}
\end{picture}$}}

\begin{document}

\maketitle

\begin{abstract}
Let $ G $ be a connected reductive algebraic group and {$ K $ its symmetric subgroup fixed by an involution}.  
The variety $ \dblFV = K/Q \times G/P $ is called a double flag variety, 
where $ Q $ and $ P $ are parabolic subgroups of $ K $ and $ G $ respectively.   

In this article, we make a survey on 
{the current status of the research and results related to }
%%the theory of 
double flag varieties for a symmetric pair $ (G, K) $ 
and report {a number of new results and theorems}.
{(By the word ``theory'', in the title, we mean 
a series of general results on finiteness and parametrization of orbits under the action of suitable algebraic groups, the geometry related to orbits, and infinitesimal character of the action including relation with nilpotent orbits.)}

Most important topic {in this article} is the finiteness of $ K $-orbits on $ \dblFV $.  
We summarize the classification of $ \dblFV $ of finite type, which are scattered in the {literature}.  
In some respects such classifications are complete, and in some cases not.  
In particular, we get a classification of double flag varieties of finite type when 
a symmetric pair is of type AIII, using the theorems of Homma who describes ``indecomposable'' objects of such double flag varieties.  
Together with these classifications, newly developed embedding theory provides double flag varieties of finite type, which are new.  
Other ingredients in this article are Steinberg theory, generalization of Robinson-Schensted correspondence, and orbit classification via quiver representations.

We hope this article is useful for those who want to study the theory of double flag varieties.
\end{abstract}

%\vspace*{12pt}

\clearpage

\tableofcontents
%type this to generate out the chapter content text. 

\section{Introduction}

Let $ G $ be a connected reductive algebraic group over the complex number field $ \C $ 
 and $ \theta \in \Aut G $ an involutive automorphism.  
There will be no harm in considering an algebraically closed field of characteristic zero, but we use $ \C $ for simplicity.  
The fixed point subgroup $ K = G^{\theta} $ is called a \emph{symmetric subgroup} and $ (G, K) $ is called a \emph{symmetric pair}.  
The subgroup $ K $ is also reductive, but it is not connected in general.  
In this paper, we assume $ K $ is connected, although most of the theory will go fine if we take the connected component of $ G^{\theta} $ as $ K $.  

We concentrate on the cases where $ G $ is classical.  
Let us list irreducible classical symmetric pairs referring to \cite[\S~10.3]{Helgason.DG.1978}.

\begin{itemize}
\item[(AI)]
$ (\SL_n, \SO_n) $ \qquad\qquad \S\S~\ref{subsection:Table.P=Borel} and \ref{subsection:Table.Q=Borel}; \; \S~\ref{subsec:AII}
\item[(AII)]
$ (\SL_{2n}, \Sp_{2n}) $ \qquad\qquad \S\S~\ref{subsection:Table.P=Borel} and \ref{subsection:Table.Q=Borel}; \; \S~\ref{subsec:AI}
\item[(AIII)]
$ (\SL_n, (\GL_p \times \GL_q) \cap \SL_n) \; (n = p + q) $ \quad 
\qquad\qquad Table~\ref{table:AIII.finite.type} in \S~\ref{subsec:classification.table.typeAIII.DFV}
\\
(Below, we consider $ (\GL_n, \GL_p \times \GL_q) $, which is easier to handle.) 
\item[(BDI)]
$ (\SO_n, \SO_p \times \SO_q) \; (n = p + q) $ \quad \S\S~\ref{subsection:Table.P=Borel} and \ref{subsection:Table.Q=Borel}; \; \S\S~\ref{subsec:BI}, \ref{subsec:DI-even}, \ref{subsec:DI-odd}
\item[(DIII)]
$ (\SO_{2n}, \GL_n) $ \qquad\qquad \S\S~\ref{subsection:Table.P=Borel} and \ref{subsection:Table.Q=Borel}; \; \S~\ref{subsec:DIII}
\item[(CI)]
$ (\Sp_{2n}, \GL_n) $ \qquad\qquad \S\S~\ref{subsection:Table.P=Borel} and \ref{subsection:Table.Q=Borel}; \; \S~\ref{subsec:CI}
\item[(CII)]
$ (\Sp_{2n}, \Sp_{2p} \times \Sp_{2q}) \; (n = p + q) $ \qquad\qquad \S\S~\ref{subsection:Table.P=Borel} and \ref{subsection:Table.Q=Borel}; \; \S~\ref{subsec:CII}
\end{itemize}

\begin{remark}
\begin{penumerate}
\item
For each type we mention relevant sections, which contain the tables of double flag varieties of finite type (for details, see the explanation below).  
For the tables of the triple flag varieties of finite type, 
see the tables in \S\S~\ref{subsec:Table.3FV.Borel} and \ref{subsec:Table.3FV.general}.  
\item
Sometimes (BDII) refers to $ (\SO_{n + 1}, \SO_n) $ (see \cite[\S~6.3]{Mimura.Toda.1991}, for example), 
but we include this into (BDI) for convenience.  
Further, in occasions, we separate (BDI) into (BI) and (DI) according as $ n $ is odd or even. 
\end{penumerate}
\end{remark}

Let $ (G, K) $ be a symmetric pair.  
For parabolic subgroups $ Q $ of $ K $ and $ P $ of $ G $, 
we call $ \dblFV = K/Q \times G/P $ a \emph{double flag variety} on which $ K $ acts diagonally.  
The study of double flag varieties started a decade ago by \cite{NO.2011}.  
Since then, there appeared many works on them; we mention  
\cite{Henderson.Trapa.2012,Kobayashi.Oshima.2013,HNOO.2013,Kobayashi.Mastuki.2014,Fresse.N.2016,Fresse.N.2020,Fresse.N.2021,Fresse.N.2023,Nishiyama.Orsted.2018,Homma.2021}
for example.  
Double flag varieties are intimately related to triple flag varieties (in fact, such triple flag varieties are considered to be special cases of double flag varieties, see \S~\ref{section:mult.FV}).  
There are also many works on these triple/multiple flag varieties, e.g.,  
 \cite{MWZ.1999,MWZ.2000,Stembridge.2003,Travkin.2009,FGT.2009,Matsuki.2013,Matsuki.2015,Matsuki.arXiv2019,BSEKG.2017.arXiv}.  
The main interests of these papers are proving finiteness of orbits, classifications of multiple flag varieties with finitely many orbits,  and generalizations of Steinberg theory including combinatorics such as Robinson-Schensted correspondence.  
There have been many results including various classification tables of double/triple flag varieties with finitely many orbits 
and they scattered over in different places.  

In this paper, we summarize important properties, classification results and tables in the theory of double flag varieties from various resources, and explain general theory of Steinberg maps.  
We try to keep the description understandable not only for experts but also for readers from different fields.  
In the course of making summary, we have obtained many of entirely new results or improvements on existent theorems.  
These new results, which advance our understanding of the theory, are in Parts~\ref{part:finite.type.DFV}--\ref{part:Orbit.Embedding}.

Let us explain the content of the paper.
\medskip

In Part~\ref{part:DFV.symmetric.pairs}, we introduce basics on double/triple flag varieties.  
We begin with the well known theory of flag varieties and Bruhat decompositions.  
{We are interested in double or triple flag varieties with the action of $ G $ or $ K $.}  
If $ X $ is a smooth $ G $-variety, then the action of $ G $ induces a Hamiltonian action on the cotangent bundle $ T^*X $ and it provides a moment map.  
We introduce the \emph{conormal variety} as the null fiber of the moment map, and describe the irreducible decomposition of the conormal variety using conormal bundles of orbits, 
if there exist only finitely many orbits.  
{In this respect, the finiteness of orbits plays a central role.  We summarize various results related to orbit structures in double/triple flag varieties as well as the classical results on (a single) flag varieties.}

\medskip

In Part~\ref{part:finite.type.DFV}, we discuss finiteness of orbits on double flag varieties, which is crucial in the Steinberg theory developed below.  
If there are only finitely many $ K $-orbits on $ \dblFV = K/Q \times G/P $, $ \dblFV $ is called of \emph{finite type}.  
Since $ \dblFV /K \simeq Q \backslash G/P $, it is equivalent to having finitely many $ Q $ orbits on $ G/P $.  
In particular, if $ Q = B_K $ is a Borel subgroup of $ K $, 
classification of the double flag {varieties} of finite type amounts that of $ K $-spherical partial flag varieties $ G/P $.  This classification is achieved by He-N-Ochiai-Oshima \cite{HNOO.2013}.  
They also classified the double flag {varieties} of finite type in the case where $ P = B_G $ is a Borel subgroup of $ G $.  
We reproduce the classification in \S~\ref{section:finite.DFV.P.or.Q=Borel} together with the key ideas of the proofs.  In fact the classification is reduced to that of multiplicity free action on certain space of nilpotent elements (see \S~\ref{subsec:ideas.proofs.Q=BK.or.P=BG}).  
Together with these results, we also collect the tables of triple flag varieties of finite type in \S\S~\ref{subsec:Table.3FV.Borel} and \ref{subsec:Table.3FV.general}.
They are quoted from \cite{MWZ.1999,MWZ.2000}, \cite{Stembridge.2003}, \cite{Matsuki.2015,Matsuki.arXiv2019}.  
These classifications are used later to obtain further examples of double flag varieties of finite type.  

To classify general double flag varieties of finite type, 
it is useful to apply representation theory of quivers and consider suitable subcategories of representations.  
The usage of quiver representations was first introduced by Magyar-Weyman-Zelevinsky \cite{MWZ.1999} 
for triple flag varieties 
based on the theory of Gabriel \cite{Gabriel.1972}, and Kac \cite{Kac.Quiver.1980,Kac.QuiverII.1982}.  
We explain core ideas of these theories in \S~\ref{sec:quiver.joint.variety}.  

For the classification of double flag varieties of finite type (for type AIII), 
Homma \cite{Homma.2021} introduced \emph{joint flag category} and \emph{joint flag varieties} which correspond 
to the quiver representations associated with the flags in the double flag varieties.  
Using his idea and results, we classify the double flag varieties of finite type in the case of the symmetric pairs of type AIII.  Homma actually classified 
indecomposable quiver representations of finite type, and it is not immediate to get the explicit form of double flag varieties.  
Our Table~\ref{table:AIII.finite.type} in \S~\ref{subsec:classification.table.typeAIII.DFV}  presents precisely when given parabolic subgroups 
 contribute to the double flag varieties of finite type, and 
this completes the classification in the case of type AIII (Corollary~\ref{cor:classification.AIII.finite.type}).  

For the symmetric pair $ (G, K) = (\GL_n, \GL_p \times \GL_q) \; (n = p + q) $ of type AIII, 
we also present a classification of $ K $-orbits on $ \dblFV = K/B_K \times G/P $ in terms of indecomposable quiver representations, 
where $ B_K $ is a Borel subgroup of $ K $ and $ P $ is a maximal parabolic subgroup of $ G $.  
See Theorem~\ref{theorem:orbits.AIII} in \S~\ref{subsec:orbit.by.quivers.typeAIII}.
%%$ \GL_p/B_p \times \GL_q/B_q \times \GL_n/P_{(r, n - r)} $

\medskip

In Part~\ref{part:Steinberg.theory}, we explain the Steinberg theory for double flag varieties.  

For a symmetric pair $ (G, K) $, denote by $ \lie{g} $ and $ \lie{k} $ the Lie algebras of $ G $ and $ K $ respectively.  
Let $ \lie{g} = \lie{k} + \lie{s} $ be the Cartan decomposition, where $ \lie{s} $ is the $ (-1) $-eigenspace of $ \theta \in \Aut \lie{g} $ (we denote the differential of $ \theta $ by the same letter).  
Note that $ \lie{s} $ is isomorphic to the tangent space of $ G/K $ at the base point $ eK $, on which $ K $ acts.  
We denote by $ \nilpotents $ the nilpotent variety of $ \lie{g} $ and 
put $ \nilpotentsof{\lie{k}} = \nilpotents \cap \lie{k} $ and $ \nilpotentsof{\lie{s}} = \nilpotents \cap \lie{s} $.  
Due to the work of Kostant, Mal'cev and Kostant-Rallis (\cite{Malcev.1944}, \cite{Kostant.1963,Kostant.Rallis.1971}, see also \cite{Collingwood.McGovern.1993}) it is known that 
there are only finitely many $ K $-orbits in both of the nilpotent varieties $ \nilpotentsof{\lie{k}} $ and $ \nilpotentsof{\lie{s}} $.  

If $ \dblFV $ has finitely many $ K $-orbits, then there exists a natural orbit map 
$ \Phi_\fk : \dblFV/K \to \nilpotentsof{\lie{k}}/K $ by using moment maps.  
In the classical case where $ \dblFV = G/B \times G/B $ is the product of two full flag varieties on which $ K = G $ acts diagonally, 
this reduces to the classical Steinberg's theory \cite{Steinberg.1988}.  Namely, there exists a map 
$ W_G \simeq (G/B \times G/B)/G \to \nilpotents/G $ called the \emph{Steinberg map}, where $ W_G $ denotes the Weyl group of $ G $.    
In the case of type A, the set of nilpotent orbits $ \nilpotents/G $ is parameterized by Young diagrams (or partitions) and the fibers of the Steinberg map are identified with the pairs of 
standard Young tableaux.  In such a way, the Steinberg map is identified with the \emph{Robinson-Schensted correspondence} (RS correspondence), a combinatorial procedure.  
Therefore, in our situation, the orbit map $ \Phi_\fk : \dblFV/K \to \nilpotentsof{\lie{k}}/K $ thus obtained is a generalization of the Steinberg map.  
In addition, in the case of double flag variety, we get one \emph{extra} map $ \Phi_\fs : \dblFV/K \to \nilpotentsof{\lie{s}}/K $ if certain condition called ``nil-projection assumption'' holds 
(see Assumption~\ref{assumption:phis} in \S~\ref{subsec:two.Steinberg.maps.nil-projection.assumption} for details).
The map $ \Phi_\fk $ is called the \emph{symmetrized Steinberg map} and $ \Phi_\fs $ is called \emph{exotic} one.  
In \S~\ref{section:generalization.RS}, for symmetric pairs of type AIII, we explain combinatorial procedures to get these two Steinberg maps, which are {generalizations} of RS correspondence.  

In this respect, Nil-projection Assumption is very important.  We propose a conjecture that says it always holds whenever the double flag variety is of finite type.  
For example, this is valid when the parabolic subgroup $ P \subset G $ has abelian unipotent radical.  
See Conjecture~\ref{conj:Im.exotic.mm.in.nilpotent.variety} and Proposition~\ref{proposition:image:nilpotent.classical}.

\medskip

Part~\ref{part:Orbit.Embedding} is devoted to the study of embedding theory of double flag varieties.  
Since the double flag varieties of finite type are completely classified when the symmetric pair is of type AIII, 
if we can embed a double flag variety of another type into one of type AIII, so that the structure of orbits is preserved, then we get various examples of double flag varieties of finite type.   
In \cite{Fresse.N.2021}, this attempt is carried out for the double flag varieties of type CI.  
In this paper, we generalize it to a large family of double flag varieties (Theorem~\ref{thm:embedding.orbits} and Proposition~\ref{prop:condition.E} in \S~\ref{subsec:Embedding.AIII})
and get fairly large, new families of double flag varieties of finite type (see \S\S~\ref{subsec:AII}--\ref{subsec:DIII}).  
The embedding theory is equally applicable to triple flag varieties (Theorem~\ref{thm:DFV.embedded.into.3FV} in \S~\ref{subsec:Embedding.3FV}), and we also recover the results in \cite{NO.2011}.  

The advantage of the embedding theory is that it also provides the classification of orbits in terms of the orbits in type AIII.  
This corresponds to the classical theory of Jordan type of nilpotent orbits for general classical groups.  
We give such examples in \S~\ref{section:embedding.typeCI.into.typeAIII}, which are taken from \cite{Fresse.N.2021}.
One more advantage of the embedding theory is that it behaves functorially in Steinberg theory.  
However, we still do not prove the full commutativity between embeddings and Steinberg maps, which are left to Conjecture~\ref{conj:question.compatibility}.

\medskip

{Finally, let us make some remarks on open questions and related future problems.}
%%important remarks.  

Although we get a large portion of double flag varieties of finite type (we believe so), 
the complete classification of them is still open.  
Also, in this article, we only deal with classical groups, and classification of double flag varieties of finite type for exceptional groups {is} widely open.  
Parametrizations of $K$-orbits on $ \dblFV $ should be interesting and 
we are expecting new combinatorial structures on orbits and Steinberg maps, which are fruitful subjects.

We only work over algebraically closed field, namely over $ \C $, but the same problem over the field of real numbers (or even over $ p $-adic numbers) will be interesting  
from the view point of infinite dimensional unitary representations (see \cite{Kobayashi.Mastuki.2014,Kobayashi.Oshima.2013} and \cite{Nishiyama.Orsted.2018} for example).  

The $ K $-orbit space on $ \dblFV $ admits a natural action of Hecke algebras (see our recent paper \cite{Fresse.N.2024}), and 
there should be a rich theory like Springer theory, or categorical study over perverse sheaves since there are finitely many orbits.  
In this respect, we mention \cite{FGT.2009,Travkin.2009} and \cite{SSI.2013,SSII.2014,SSIII.2014}.  

\medskip

As we mention above the tables in this paper do not exhaust the double flag varieties of finite type (in some cases, it is complete though).  
However, they give rich examples of them, and we indicate the sections and table numbers at the beginning of this introduction together with the classification of the classical symmetric pairs.  
We hope they are useful for readers.

\part{\textbf{Double flag varieties for symmetric pairs}}
%%\section*{\textbf{Part I. Double flag varieties for symmetric pairs}}
%%\addcontentslinentsline{toc}{section}{\textbf{Part I. Double flag varieties for symmetric pairs}\partstrut}
\partlabel{part:DFV.symmetric.pairs}

\section{Flag varieties and moment maps}\label{sec:flag.varieties.and.moment.maps}

In this section we introduce (without proofs) basic notions related to flag varieties, which are needed later.  
For the proofs and more detailed description, see \cite{Borel.1991,Humphreys.LAG.1975,Springer.1998} and references therein.
Also we will review some facts on moment maps.  For this we refer the readers to \cite{Chriss.Ginzburg.1997}.

\subsection{Flag varieties}\label{subsec:flag.varieties}

Let $ G $ be a connected reductive algebraic group over the complex number field $ \C $.  
Most of the story in this article goes through over an algebraically closed field of characteristic zero without changes. {(See, for instance, \cite[\S3]{Losev} for the general background on Hamiltonian actions and moment maps over algebraically closed fields.)}
%%(or even over a field of good characteristic, with some modifications).    
However, we prefer a simpler assumption which will highlight the key points of the theory.

A subgroup $ B = B_G $ of $ G $ is called a Borel subgroup if it is maximal among connected solvable subgroups of $ G $.  
It is well known that all the Borel subgroups are conjugate to each other by inner automorphisms of $ G $, 
and moreover self-normalizing, i.e.,  $ B = N_G(B) = N_G(\lie{b}) $, where $ \lie{b} $ denotes the Lie algebra of $ B $.  

\begin{theorem}
The homogeneous space $ G/B $ is a projective variety (hence compact).  
As a set, it can be identified with the set of all Borel subgroups in $ G $, 
as well as the set of all Borel subalgebras 
%$ \lie{b'} $ 
in $ \lie{g} = \Lie(G) $.
\end{theorem}

The projective variety $ \FlB = G/B $ is called a \emph{complete flag variety} (or a \emph{full} flag variety), and is independent of choice of $ B $.  The adjective ``complete'' is often omitted, but it is sometimes useful to tell a difference from ``partial'' flag varieties.

A closed subgroup $ P $ of $ G $ is called \emph{parabolic} if it contains a Borel subgroup.  
Then $ P $ is automatically connected.  
As well as Borel subgroups, $ P $ is self-normalizing: 
$ P = N_G(P) = N_G(\lie{p}) $, where $ \lie{p} $ is the Lie algebra of $ P $.  
In the following, we will denote the Lie algebra of a reductive (or ``Lie'') group $ G $ by its corresponding lower case German letter $ \lie{g} $.

\begin{theorem}
A homogeneous space $ G/P $ is projective if and only if $ P $ is parabolic.  
As a set, 
it can be identified with the set of all parabolic subgroups in $ G $ that are conjugate to $ P $.  
Similarly, $ G/P $ can be identified with the set of all parabolic subalgebras in $ \lie{g} $ 
that are conjugate to $ \lie{p} $.  
\end{theorem}

If $ P $ is a parabolic subgroup, 
the projective variety $ \FlP = G/P $ is called a (\emph{partial}) \emph{flag variety}.  
There is a natural projection $ G/B \to G/P $ with fiber $ P/B $ which is also a flag variety for 
a smaller group, namely, a Levi subgroup of $ P $.  
We often omit the word ``partial'' and simply call flag variety if there is no need to distinguish 
``complete'' ones and ``partial'' ones.

In the following, the cotangent bundle of flag varieties will play an important role.  
Since the tangent space at $ \lie{b} \in G/B = \FlB $ is $ \lie{g}/\lie{b} $ 
(note that the {stabilizer} of $ \lie{b} $ is $ N_G(\lie{b}) = B $), 
we get the cotangent space $ T_{\lie{b}}^{\ast} \FlB = (\lie{g}/\lie{b})^{\ast} $ at $ \lie{b} $.  
By choosing a non-degenerate invariant bilinear form on $ \lie{g} $,  
we can identify $ (\lie{g}/\lie{b})^{\ast} $ with the nilradical $ \lie{n} $ of $ \lie{b} $.  
\skipover{Here $ B = T N $ is a Levi decomposition with a choice of maximal torus $ T $ and $ N $ is the unipotent radical.  }
Thus the cotangent bundle can be described as
\begin{equation}
T^*\FlB 
\simeq \{ (\lie{b_1}, x) \in \FlB \times \lie{g} \mid x \in \lie{n_1} \} 
\simeq G \times_B \lie{n} ,
\end{equation}
where in the second expression $ \FlB $ is considered to be the set of Borel subalgebras as explained above and 
$ \lie{n_1} $ is the nilradical of a Borel subalgebra $ \lie{b_1} $.

In a very similar way, we can describe the cotangent bundle of a partial flag variety $ \FlP = G/P $.  
So let $ \lie{p_1} \in \FlP $ be a parabolic subalgebra conjugate to $ \lie{p} $, 
and denote its nilpotent radical by 
$ \lie{u_1} $ (we will denote a Levi decomposition by $ P = L U $, where $ L $ is a Levi part and $ U $ its unipotent radical in general).
Then we get
\begin{equation}
T^*\FlP \simeq G \times_P \lie{u} 
\simeq \{ (\lie{p_1}, x) \in \FlP \times \lie{g} \mid x \in \lie{u_1} \} .
\end{equation}

\subsection{Moment maps}\label{subsec:moment.map.for.FlB}

Let $ X $ be a smooth variety on which $ G $ acts.  We call $ X $ a smooth $ G $-variety for short.  
The cotangent bundle $ T^*X $ inherits a natural structure of symplectic variety which is $ G $-invariant.  
Choose $ (x, \xi) \in T^*X $, where $ x \in X $ and $ \xi \in T_x^*X $.  
The tangent space of $ T^*X $ at $ (x, \xi) $ is 
$ T_x X \oplus T_\xi (T_x^* X) \simeq  T_x X \oplus T_x^* X $, and 
there is a natural symplectic form $ \omega $, for which $ T_x X \oplus T_x^* X $ is a polar decomposition by 
Lagrangian subspaces $ T_x X $ and $ T_x^* X $.   
{To get the symplectic form, we adopt the natural pairing between $ T_x X $ and $ T_x^* X $, which is the dual of $ T_x X $.}  

Pick a point $ x \in X $ and consider a map $ v_x : G \to X, \;\; v_x(g) = g x $.   
The differential (at $ e \in G $) of this map induces a map from $ \lie{g} $ to $ T_x X $, which we denote by the same letter $ v_x $ by abuse of notation.  
For $ A \in \lie{g} $, $ v_x(A) $ is simply a vector field generated by the action of $ \exp t A $ on $ X $.  
Let us consider a $ G $-equivariant map
\begin{equation}
\mu_X : T^*X \to \lie{g}^* , \qquad
\mu_X((x, \xi))(A) = \xi(v_x(A)) \quad  \text{ for $ A \in \lie{g} $}, 
\end{equation}
which is called a ``moment map'' \footnote{For any Hamiltonian symplectic $ G $-variety, one can define a moment map.  For general definition and properties of moment maps, 
see \cite{Chriss.Ginzburg.1997} and \cite{Kostant.1979}.}.

Let us apply the framework above to flag varieties.  
At a point $ b \in \FlB $ which is represented by a Borel subalgebra $ \lie{b_1} $, 
the tangent space is given by $ T_{b}\FlB = \lie{g}/\lie{b_1} $ and the cotangent space is 
$ T_{b}^*\FlB = (\lie{g}/\lie{b_1})^* \simeq \lie{n_1} $.  
The vector field generated by $ A \in \lie{g} $ is just $ A \mod{\lie{b_1}} $.
Thus we get 
\begin{equation}
\mu_{\FlB} : T^*{\FlB} \to \lie{g}^* , \qquad
\mu_{\FlB}((\lie{b_1}, x))(A) = \langle x, A \rangle \quad  \text{ for $ x \in \lie{n_1}, \; A \in \lie{g} $}, 
\end{equation}
where $ \langle \cdot, \cdot \rangle $ denotes the nondegenerate invariant bilinear form on $ \lie{g} $.  
Note that $ x $ kills $ \lie{b_1} $ so that the last expression is well defined.
By this formula, if we identify $ \lie{g} $ and $ \lie{g}^* $ by the invariant bilinear form (we always do so below),  
\begin{equation}
\mu_{\FlB}((\lie{b_1}, x)) = x \qquad ((\lie{b_1}, x) \in T^*\FlB)
\end{equation}
and it is just the second projection.  
From this formula, since $ x $ is nilpotent, 
the image of the moment map $ \Im \mu_{\FlB} $ is contained in $ \nilpotentsof{\lie{g}} $, the nilpotent variety of $ \lie{g} $.  
It is easy to see that in fact $ \Im \mu_{\FlB} = \nilpotentsof{\lie{g}} $.  
Thus we get $ \mu_{\FlB} : T^*\FlB \to \nilpotentsof{\lie{g}} $.  This map is known to be birational, and 
gives a {$ G $-equivariant} resolution of singularity of the nilpotent variety $ \nilpotentsof{\lie{g}} $ called the \emph{Springer resolution} 
(\cite{Springer.1968}; see \cite{Slodowy.LNM815.1980} for details).  
{Here we chose a nondegenerate invariant bilinear form on $ \lie{g} $.  The choice does not affect the structure of orbits or geometric properties of double flag varieties discussed below, although we do not explicitly state it.}

Similarly, for the partial flag variety $ \FlP = G/P $, 
the moment map is also the second projection
\begin{equation}
\mu_{\FlP}((\lie{p_1}, x)) = x \qquad 
((\lie{p_1}, x) \in T^*\FlP; \; \lie{p_1} \in \FlP, \, x \in \lie{u_1}).
\end{equation}
In this case, the image $ \Im \mu_{\FlP} = \Ad(G) \, \lie{u} $ is the closure of a certain nilpotent orbit.  
This orbit is called a \emph{Richardson orbit} associated with the parabolic subalgebra $ \lie{p} $, and often denoted by $ \calorbit_P $.  
{For a parabolic subgroup $ P $, the moment map} $ \mu_{\FlP} : T^*\FlP \to \closure{\calorbit_P} $ is not a resolution in general 
(see \cite{Spaltenstein.1976}, \cite[\S~3.5]{Borho.MacPherson.1983}, \cite[\S~8.8]{Jantzen.PM228.2004} for example).  
It is a resolution if the nilpotent orbit is even or $ G =\GL_n $.  

\medskip

Let us {return to} a general $ G $-variety $ X $ and the moment map $ \mu_X: T^*X \to \lie{g} $ above.  
We identify $ \lie{g}^* $ with $ \lie{g} $ as we mentioned.

Take a point $ x_0 \in X $ and consider the $ G $-orbit $ \Xorbit = G \cdot x_0 $.  
We define a Lagrangian subvariety $ T_{\Xorbit}^*X \subset T^*X $ as follows.
\begin{equation}
T_{\Xorbit}^*X = \{ (x, \xi) \in T^*X \mid x \in \Xorbit, \; \xi \in (T_x\Xorbit)^{\bot} \} ,
\end{equation}
{where $ (T_x\Xorbit)^{\bot} $ denotes the subspace of $ T_x^*X $ consisting of all linear forms $ \xi $ killing 
$ T_x\Xorbit $.  }
The subvariety $ T_{\Xorbit}^*X $ is called the \emph{conormal bundle} of the orbit $ \Xorbit $.  
{Note that a conormal bundle is irreducible since it is a vector bundle over a $ G $-orbit, which is irreducible because $ G $ is assumed to be connected.}

\begin{lemma}
\label{lemma:conormal}
Put $ Y = \mu_X^{-1}(0) \subset T^*X $, the null fiber of the moment map.  
Then it is a disjoint union of all conormal bundles.
\begin{equation}
Y = \coprod_{\Xorbit \in X/G} T_{\Xorbit}^*X \qquad (disjoint).
\end{equation}
If there are only finitely many $G$-orbits in $ X $, {then} 
$ Y $ is an equidimensional variety of dimension $ \dim X $, and 
its irreducible components are given by the closures of conormal bundles: 
$ \Irr(Y) = \{ \closure{T_{\Xorbit}^*X} \mid \Xorbit \in X/G \} $.
\end{lemma}

\begin{proof}
{The fact that} $ Y $ is a disjoint union of conormal bundles is straightforward from definitions.  
Since $ \dim T_{\Xorbit}^*X = \dim X $, if there appear only finitely many of them, clearly their closures give irreducible components and 
hence $ Y $ is equidimensional of dimension $ \dim X $.
\end{proof}

Since $ Y $ consists of conormal bundles, 
we call $ Y = \mu_X^{-1}(0) $ a \emph{conormal variety}.

\begin{corollary}
Assume that there are only finitely many $ G $-orbits in $ X $.  
There exists a natural bijection between irreducible components of the conormal variety $ Y $ and 
the $ G $-orbits in $ X $: 
$ X/G \simeq \Irr(Y) $, which maps $ \Xorbit \in X/G $ to the closure of its conormal bundle 
$ \closure{T_{\Xorbit}^*X} \subset Y $.
\end{corollary}

\section{Orbits in flag varieties: known results, summary}

We give basic facts on the orbits in flag varieties, which should be well-known.  
See \cite{Borel.1991,Springer.1998,Humphreys.LAG.1975} for example.

\subsection{Bruhat decomposition}

Let $ B \subset G $ be a Borel subgroup and 
fix a Levi decomposition $ B = T N $, where $ T $ is a maximal torus (in $ G $) and $ N $ the unipotent radical.  
The finite group $ W = N_G(T)/T $ is called the Weyl group of $ G $, and 
it coincides with the Weyl group of the root system of $ G $.  
We often pick representatives for each element in $ W $, and if there is no confusion we simply write $ w \in W $ for an element $ w \in N_G(T) $.  

We have a double coset decomposition 
\begin{equation}
G = \coprod_{w \in W} B w B ,  
\end{equation}
which is called the \emph{Bruhat decomposition}.
From this, we know $ B $-orbits on the complete flag variety $ \FlB = G/B $ have representatives in $ W $.  
Thus $ \FlB/B \simeq W $.

\skipover{
Moment maps for the action of $ B $ and orbital varieties.  
I skip this for now.
}

Let $ P $ be a parabolic subgroup which contains $ B $.  
In this case $ P $ is called \emph{standard}.
We get the standard Levi decomposition $ P = L U $, where $ L $ is a Levi subgroup containing $ T $ and $ U $ is the unipotent radical.  
Since $ L $ is reductive, we can consider the Weyl group $ W_L $ of $ L $, which is considered to be a subgroup of $ W = W_G $.  
We also denote $ W_L $ by $ W_P $.  

There is a generalization of the Bruhat decomposition extending to the case of partial flag varieties.

\begin{lemma}\label{lemma:Bruhat.decomposition.for.parabolics}
Let $ P $ be a parabolic subgroup.  
Recall the notation 
$ \FlB = G/B $ and $ \FlP = G/P $.  
\begin{penumerate}
\item
$ \FlB/P \simeq \FlP/B \simeq W/W_P $
\item
For another parabolic subgroup $ P' $ of $ G $, the generalized Bruhat decomposition  
$ \FlP/P' \simeq W_P \backslash W / W_{P'} $ holds.  
In particular, we get $ P \backslash G / P = \FlP/P \simeq W_P \backslash W / W_P $.
\end{penumerate}
\end{lemma}

\begin{proof}
This follows from the Bruhat decomposition.  
See Borel \cite[\S~14.16]{Borel.1991} and Springer \cite[Exercise 8.4.6(3)]{Springer.1998}.
\end{proof}

If $ X $ is an irreducible normal $ G $-variety on which a Borel subgroup $ B $ acts with finitely many $ B $-orbits, 
the variety $ X $ is called \emph{spherical}.
The following theorem is well known.

\begin{theorem}[{Brion \cite{Brion.MM.1986}, Vinberg \cite{Vinberg.1986}}]\label{thm:Brion.Vinberg}
An irreducible normal $ G $-variety 
$ X $ has finitely many $ B $-orbits if and only if 
it admits an open $ B $-orbit.  
\end{theorem}

From above considerations, flag varieties are spherical.

Spherical varieties are ubiquitous in representation theory and 
bear many nice properties.  
For this, see Brion \cite{Brion.2012}, 
Knop and others \cite{Knop.1997,Knop.VanSteirteghem.2006,MR3916090,MR3948942}, 
Perrin \cite{Perrin.2014}, Timashev \cite{Timashev.EMS.2011} and references therein.

\skipover{
Brion, M. (1995). Spherical varieties. In Proceedings of the International Congress of Mathematicians (pp. 753-760). Birkh\"{a}user, Basel.

Knop, F., & Van Steirteghem, B. (2006). Classification of smooth affine spherical varieties. Transformation groups, 11(3), 495-516.

Perrin, N. (2014). On the geometry of spherical varieties. Transformation Groups, 19(1), 171-223.

Knop, F. (1998). Some remarks on multiplicity free spaces. In Representation theories and algebraic geometry (pp. 301-317). Springer, Dordrecht.

Knop, F., Kr\"{o}tz, B., & Schlichtkrull, H. (2015). The local structure theorem for real spherical varieties. Compositio Mathematica, 151(11), 2145-2159.

Timashev??
}

\subsection{Symmetric pairs}

Let $ \theta $ be an involutive automorphism of $ G $ and 
$ K = G^{\theta} $ the subgroup of fixed points of $ \theta $.  
%%For simplicity, we assume $ \theta $ is not an identity.  
The pair $ (G, K) $ is called a \emph{symmetric pair} and 
$ K $ a \emph{symmetric subgroup}.
The quotient space $ G/K $ is called a \emph{symmetric space}.
Note that since 
the group $ K $ is closed and reductive, 
the quotient $ G/K $ is indeed an affine variety.  

If $ G $ is simply connected, then $ K $ is connected.  
However, without the assumption, it need not be connected in general.  
In the following, we always assume $ K $ is connected
\footnote{
Many of the facts and results below will hold after some modifications even if $ K $ is not connected, 
and for most of our purpose it is sufficient to take a connected component of the identity instead of 
the whole $ K $.
}.  

\begin{example}\label{ex:symmetric.pairs.AIII.CI}
For $ G = \GL_n(\C) $, put $ \theta(g) = I_{p,q} g I_{p,q} \;\; (g \in G) $, 
where $ I_{p,q} = \diag(\unitmatrix_p, - \unitmatrix_q) $ with $ \unitmatrix_m $ being the identity matrix of size $ m $.  
Then $ K = \GL_p(\C) \times \GL_q(\C) $, block-diagonally embedded into $ G $.  
The pair $ (\GL_n(\C), \GL_p(\C) \times \GL_q(\C)) $ is said to be a symmetric pair of type AIII.

Next, let us consider 
$ G = \Sp_{2n}(\C) $, the symplectic group of size $ 2 n $ (and rank $ n $) preserving the standard symplectic form.
\begin{equation}
\Sp_{2n}(\C) = \{ g \in \GL_{2n}(\C) \mid \transpose{g} J_n g = J_n \} , 
\;\;
%%\text{ where  } \;\;
J_n = \mattwo{0}{\unitmatrix_n}{-\unitmatrix_n}{0} .
\end{equation}
The involution $ \theta(g) = I_{n, n} g I_{n, n} \;\; (g \in G) $ defines 
$ K \simeq \GL_n(\C) $.  
This symmetric pair $ (\Sp_{2n}(\C), \GL_n(\C)) $ is called of type CI. 

Finally, 
let us consider $ G = \GL_n(\C) $ with the involution 
$ \theta(g) = \transpose{g}^{-1} $.  
In this case 
$ K = \OO_n(\C) $ is an orthogonal group, which is not connected but has $ 2 $ connected components.  
\end{example}

\subsection{KGB theory: $ K $-orbits in flag varieties}

The following lemma is well known, but we will give a proof which will clarify some parametrization of 
$ K $-orbits in $ G/B $.

\begin{lemma}\label{lemma:KGB.decomposition}
$ G/K $ is a $ G $-spherical variety, i.e., 
a Borel subgroup $ B $ of $ G $ has finitely many orbits in $ G/K $.  
\end{lemma}

The lemma is equivalent to the claim that 
the double coset space $ B \backslash  G/K $ is finite, which is further equivalent to 
that there are finitely many $ K $-orbits in the complete flag variety $ G/B $.

\begin{proof}
We will show there are finitely many $ K $-orbits on $ \FlB = G/B $.  
For this, we follow the strategy of Milli\'{c}i\v{c} \cite[\S~H.2, Theorem 1]{Milicic.1993} (see also 
\cite[Remark~1.8]{Richardson.Springer.1994}).  
We fix a $ \theta $-stable Borel subgroup $ B $, which always exists (see \cite[Theorem~7.2]{Steinberg.1968}).
Consider the double flag variety 
$ \FlB^2 = G/B \times G/B $, on which $ G $ acts diagonally.    
We define another $ G $ action on $ \FlB^2 $, which is called the $ \theta $ twisted $ G $ action, by 
\begin{equation*}
\delta_{\theta}(g) (B_1, B_2) := (g B_1, \theta(g) B_2) 
\qquad
g \in G, \; (B_1, B_2) \in \FlB^2. 
\end{equation*}
We denote the double flag variety with this twisted action by $ \FlB^2_{\theta} $.   
As $ G $-varieties, $ \FlB^2 $ and $ \FlB^2_{\theta} $ are isomorphic via 
\begin{equation*}
\Phi : \FlB^2 \to \FlB^2_{\theta} , \qquad
\Phi(B_1, B_2) = (B_1, \theta(B_2)) .
\end{equation*}
Thus $ \FlB^2_{\theta}/\delta_{\theta}(G) \simeq W $ by the Bruhat decomposition.  
Let us denote the twisted $ G $-orbit corresponding to $ w \in W $ by $ \calorbit_w^{\theta} $.  
Let $ X = \Delta(\FlB) = \{ (B_1, B_1) \mid B_1 \in \FlB \} \subset \FlB^2_{\theta} $ be 
the (non-twisted) diagonal embedding.

We will show that any irreducible component of $ X \cap \calorbit_w^{\theta} $ is precisely a single $ K $-orbit, 
which implies there are finitely many orbits in $ X \cap \calorbit_w^{\theta} $.  
Since $ \{ \calorbit_w^{\theta} \}_{w \in W} $ covers $ \FlB^2_{\theta} $ as well as $ X $, we conclude 
$ X \simeq \FlB $ has finitely many $ K $-orbits (note that the twisted action and the natural diagonal action of $ K $ {coincide}).

Take $ \xi \in X \cap \calorbit_w^{\theta} $ and consider its tangent space:
\begin{equation}\label{eq:Txi.X.cap.Owtheta}
T_{\xi}(X \cap \calorbit_w^{\theta}) \subset T_{\xi}X \cap T_{\xi}\calorbit_w^{\theta}.
\end{equation}
Since $ \xi \in X $, it can be written as 
$ \xi = (x B, x B) \in X $ for some $ x \in G $.  
Then the right hand side of \eqref{eq:Txi.X.cap.Owtheta} is
\begin{equation}\label{eq:tangent.of.intersection}
\{ (y + \lie{b}^x, y + \lie{b}^x) \mid y \in \lie{g} \} 
\cap 
\{ (y + \lie{b}^x, \theta(y) + \lie{b}^x) \mid y \in \lie{g} \} \subset \lie{g}/\lie{b}^x \oplus \lie{g}/\lie{b}^x.  
\end{equation}
From this, considering $ (y + \lie{b}^x, y + \lie{b}^x) $ in the left hand side of \eqref{eq:tangent.of.intersection}, 
we conclude that $ y - \theta(y) \in \lie{b}^x $ which implies 
\begin{equation*}
y = \frac{1}{2}(y + \theta(y)) + \frac{1}{2} ( y - \theta(y)) \in \lie{k} + \lie{b}^x.
\end{equation*}
This shows in fact we can take $ y \in \lie{k} $ modulo $ \lie{b}^x $.  
We denote the $ K $-orbit through $ \xi $ by $ \bborbit_K $.  
The above arguments show $ T_{\xi}X \cap T_{\xi}\calorbit_w^{\theta} \subset T_{\xi}\bborbit_K $.  
On the other hand, since $ \bborbit_K \subset X \cap \calorbit_w^{\theta} $, it follows 
$ T_{\xi}\bborbit_K \subset T_{\xi}(X \cap \calorbit_w^{\theta}) $.  
Together with Equation~\eqref{eq:Txi.X.cap.Owtheta}, we get 
\begin{equation}
T_{\xi}\bborbit_K \subset T_{\xi}(X \cap \calorbit_w^{\theta}) \subset T_{\xi}X \cap T_{\xi}\calorbit_w^{\theta} \subset T_{\xi}\bborbit_K .
\end{equation}
Thus all the terms in the above containments are equal in fact.  

Now consider an irreducible component $ Z $ of $ X \cap \calorbit_w^{\theta} $.  
We choose a point $ \xi $ from the set of smooth points in $ Z $.  
Then $ T_{\xi}Z = T_{\xi}\bborbit_K $ holds by the above consideration, which means 
$ \bborbit_K $ is open dense in $ Z $ and we get $ Z = \closure{\bborbit_K} $.  

If $ \xi_0 \in Z \setminus \bborbit_K $, 
the general theory of algebraic varieties tells 
$ \dim T_{\xi_0}Z \geq \dim Z = \dim \bborbit_K $.  
While $ \xi_0 \in \closure{\bborbit_K} \setminus \bborbit_K $, and in that case, 
$ \dim K \cdot \xi_0 < \dim \bborbit_K $ holds.  Thus we get 
\begin{equation*}
\dim T_{\xi_0}Z \geq \dim Z = \dim \bborbit_K > \dim K \cdot \xi_0 = \dim T_{\xi_0}(X \cap \calorbit_w^{\theta})
\geq \dim T_{\xi_0}Z, 
\end{equation*}
contradiction.  Thus we conclude exactly $ Z = \bborbit_K $ holds.

So irreducible components of $ X \cap \calorbit_w^{\theta} $ are $ K $-orbits and coincide with connected components.  They are finite in number.
\end{proof}

\begin{remark}
By Theorem~\ref{thm:Brion.Vinberg}, the above lemma is equivalent to 
the claim that there is an open dense $ B $-orbit in $ G/K $.  
It is not hard to prove the existence of an open dense $ B $-orbit.  
In fact, it suffices to show $ \lie{g} = \lie{k} + \lie{b} $ for a certain Borel subalgebra.  

It is known that there is a $ \theta $-split Borel subalgebra $ \lie{b} $ (see \cite{Vust.1974}).  
Here, the term ``$ \theta $-split'' means $ \theta(\lie{b}) $ is the opposite Borel subalgebra of $ \lie{b} $, 
so that we have $ \lie{g} = \lie{b} + \theta(\lie{b}) $.  
Take an arbitrary $ x \in \lie{g} $ and write it as 
$ x = x_1 + x_2 $ for $ x_1 \in \lie{b} $ and $ x_2 \in \theta(\lie{b}) $.  
Then we have 
$ x_2 = \bigl( x_2 + \theta(x_2) \bigr) - \theta(x_2) $ and 
clearly $ x_2 + \theta(x_2) \in \lie{k}, \; \theta(x_2) \in \lie{b} $.  
This proves that $ x \in \lie{k} + \lie{b} $, and hence proves $ \lie{g} = \lie{k} + \lie{b} $ for a $ \theta $-split $ \lie{b} $.
\skipover{
If you are familiar with the structure theory of real semisimple Lie group, 
this follows immediately.  

In fact, for a suitable real form $ G_{\R} $ which is $ \theta $-stable, 
$ K $ is the complexification of a maximal compact subgroup $ K_{\R} $ of $ G_{\R} $.  
Consider Iwasawa decomposition $ G_{\R} = K_{\R} A_{\R} N^+ $.  
By complexification, we get $ \lie{g} = \lie{k} + \lie{a} + \lie{n^+} $.  
Put $ \lie{m} = \lie{z}_{\lie{k}}(\lie{a}) $, the centralizer of $ \lie{a} $ in $ \lie{k} $, which is reductive.  
Then $ \lie{m} + \lie{a} + \lie{n^+} $ is a parabolic subalgebra, which is the complexification of the Lie algebra of a minimal parabolic subgroup for $ G_{\R} $.  
Choose a Borel subalgebra $ \lie{b}_{\lie{m}} $ of $ \lie{m} $, then 
$ \lie{b} = \lie{b}_{\lie{m}} + \lie{a} + \lie{n^+} $ is a Borel subalgebra of $ \lie{g} $ and it is easy to see 
$ \lie{g} = \lie{k} + \lie{b} $ for this $ \lie{b} $.  
This proves existence of a dense orbit of $ B $ in $ G/K $.
}
\end{remark}

\begin{corollary}
There exists a natural bijection 
\begin{equation*}
\FlB/K = K \backslash G/B \simeq \coprod_{w \in \mathscr{I}} \Irr (X \cap \calorbit_w^{\theta}) , 
\end{equation*}
%
%%$ K $-orbits in the flag variety $ \FlB = G/B $ is in bijection with 
%%$ \coprod_{w \in \mathscr{I}} \Irr (X \cap \calorbit_w^{\theta}) $, 
where $ \mathscr{I} = \{ w \in W \mid \theta(w) = w^{-1} \} $ is the set of $\theta$-twisted involutions and 
$ X $ and $ \calorbit_w^{\theta} $ are given in the proof of Lemma~\ref{lemma:KGB.decomposition}.
\end{corollary}

\begin{proof}
Except for the condition $ w \in \mathscr{I} $, the claim has been already established in the proof of the above lemma.  

If $ X \cap \calorbit_w^{\theta} $ is not empty, 
then there exists $ g \in G $ which satisfies 
$ (g B, g B) = (g B, \theta(g) w B) $.  
By \cite[Theorem~4.2]{Springer.1985}, 
a representative of the double coset $ K v B \; (v \in G) $ can be chosen so that 
$ v \theta(v)^{-1} \in N_G(T) $, where $ T $ is a $ \theta $-stable maximal torus in $ B $, which is also assumed to be $ \theta $-stable. 
Thus, if we translate $ g B $ by $ K $, we can assume 
$ (v B, v B) = (v B, \theta(v) w B) $ for some $ v $ with the property $ v \theta(v)^{-1} \in N_G(T) $.  
Since $ v B = \theta(v) w B $, we get 
$ v^{-1} \theta(v) w \in B $.  
This means $ v^{-1} \theta(v) w \in T $ and 
$ w = v \theta(v)^{-1} $ in $ W = N_G(T)/T $.  
Thus we see 
$ \theta(w) = \theta(v \theta(v)^{-1}) = \theta(v) v^{-1} = w^{-1} $, 
which concludes that $ w \in \mathscr{I} $.
\end{proof}

From the corollary, we get a natural projection 
\begin{equation}
\FlB/K = K \backslash G/B \to \mathscr{I},
\end{equation}
which is essentially the same as Richardson-Springer (\cite[\S~1.6]{Richardson.Springer.1994}, \cite[Proposition~1.7.1]{Richardson.Springer.1993}).  

There is another way to parameterize $ K \backslash G / B $ by using nilpotent orbits and their Springer fibers.  
For this, see \cite{CNT.2009,Trapa.IMRN.1999} and references therein.

\subsection{Reductive subgroups acting on flag varieties}

{There is a classification, due to Avdeev and Petukhov \cite{Avdeev.Petukhov.2021}, of the reductive subgroups $ H $ in $ G $ which act spherically on any given partial flag variety $ G/P $.} 
In particular these subgroups $ H $ have finitely many orbits in $ G/P $.  
There are many examples of such subgroups other than symmetric subgroups, which we already consider (see Lemma~\ref{lemma:KGB.decomposition} above).

\skipover{
Classification of spherical reductive subgroups acting on partial flag varieties:
Avdeev, Roman ;  Petukhov, Alexey . Spherical actions on isotropic flag varieties and related branching rules.
 Transform. Groups  26  (2021),  no. 3, 719--774.

Description of lattice of HW of Spherical subgroups acting on partial flag varieties (symmetric subgroup and GL_n):
MR4109129  Avdeev, Roman ;  Petukhov, Alexey . Branching rules related to spherical actions on flag varieties.
 Algebr. Represent. Theory  23  (2020),  no. 3, 541--581.

Classification of spherical reductive subgroups acting on partial flag varieties of type A:
MR3288423  Avdeev, R. S. ;  Petukhov, A. V.  Spherical actions on flag varieties.
(Russian) ; translated from  Mat. Sb.  205  (2014),  no. 9, 3--48 Sb. Math.  205  (2014),  no. 9-10, 1223--1263
}

\section{Product of flag varieties}\label{section:mult.FV}

For applications of the theory of moment maps introduced in \S\ref{sec:flag.varieties.and.moment.maps}, 
we need a $ G $-variety $ X $ with finitely many orbits.  
For example, 
%%We are interested in the case where $ X $ is a projective variety.  
if $ X $ is a product of partial flag varieties $ G/P_1 $ and $ G/P_2 $, where $ P_i \; (i = 1, 2) $ are parabolic subgroups, 
then the diagonal action of $ G $ on $ X = G/P_1 \times G/P_2 $ has finitely many orbits.  
To see this, we note
\begin{equation*}
X/G = (G/P_1 \times G/P_2)/G \simeq P_1 \backslash G / P_2 \simeq W_{P_1} \backslash W_G / W_{P_2} ,
\end{equation*}
where the last isomorphism is given by the Bruhat decomposition (see Lemma~\ref{lemma:Bruhat.decomposition.for.parabolics}).  
In this way, we also get a parametrization of $ G $-orbits in $ X $.  

It is natural to ask the problem of when a product of flag varieties has finitely many orbits under the diagonal action of $ G $.

\subsection{Triple flag varieties}\label{subsec:3FV}

\begin{theorem}\label{thm:3FV.of.finite.type}
Let $ G $ be a classical group of type $ A, B, C, D $.  
and $ P_i \; (1 \leq i \leq k) $  parabolic subgroups of $ G $, {none of which is equal to $ G $}.  
If the product $ G/P_1 \times G/P_2 \times \cdots \times G/P_k $ admits finitely many $ G $-orbits under the diagonal $ G $-action then 
$ k \leq 3 $ holds.  
\end{theorem}

The case $ k = 1 $ is trivial, and 
as mentioned above, the case $ k = 2 $ reduces to the Bruhat decomposition.    
So, in this case, the number of orbits is always finite.  

Let us assume $ k = 3 $, and consider a triple flag variety $ G/P_1 \times G/P_2 \times G/P_3 $.  
In general this variety has infinitely many $ G $-orbits, but there are many interesting cases where 
they admit a finite number of orbits.  
For classical groups, those cases are classified by Magyar-Weyman-Zelevinsky \cite{MWZ.1999,MWZ.2000} and Matsuki \cite{Matsuki.2015,Matsuki.arXiv2019}, 
which gives Theorem~\ref{thm:3FV.of.finite.type}.  
In \cite{MWZ.1999,MWZ.2000}, they also classified orbits in terms of quiver representations (see \S~\ref{subsection:JV} below for type A).  
For odd orthogonal groups, Matsuki \cite{Matsuki.2013} classified orbits in the case of (I) in \cite[Theorem~1.6]{Matsuki.2015}.  

{Up to now, there is no simple way to prove the number $ k $ should be smaller than $ 4 $ for multiple flag varieties of finite type for classical groups, to the knowledge of authors.}

The classification of multiple flag varieties of finite type for exceptional groups remains open, except for the cases of type $ G_2 $ and $ F_4 $ (see \cite{BSEKG.2017.arXiv}).  
It is known the number $ k $ of the flag varieties of finite type is less than or equal to $ 4 $ for type $ E_6, E_7, E_8 $, that it is $ 3 $ for type $ F_4 $, 
and that there is no multiple flag varieties {of finite type} for type $ G_2 $ if $k$ is greater than $ 2 $ (\cite[Theorem 2]{Popov.2007}).  
Actually the result of Popov \cite[Theorem 2]{Popov.2007} specifies the upper bound of the number of 
flag varieties for which the multiple flag varieties admit an \emph{open} orbit.  
We reproduce a table in \cite{Popov.2007} here for the convenience of readers.  
For the number $ b_G $ specified in Table~\ref{Table:Popov2007}, 
if a multiple flag variety $ G/P_1 \times G/P_2 \times \cdots \times G/P_d $ 
has an open $ G $-orbit then $ d \leq b_G $ holds. 
\begin{table}[H]
\caption{Number $ b_G $}\label{Table:Popov2007}
%%\begin{equation*}
%%\hfil{Table~\ref{Table:Popov2007}: Number $ b_G $}\hfil
%%\\[1ex]
{
\hfil
$
\begin{array}{c||c|c|c|c|c|c|c|c|ccccccccccccccccc}
%%\text{Type of $ G $} 
G & A_n \; (n \geq 1 ) & B_n \; (n \geq 1) & C_n \; (n \geq 2) & D_n \; (n \geq 4) & E_6 & E_7 & E_8 & F_4 & G_2 \\ 
\hline
b_G & n + 2 & n + 1 & n + 1 & n & 4 & 4 & 4 & 3 & 2
\end{array}
$
\hfil
}
%%\\[1ex]
%%\end{equation*}
%
\end{table}
\indent
{For type $A_n$, it is widely known that there exists an open orbit in an  $ (n + 2) $-tuple projective spaces.  
Beyond $ (n + 2) $ there is no open orbit, which means $  b_G = n + 2 $ is best possible for type A 
(cf. Shimamoto \cite{Shimamoto.2020}, which describes families of generic orbits in the product of projective spaces).  
For other types, these bound seems over estimated.  For this, see Theorem 3 in \cite{Popov.2007}.}
%%
%%Coskun, Izzet and Hadian, Majid and Zakharov, Dmitry: open orbits in the product of Grassmannians

Even if there is an open orbit, it does not imply the finiteness of orbits.  
However, if one of the parabolic subgroups is a Borel subgroup, 
the existence of an open orbit tells that the multiple flag variety is of finite type.  
Stembridge \cite{Stembridge.2003} gets a complete classification of 
the triple flag varieties $ G/P_1 \times G/P_2 \times G/B $ with $ B $ being a Borel subgroup of $ G $, 
including exceptional cases.  
For a summary, we refer the readers to \cite[Table 1]{HNOO.2013}.  
In this case, there is no such triple flag variety for type $ E_8, F_4, G_2 $.  
{If $ G = E_6, E_7 $, there is no quadruple flag variety of finite type with one of the member $ G/B $ because of the dimension reason, i.e., the dimension of such quadruple flag variety is strictly greater than $ \dim G $.}

\subsection{Double flag varieties for symmetric pairs}\label{subsec:DFl.symmetric.pair}

Let us interpret the triple flag varieties in a different way {as} explained below.  
Let $ \bbG = G \times G $ and put $ \bbP = P_2 \times P_3 $, a parabolic subgroup of $ \bbG $.
Choose an involution $ \Theta $ on $ \bbG $ as a flip of pairs, namely we put 
$ \Theta(g_1, g_2) = (g_2, g_1) $ for $ (g_1, g_2) \in \bbG $.  
For this involution, the fixed point subgroup is the diagonal embedding of $ G $ into $ G \times G $, 
which we denote by $ \bbK \simeq G $, the symmetric subgroup associated with $ \Theta $.  
Hence we can (and do) identify $ P_1 $ with a parabolic subgroup of $ \bbK $, and we write $ \bbQ = P_1 $.  
Thus we see
\begin{equation}\label{eq:DFl.symmetric.pair}
\bbK/\bbQ \times \bbG/\bbP \simeq G/P_1 \times G/P_2 \times G/P_3
\end{equation}
is the triple flag variety on which $ \bbK \simeq G $ acts.  
Note that $ \bbK/\bbQ $ is a flag variety for the symmetric subgroup $ \bbK $ and 
$ \bbG/\bbP $ is a flag variety for $ \bbG $.

Following this example, let us define a double flag variety for a symmetric pair.  
Thus, now we consider $ G $ with an involution $ \theta $, 
and denote the corresponding symmetric subgroup by $ K = G^{\theta} $.  
We pick a parabolic subgroup $ Q \subset K $ of $ K $ and 
a parabolic subgroup $ P \subset G $ of $ G $.  
We call the following variety a \emph{double flag variety for the symmetric pair} $ (G, K) $, 
on which $ K $ acts diagonally,
\begin{equation}
\dblFV = K/Q \times G/P.
\end{equation}
The above discussion shows that triple flag varieties are special cases of double flag varieties.

Our main interest in this paper is the theory of double flag varieties of symmetric pairs.  

\begin{example}\label{ex:DFV.AIII}
We keep the notation in Example~\ref{ex:symmetric.pairs.AIII.CI}.  
Let us consider a symmetric pair of type AIII: 
\begin{equation*}
(G, K)  = (\GL_n(\C),  \GL_p(\C) \times \GL_q(\C)) \qquad (n = p + q) ,    
\end{equation*}
where $ K $ is embedded into $ G $ block diagonally.  
A parabolic subgroup $ P $ of $ G $ is determined by a composition 
$ \mbfa = (a_1, \dots, a_r) $ of $ n $.  
We denote the parabolic subgroup determined by this composition $ \mbfa $ by $ P_{\mbfa} $.  
If we pick a standard basis $ \{ \eb_i \mid 1 \leq i \leq n \} $ and 
define a \emph{standard flag} $ \Flag(\mbfa) $ associated with $ \mbfa $ by  
\begin{equation}\label{eq:flag.by.composition}
\begin{aligned}
\Flag(\mbfa): V_0 = \{ 0 \} \subset & V_1 \subset V_2 \subset \cdots \subset V_r = \C^n , 
\qquad
\\
&
V_i = \langle \eb_k \mid 1 \leq k \leq a_1 + a_2 + \cdots + a_i \rangle .
\end{aligned}
\end{equation}
The parabolic subgroup $ P_{\mbfa} $ is defined to be the stabilizer of the flag $ \Flag(\mbfa) $, 
which has diagonal blocks of size $ a_1, a_2, \dots, a_r $ and blockwise upper triangular.

Similarly, take compositions $ \mbfb = (b_1, \dots, b_{l_1}) $ of $ p $ and  
$ \mbfc = (c_1, \dots, c_{l_2}) $ of $ q $.  
Then $ Q = P_{\mbfb} \times P_{\mbfc} $ up to $ K $-conjugacy, and 
$ K/Q = \GL_p/P_{\mbfb} \times \GL_q/P_{\mbfc} $ is a product of flag varieties of type A.  
Thus we get a double flag variety 
\begin{equation}
\dblFV = K/Q \times G/P 
= \GL_p/P_{\mbfb} \times \GL_q/P_{\mbfc} \times \GL_n/P_{\mbfa} 
\end{equation}
on which $ K = \GL_p \times \GL_q $ acts.
\end{example}

\begin{example}\label{ex:DFV.CI}
We continue examining the examples given in Example~\ref{ex:symmetric.pairs.AIII.CI}, 
and consider a symmetric pair of type CI: 
$ (G, K) = (\Sp_{2n}(\C), \GL_n(\C)) $.  

Let $ V $ be a symplectic vector space of dimension $ 2n $ with a symplectic form $ \langle \;,\; \rangle $, and 
$ V = V^+ \oplus V^- $ a polar decomposition of $ V $ by Lagrangian subspaces $ V^{\pm} $.  
For a composition $ \mbfa $ of $ n $ and 
$ \Flag(\mbfa) = (V_i^+)_{i = 1}^n $ the corresponding flag of isotropic subspaces inside $ V^+ $  
(cf.~Equation~\eqref{eq:flag.by.composition}),
we define the subgroup of elements in $ G $ which stabilize $ \Flag(\mbfa) $ and denote it by $ P_{\mbfa} $.  
Put 
\begin{equation}
V_i^- = (V_{n - i}^+)^{\bot} 
= \{ v \in V^- \mid \langle v, V_i^+ \rangle = 0 \} .
\end{equation}
It is the dual flag of $ \Flag(\mbfa) $ inside $ V^- $.  
An element $ h \in P_{\mbfa} $ preserves both the flag $ \Flag(\mbfa) $ and its dual.
\end{example}

\skipover{
There are several ways to identify $ K $-orbits on the double flag varieties $ \dblFV $.  
Here we will explain a classification based on \cite[Theorem 2.7]{HNOO.2013}, which is useful to identify finiteness of the number of orbits.  
Another way of description is that in terms of quiver representations, which we will explain in detail in \S~\ref{quiver.rep}.  Combinatorial description is also useful.  Steinberg maps and their Springer fibers....
}

\part{\textbf{Double flag varieties of finite type}}
%%\section*{\textbf{Part II. Double flag varieties of finite type}}
\partlabel{part:finite.type.DFV}
%%\addcontentsline{toc}{section}{\textbf{Part II. Double flag varieties of finite type}\partstrut}

\section{Finiteness criterions}\label{sec:finiteness.criterions}

We will apply the theory explained in \S~\ref{subsec:moment.map.for.FlB}
to the double flag varieties of symmetric pairs.  
For this purpose, it is important to know if there are finitely many orbits on a double flag variety.  
Thus we say, a double flag variety $ \dblFV = K/Q \times G/P $ is \emph{of finite type} if there are only finitely many $ K $-orbits on $ \dblFV $.  
Although the classification of the double flag varieties of finite type is still largely open, 
there are several different classifications of such varieties.  

Roughly speaking, there are three major 
methods to classify the $ K $-orbits on $ \dblFV $.  
Let us explain them briefly.  
\begin{penumerate}
\item
Since $ \dblFV /K \simeq P \backslash G / Q $, 
the double coset decomposition by $ (P,Q) $ will give the classification.  
This is the strategy used in \cite{HNOO.2013}, but explicit decomposition is often very subtle.  
In \cite{HNOO.2013}, the decomposition is reduced to that of unipotent subgroups, and 
by linearization, it {is} further reduced to the multiplicity free representations if one of $ P $ or $ Q $ is a Borel subgroup.

\item
A double flag variety can be considered as configuration spaces of flags (or vector spaces).  
In this viewpoint, the theory of quiver representations is useful, especially for type A groups.  
This idea first appeared in \cite{MWZ.1999,MWZ.2000} for triple flag varieties, 
 and then applied to the double flag variety by Homma in \cite{Homma.2021}.  
We will explain this in detail later in \S~\ref{sec:quiver.joint.variety}.

\item
Let $ \nilpotentsof{\lie{k}} $ be the nilpotent variety of $ \lie{k} $.  Then there is a well defined map between the orbit spaces 
$ \Phi_\fk : \dblFV / K \to \nilpotentsof{\lie{k}} / K $ (see \S~\ref{sec:Definition.Steinberg.maps} for the precise definition).  
We call this map a \emph{symmetrized Steinberg map}.  
Since the space of nilpotent orbits $ \nilpotentsof{\lie{k}} / K $ is finite and well-understood, 
if we can determine the fiber of the map $ \Phi_\fk $, it will produce the classification of $ \dblFV / K $.  
The fibers often can be described by  some combinatorial objects related to Weyl group representations  
(like Robinson-Schensted correspondence), though there is no general theory yet.  
The origin of this method is the classical Steinberg theory, and it is generalized to 
double flag varieties of type AIII in \cite{Fresse.N.2016,Fresse.N.2020,Fresse.N.2021,Fresse.N.2023}; see also \S~\ref{section:generalization.RS} below. 
\end{penumerate}

\section{Double flag varieties associated to Borel subgroups}\label{section:finite.DFV.P.or.Q=Borel}

Let us assume that either 
\begin{penumerate}
\item
$ G $ is a \emph{simple} algebraic group, which is simply connected, and 
$ K = G^{\theta} $ is the fixed point subgroup of an involution $ \theta $ as above; 
or 
\item
$ G = G_1 \times G_1 $ with $ K = \Delta G_1 $, which is the fixed point subgroup of the ``flip'' (see \S~\ref{subsec:DFl.symmetric.pair} around Equation~\eqref{eq:DFl.symmetric.pair}).  
Here $ G_1 $ is a simple, simply connected algebraic group.
\end{penumerate}
Note that in both cases $ K $ is connected.  

For a double flag variety $ \dblFV = K/Q \times G/P $, 
if one of $ P $ or $ Q $ is a Borel subgroup of $ G $ or $ K $ respectively, 
then there is a complete classification of $ \dblFV $ of finite type in this setting 
due to He-N-Ochiai-Oshima \cite{HNOO.2013}.  
%%Let us introduce the results and idea of the proof.

%%\input{ideas_proofs.tex}

\subsection{Ideas of the proofs}\label{subsec:ideas.proofs.Q=BK.or.P=BG}

Let us explain the ideas of the proofs in accordance with \cite{HNOO.2013}.  

We begin with the case where $ Q = B_K $ is a Borel subgroup.  

Take a $ \theta $-stable Borel subgroup $ B $ of $ G $ and 
$ \theta $-stable maximal torus $ T \subset B $, which always exist.  
Put $ B_K = B^{\theta} = B \cap K $.  Since $ B $ is $ \theta $-stable, $ B_K $ is a Borel subgroup of $ K $.  

Let $ P $ be a parabolic subgroup of $ G $ which contains $ B $.  
Choose a Levi decomposition $ P = L U $, where $ U $ is the unipotent radical and $ T \subset L $, i.e., 
the decomposition is standard.    

\begin{lemma}\label{lemma:LcapK}
$ L \cap K $ is a connected reductive subgroup of $ K $ and 
$ L \cap B_K $ is a Borel subgroup of $ L \cap K $.  
\end{lemma}

\begin{proof}
Note that $ P' := P \cap \theta(P) $ is a minimal $ \theta $-stable parabolic subgroup 
among those that contain $ P \cap K $.  It is easy to see $ L' := L \cap \theta(L) $ is a Levi part of $ P' $.  
Since $ P' $ is $ \theta $-stable, 
$ Q' = P' \cap K $ is a parabolic subgroup of $ K $ whose Levi component is $ L' \cap K $.  
Now $ L \cap K = L \cap \theta(L) \cap K = L' \cap K $ is reductive and connected, since it is a Levi component of a parabolic subgroup $ Q' $.  
Since $ L' \cap B $ is a $ \theta $-stable Borel subgroup of $ L' $, 
$ L \cap B_K = L \cap \theta(L) \cap B_K = L' \cap B_K = (L' \cap B)^{\theta} $ is a Borel subgroup of $ L \cap K $.
\end{proof}

Let $ \lie{u} $ be the Lie algebra of the unipotent radical $ U $ of $ P $.  
We denote
\begin{equation}
\lie{s} = \lie{g}^{- \theta} = \{ x \in \lie{g} \mid \theta(x) = - x \} ,
\end{equation}
so that $ \lie{g} = \lie{k} + \lie{s} $ is the Cartan decomposition with respect to $ \theta $.  

The following theorem is a key theorem for the classification.  

\begin{theorem}[{He-N-Ochiai-Oshima \cite[Theorem~4.2]{HNOO.2013}}]\label{thm:Q=BK.DFV.finite.type.MFA}
Let $ B_K $ be a Borel subgroup of $ K $ and $ P = L U $ a parabolic subgroup of $ G $ as explained above.  
The double flag variety $ \dblFV = K/B_K \times G/P $ is of finite type if and only if 
the adjoint action of $ L \cap K $ on $ \lie{u} \cap \lie{s} $ is spherical. 
\end{theorem}

To prove this theorem, the following theorem due to Panyushev is essential. 

\begin{theorem}[{Panyushev \cite[Theorem~2.1]{Panyushev.1999.MMath}}]\label{thm:Panyushev99}
Let $ G $ be a connected reductive algebraic group and $ X $ a smooth $ G $-variety. 
For a smooth locally closed $ G $-stable subvariety $ M $ of $ X $, 
the following {\upshape(1)--(3)} are equivalent.
\begin{penumerate}
\item
The $ G $-variety $ X $ is spherical.
\item
The normal bundle $ T_M X $ is $ G $-spherical.
\item
The conormal bundle $ T_M^* X $ is $ G $-spherical.  
\end{penumerate}
\end{theorem}

\begin{proof}[Proof of Theorem~\ref{thm:Q=BK.DFV.finite.type.MFA}]
Let us assume $ \lie{u} \cap \lie{s} $ is spherical.  
We will prove that $ X := G/P $ is a spherical $ K $-variety so that it contains finitely many $ B_K $-orbits.  

Let $ \orbit $ be a $ K $-orbit through the base point $ p = e P $.  
By Panyushev's theorem, it suffices to prove that the conormal bundle $ T_{\orbit}^* X $ is $ K $-spherical.  
For that, let us see existence of an open $ B_K^- $-orbit, where $ B_K^- $ is the opposite Borel subgroup of $ B_K $.  
Since $ P $ contains $ B_K $ by assumption, the tangent space of $ \orbit $ at $ p $ is 
$ (\lie{k} + \lie{p})/\lie{p} = (\Lie(B_K^-) + \lie{p})/\lie{p} $.  Thus we conclude $ B_K^- \cdot p \subset \orbit $ is open dense.
The conormal direction of $ \orbit $ at $ p $ is $ (\lie{k} + \lie{p})^{\bot} = \lie{s} \cap \lie{p}^{\bot} = \lie{s} \cap \lie{u} $, 
which is a $ L \cap K $-variety.  Note that $ L \cap K $ is reductive by Lemma~\ref{lemma:LcapK}.
The stabilizer $ \Stab_{B_K^-}(p) = (P \cap K) \cap B_K^- \supset L \cap B_K^- $, which is a Borel subgroup of $ L \cap K $ by the same lemma.
Since $ \lie{u} \cap \lie{s} $ is spherical, there is a dense open $ L \cap B_K^- $-orbit, and we get an open $ B_K^- $-orbit on $ T_{\orbit}^* X $.

The converse implication is proved by following the above arguments in reversed order.
\end{proof}

By Theorem~\ref{thm:Q=BK.DFV.finite.type.MFA}, the classification of the double flag varieties of finite type is reduced to the classification of  spherical linear actions (they are often called multiplicity free actions, too).  This is already done by Kac \cite{Kac.1980}, and independently by Benson-Ratcliff \cite{Benson.Ratcliff.1996} and Leahy \cite{Leahy.1998}.  

\medskip

Next, we consider the case where $ P = B $ is a Borel subgroup of $ G $.  
For this case, we only quote a key theorem in \cite{HNOO.2013}.  

\newcommand{\PLmin}{P_{\mathrm{min}}^{L'}}
To do so, we choose a $ \theta $-stable parabolic subgroup $ P' $ of $ G $ which cuts out $ Q $, 
i.e., $ Q = P' \cap K $, which is always possible.  
By taking conjugation, 
we can also assume $ B $ is $ \theta $-stable and $ P' $ contains $ B $.  
Let $ P' = L' U' $ be the standard Levi decomposition with respect to a $ \theta $-stable maximal torus $ T \subset B $, 
and $ \PLmin $ be a minimal theta split parabolic subgroup of $ L' $.  
Here, ``theta split'' means $ \theta(\PLmin) $ becomes the opposite parabolic subgroup of $ \PLmin $.  
Put $ M' = \PLmin \cap K $ and 
denote by $ M'_0 $ its connected component.  

\begin{theorem}[{He-N-Ochiai-Oshima \cite[Theorem~4.4]{HNOO.2013}}]
The double flag variety $ \dblFV = K/Q \times G/B $ is of finite type if and only if 
the adjoint action of $ M'_0 $ on $ \lie{u'} \cap \lie{s} $ is spherical.  
\end{theorem}

Thus, again, the classification reduces to that of spherical linear actions.

\smallskip

Let us summarize the classifications of double flag varieties of finite types 
where one of the parabolic subgroups is a Borel subgroup 
into several tables in \S\S~\ref{subsection:Table.P=Borel}--\ref{subsec:Table.3FV.general} for the convenience and later use.  
These tables are all quoted from \cite{HNOO.2013}, \cite{Stembridge.2003}, \cite{MWZ.1999,MWZ.2000} and \cite{Matsuki.2015,Matsuki.arXiv2019}, 
with modifications to the present paper.

\subsection{Classification when $ P = B_G $ is a Borel subgroup of $ G $}\label{subsection:Table.P=Borel}

%%table: finite type where P = B_G Borel of G
{In the following tables, we separate the conditions by semicolons.  
If there is a semicolon, then the conditions should be considered independently from the former one(s). }

{\small
\begin{longtable}{c|c|c}
\caption{Double flag {varieties} 
$ K/Q \times G/B_G $ of finite type}
%%$G$-spherical $G/Q$
\label{table:g.spherical}
%\vphantom{$\Biggm|$}
%
\\
\hline
%\rule[0pt]{0pt}{12pt}
$\lie{g}$ & $\lie{k}$ & $\Pi_K \setminus J_K\ (Q=Q_{J_K})$
\\
\hline\hline
%\rule[0pt]{0pt}{12pt}
%\rule[-5pt]{0pt}{0pt}
$\lie{sl}_{2n}$ & $\lie{sp}_{2n}$ & 
\begin{xy}
\ar@{-} (0,0) *++!D{\alpha_1} *{\circ}="A"; (10,0) *++!D{\alpha_2} 
 *{\circ}="B"
\ar@{-} "B"; (20,0)="C" 
\ar@{.} "C"; (30,0)="D" 
\ar@{-} "D"; (40,0) *++!D{\alpha_{n-1}} *{\circ}="E"
\ar@{<=} "E"; (50,0) *++!D{\alpha_n} *{\circ}="F"
\end{xy}
\\
%\rule[0pt]{0pt}{12pt}
 \small$n\geq 2$ &   & 
%\\
%\rule[0pt]{0pt}{12pt}
%%  &  & 
\begin{tabular}{c}
$\{\alpha_1\}$; $\{\alpha_3\}$ if $n=3 \smallvstrut$; 
\\
%\rule[0pt]{0pt}{12pt}
%\rule[-5pt]{0pt}{0pt}
%%  &  & 
any subset of $\Pi_K$ if $n=2 \smallvstrut$ \\
\end{tabular}
\\
\hline
%\rule[0pt]{0pt}{12pt}
$\lie{sl}_{p+q}$ & $\lie{sl}_{p} \oplus \lie{sl}_{q} \oplus \C$ &
\begin{xy}
\ar@{-} (0,0) *++!D{\alpha_1} *{\circ}="A"; (10,0) *++!D{\alpha_2} 
 *{\circ}="B"
\ar@{-} "B"; (20,0)="C" 
\ar@{.} "C"; (30,0)="D" 
\ar@{-} "D"; (40,0) *++!D{\alpha_{p-1}} *{\circ}="E"
\end{xy}
\\
%\rule[0pt]{0pt}{12pt}
%\rule[-5pt]{0pt}{0pt}
\small $p+q\geq 3$ & \small $1\leq p\leq q$ &  
\begin{xy}
\ar@{-} (0,10) *++!D{\beta_{1}} *{\circ}="A"; (10,10) *++!D{\beta_{2}} 
 *{\circ}="B"
\ar@{-} "B"; (20,10)="C" 
\ar@{.} "C"; (30,10)="D" 
\ar@{-} "D"; (40,10) *++!D{\beta_{q-1}} *{\circ}="E"
\end{xy}
\\
%\rule[0pt]{0pt}{12pt}
  &  &  
\begin{tabular}{c}
 $\{\alpha_1\}$, $\{\alpha_{p-1}\}$,  $\{\beta_{1}\}$, $\{\beta_{q-1}\} \smallvstrut$; 
\\
%\rule[0pt]{0pt}{12pt}
%%  &  &  
{$\{\beta_i\} (\forall i)$ if $p=2 \smallvstrut$;} \\
%\rule[0pt]{0pt}{12pt}
%\rule[-5pt]{0pt}{0pt}
%%  &  &  
 any subset of $\Pi_K$ if $p=1 \smallvstrut$ \\
\end{tabular}
\\
\hline
%\rule[0pt]{0pt}{12pt}
%\rule[-5pt]{0pt}{0pt}
\begin{tabular}{c}
{\ }
\\
$\lie{so}_{2n+2}$ 
\\[2ex]
{\small $n\geq 3$ }
\end{tabular}
& $\lie{so}_{2n}\oplus\C$ &
\begin{xy}
\ar@{-} (0,0) *++!D{\alpha_1} *{\circ}="A"; (10,0)="C" 
\ar@{.} "C"; (20,0)="D" 
\ar@{-} "D"; (30,0) *+!DR{\alpha_{n-2}} *{\circ}="E"
\ar@{-} "E"; (35,8.6)  *+!L{\alpha_{n-1}} *{\circ}
\ar@{-} "E"; (35,-8.6)  *+!L{\alpha_n} *{\circ}
\end{xy}
\\
%\rule[0pt]{0pt}{12pt}
%\rule[-5pt]{0pt}{0pt}
%$n\geq 3$  
&  & $\{\alpha_{n-1}\}$, $\{\alpha_n\} \smallvstrut$
\\
\hline
%\rule[0pt]{0pt}{12pt}
%\rule[-5pt]{0pt}{0pt}
\begin{tabular}{c}
{\ }
\\
$\lie{so}_{2n+1}$ 
\\[2ex]
{\small $n\geq 3$ }
\end{tabular}
& $\lie{so}_{2n}$ &
\begin{xy}
\ar@{-} (0,0) *++!D{\alpha_1} *{\circ}="A"; (10,0)="C" 
\ar@{.} "C"; (20,0)="D" 
\ar@{-} "D"; (30,0) *+!DR{\alpha_{n-2}} *{\circ}="E"
\ar@{-} "E"; (35,8.6)  *+!L{\alpha_{n-1}} *{\circ}
\ar@{-} "E"; (35,-8.6)  *+!L{\alpha_n} *{\circ}
\end{xy}
\\
%\rule[0pt]{0pt}{12pt}
%\rule[-5pt]{0pt}{0pt}
%$n\geq 3$  
& & 
 any subset of $\Pi_K \smallvstrut$
\\
\hline
%\rule[0pt]{0pt}{12pt}
%\rule[-5pt]{0pt}{0pt}
$\lie{so}_{2n+2}$ & $\lie{so}_{2n+1}$ &
\begin{xy}
\ar@{-} (0,0) *++!D{\alpha_1} *{\circ}="A"; (10,0)="C" 
\ar@{.} "C"; (20,0)="D" 
\ar@{-} "D"; (30,0) *++!D{\alpha_{n-1}} *{\circ}="E"
\ar@{=>} "E"; (40,0) *++!D{\alpha_n} *{\circ}="F"
\end{xy}
\\
%\rule[0pt]{0pt}{12pt}
%\rule[-5pt]{0pt}{0pt}
{\small$n\geq 3$} & & 
 any subset of $\Pi_K \smallvstrut$
\\
\hline
%\rule[0pt]{0pt}{12pt}
%\rule[-5pt]{0pt}{0pt}
$\lie{so}_{2n+2}$ & $\lie{sl}_{n+1} \oplus \C$ &
\begin{xy}
\ar@{-} (0,0) *++!D{\alpha_1} *{\circ}="A"; (10,0) *++!D{\alpha_2} 
 *{\circ}="B"
\ar@{-} "B"; (20,0)="C" 
\ar@{.} "C"; (30,0)="D" 
\ar@{-} "D"; (40,0) *++!D{\alpha_{n}} *{\circ}="E"
\end{xy}
\\
%\rule[0pt]{0pt}{12pt}
%\rule[-5pt]{0pt}{0pt}
{\small$n\geq 3$} & & 
 $\{\alpha_1\}$, $\{\alpha_n\} \smallvstrut$
\\
\hline
\end{longtable}
}

\clearpage

\hfil
{Double flag {varieties} 
$ K/Q \times G/B_G $ of finite type (continued)}
\hfil

\begin{longtable}{c|c|c}
%%$G$-spherical $G/Q$
%%\label{table:g.spherical.continued}
%%\\
\hline
%\rule[0pt]{0pt}{12pt}
$\lie{g}$ & $\lie{k}$ & $\Pi_K \setminus J_K\ (Q=Q_{J_K})$
\\
\hline\hline
%\rule[0pt]{0pt}{12pt}
$\lie{sp}_{2(p+q)}$ & $\lie{sp}_{2p} \oplus \lie{sp}_{2q}$ &
\begin{xy}
\ar@{-} (0,0) *++!D{\alpha_1} *{\circ}="A"; (10,0)="C" 
\ar@{.} "C"; (20,0)="D" 
\ar@{-} "D"; (30,0) *++!D{\alpha_{p-1}} *{\circ}="E"
\ar@{<=} "E"; (40,0) *++!D{\alpha_p} *{\circ}="F"
\end{xy}
\\
%\rule[0pt]{0pt}{12pt}
%\rule[-5pt]{0pt}{0pt}
&{\small$1\leq p\leq q$}& 
\begin{xy}
\ar@{-} (0,0) *++!D{\beta_{1}} *{\circ}="A"; (10,0)="C" 
\ar@{.} "C"; (20,0)="D" 
\ar@{-} "D"; (30,0) *++!D{\beta_{q-1}} *{\circ}="E"
\ar@{<=} "E"; (40,0) *++!D{\beta_{q}} *{\circ}="F"
\end{xy}
\\
%\rule[0pt]{0pt}{12pt}
  &  &  
 $\{\alpha_1\}$, $\{\beta_{1}\}$;\ \ 
\\
%\rule[0pt]{0pt}{12pt}
  &  &  
 $\{\alpha_p\}$ if $p\leq 3$;\ $\{\beta_{q}\}$ if $p\leq 2$;
 \  $\{\beta_{q}\}$ if $q\leq 3$; 
 \\
%\rule[0pt]{0pt}{12pt}
  &  &  
 $\{\alpha_1, \alpha_2\}$ if $p=2$;
 \ \ $\{\beta_{1}, \beta_{2}\}$ if $q=2$; 
 \\
%\rule[0pt]{0pt}{12pt}
  &  &  
 $\{\beta_i\} (\forall i)$, 
%%if $p=1$; 
$\{\beta_i, \beta_j\} (\forall i, j)$ if $p=1$ \\[.7ex]
\hline
%\rule[0pt]{0pt}{12pt}
%\rule[-5pt]{0pt}{0pt}
$\lie{f}_4$ & $\lie{so}_9$ & 
\begin{xy}
\ar@{-} (0,0) *++!D{\alpha_1} *{\circ}="A"; (10,0) *++!D{\alpha_2} 
 *{\circ}="B"
\ar@{-} "B"; (20,0)*++!D{\alpha_3}  *{\circ}="C"
\ar@{=>} "C"; (30,0) *++!D{\alpha_4}  *{\circ}="D"
\end{xy}
\\
%\rule[0pt]{0pt}{12pt}
%\rule[-5pt]{0pt}{0pt}
  &  &  
 $\{\alpha_i\} (\forall i)$,
 $\{\alpha_1, \alpha_2\} \smallvstrut$ \\
\hline
%\rule[0pt]{0pt}{12pt}
%\rule[-5pt]{0pt}{0pt}
$\lie{e}_6$ & $\lie{so}_{10} \oplus \C$ &
\begin{xy}
\ar@{-} (0,0) *+!DR{\alpha_1} *{\circ}="A"; (10,0) *+!DR{\alpha_2} 
 *{\circ}="B"
\ar@{-} "B"; (20,0) *+!DR{\alpha_3} *{\circ}="C" 
\ar@{-} "C"; (25,8.6)  *+!L{\alpha_4} *{\circ}
\ar@{-} "C"; (25,-8.6)  *+!L{\alpha_5} *{\circ}
\end{xy}
\\
%\rule[0pt]{0pt}{12pt}
%\rule[-5pt]{0pt}{0pt}
  &  & 
 $\{\alpha_1\}$\\
\hline
%\rule[0pt]{0pt}{12pt}
%\rule[-5pt]{0pt}{0pt}
$\lie{e}_6$ & $\lie{f}_4$ &
\begin{xy}
\ar@{-} (0,0) *++!D{\alpha_1} *{\circ}="A"; (10,0) *++!D{\alpha_2} 
 *{\circ}="B"
\ar@{=>} "B"; (20,0)*++!D{\alpha_3}  *{\circ}="C"
\ar@{-} "C"; (30,0) *++!D{\alpha_4}  *{\circ}="D"
\end{xy}
\\
%\rule[0pt]{0pt}{12pt}
%\rule[-5pt]{0pt}{0pt}
  &  & 
 $\{\alpha_1\}$\\
\hline
\end{longtable}

\addtocounter{table}{-1}

\clearpage 

\subsection{Classification when $ Q = B_K $ is a Borel subgroup of $ K $}\label{subsection:Table.Q=Borel}
%%table: finite type where Q = B_K Borel of K

\vskip -3ex

{
\begin{longtable}{c|c|c}
\caption{Double flag {varieties} $ K/B_K \times G/P $ of finite type}
%%$K$-spherical $G/P$
\label{table:k.spherical}
\vphantom{$\Biggm|$}
\\
\hline
%\rule[0pt]{0pt}{12pt}
$\lie{g}$ & $\lie{k}$ & $\Pi \setminus J\ (P=P_J)$
\\
\hline\hline
\multicolumn{1}{c}{
%\rule[0pt]{0pt}{12pt}
%\rule[-5pt]{0pt}{0pt}
$\lie{sl}_{n}$} & \multicolumn{1}{c}{} & 
\multicolumn{1}{c}{
\begin{xy}
\ar@{-} (0,0) *++!D{\alpha_1} *{\circ}="A"; (10,0) *++!D{\alpha_2} 
 *{\circ}="B"
\ar@{-} "B"; (20,0)="C" 
\ar@{.} "C"; (30,0)="D" 
\ar@{-} "D"; (40,0) *++!D{\alpha_{n-1}} *{\circ}="E"
\end{xy}
}\\[1ex]
\hline

%\rule[0pt]{0pt}{12pt}
%\rule[-5pt]{0pt}{0pt}
$\lie{sl}_{n}$ & $\lie{so}_{n}$ & $\{\alpha_i\} (\forall i) \smallvstrut$ \\
\hline

%\rule[0pt]{0pt}{12pt}
$\lie{sl}_{2m}$ & $\lie{sp}_{2m}$ &
 $\{\alpha_i\} (\forall i)$,
 $\{\alpha_i, \alpha_{i+1}\} (\forall i)$,
\\
{\small$2m\,{=}\,n \,{\geq}\, 4$} 
&&
 $\{\alpha_1, \alpha_i\} (\forall i)$, 
 $\{\alpha_i, \alpha_{n-1}\}  (\forall i)$, \\
&& $\{\alpha_1, \alpha_2, \alpha_3\}$,
 $\{\alpha_{n-3}, \alpha_{n-2}, \alpha_{n-1}\}$,
\\
%\rule[-5pt]{0pt}{0pt}
&& $\{\alpha_1, \alpha_2, \alpha_{n-1}\}$,
 $\{\alpha_1, \alpha_{n-2}, \alpha_{n-1}\}$
 \\
\hline

%\rule[0pt]{0pt}{12pt}
$\lie{sl}_{p+q}$ & $\lie{sl}_p\oplus\lie{sl}_q\oplus\C$&
 $\{\alpha_i\} (\forall i)$,
 $\{\alpha_i, \alpha_{i+1}\} (\forall i)$,\\
{\small$p+q=n$} & {\small$1\leq p\leq q$} &
 $\{\alpha_1, \alpha_i\} (\forall i)$, 
 $\{\alpha_i, \alpha_{n-1}\} (\forall i)$; \\
&& $\{\alpha_i, \alpha_j\} (\forall i, j)$ if $p=2$;
\\
%\rule[-5pt]{0pt}{0pt}
&& 
any subset of $\Pi$ if $p=1$\\
\hline\hline

\multicolumn{1}{c}{
%\rule[0pt]{0pt}{12pt}
%\rule[-5pt]{0pt}{0pt}
$\lie{so}_{2n+1}$} & \multicolumn{1}{c}{} & 
\multicolumn{1}{c}{
\begin{xy}
\ar@{-} (0,0) *++!D{\alpha_1} *{\circ}="A"; (10,0) *++!D{\alpha_2} 
 *{\circ}="B"
\ar@{-} "B"; (20,0)="C" 
\ar@{.} "C"; (30,0)="D" 
\ar@{-} "D"; (40,0) *++!D{\alpha_{n-1}} *{\circ}="E"
\ar@{=>} "E"; (50,0) *++!D{\alpha_n} *{\circ}="F"
\end{xy}
}\\[1.5ex]
\hline

%\rule[0pt]{0pt}{12pt}
$\lie{so}_{p+q}$ & $\lie{so}_p \oplus \lie{so}_q$
 & $\{\alpha_1\}$, $\{\alpha_n\}$; \\
{\small$p{+}q=2n{+}1$} & {\small$1\leq p\leq q$} &
 $\{\alpha_i\} (\forall i)$ if $p=2$; 
\\
%\rule[-5pt]{0pt}{0pt}
&& 
any subset of $\Pi$ if $p=1$\\
\hline\hline

\multicolumn{1}{c}{
%\rule[0pt]{0pt}{12pt}
%\rule[-5pt]{0pt}{0pt}
$\lie{so}_{2n}$} &\multicolumn{1}{c}{} & 
\multicolumn{1}{c}{
\begin{xy}
\ar@{-} (0,0) *++!D{\alpha_1} *{\circ}="A"; (10,0) *++!D{\alpha_2} 
 *{\circ}="B"
\ar@{-} "B"; (20,0)="C" 
\ar@{.} "C"; (30,0)="D" 
\ar@{-} "D"; (40,0) *+!DR{\alpha_{n-2}} *{\circ}="E"
\ar@{-} "E"; (45,8.6)  *+!L{\alpha_{n-1}} *{\circ}
\ar@{-} "E"; (45,-8.6)  *+!L{\alpha_n} *{\circ}
\end{xy}
}\\[1ex]
\hline

\rule[0pt]{0pt}{2.5ex}
$\lie{so}_{p+q}$ & 
$\lie{so}_p \oplus \lie{so}_q$ \renewcommand{\thefootnote}{\fnsymbol{footnote}}\footnotemark[2] 
& $\{\alpha_1\}$,
 $\{\alpha_{n-1}\}$, $\{\alpha_n\}$; \\
\begin{tabular}{c}
{\small$p+q=2n$} 
\\
{\small$n\geq 4$} 
\end{tabular}
& {\small$1\leq p\leq q$} &
\begin{tabular}{c}
$\{\alpha_i\} (\forall i)$, if $p=2$,\\
$\{\alpha_i, \alpha_{n-1}\} (\forall i)$, if $p=2$,\\
$\{\alpha_i, \alpha_{n}\} (\forall i)$ if $p=2$;
\end{tabular}
\\
%%&& 
%%%\rule[-5pt]{0pt}{0pt}
&& any subset of $\Pi$ if $p=1$\\
\hline

\rule[0pt]{0pt}{2.5ex}
%\rule[-5pt]{0pt}{0pt}
$\lie{so}_{2n}$ & 
 $\lie{sl}_n\oplus\C \; $\renewcommand{\thefootnote}{\fnsymbol{footnote}}\footnotemark[3]
& $\{\alpha_1\}$,
 $\{\alpha_2\}$, $\{\alpha_3\}$, $\{\alpha_{n-1}\}$, $\{\alpha_n\}$, \\
{\small$n\geq 4$} && $\{\alpha_1,\alpha_2\}$, $\{\alpha_1,\alpha_{n-1}\}$,
\\
&& $\{\alpha_1,\alpha_n\}$, $\{\alpha_{n-1},\alpha_n\}$,\\
&& $\{\alpha_2, \alpha_3\}${, $\{\alpha_2, \alpha_4\}$} if $n=4$\\
\hline
\end{longtable}

\clearpage

\begin{longtable}{c|c|c}
\hline
\multicolumn{1}{c}{
%\rule[0pt]{0pt}{12pt}
%\rule[-5pt]{0pt}{0pt}
$\lie{sp}_{2n}$} &\multicolumn{1}{c}{} & 
\multicolumn{1}{c}{
\begin{xy}
\ar@{-} (0,0) *++!D{\alpha_1} *{\circ}="A"; (10,0) *++!D{\alpha_2} 
 *{\circ}="B"
\ar@{-} "B"; (20,0)="C" 
\ar@{.} "C"; (30,0)="D" 
\ar@{-} "D"; (40,0) *++!D{\alpha_{n-1}} *{\circ}="E"
\ar@{<=} "E"; (50,0) *++!D{\alpha_n} *{\circ}="F"
\end{xy}
}\\[1.5ex]
\hline

%\rule[0pt]{0pt}{12pt}
%\rule[-5pt]{0pt}{0pt}
$\lie{sp}_{2n}$ & $\lie{sl}_n\oplus \C$ & $\{\alpha_1\}$,
 $\{\alpha_n\}$ \\
\hline

%\rule[0pt]{0pt}{12pt}
$\lie{sp}_{2(p+q)}$ & $\lie{sp}_{2p} \oplus \lie{sp}_{2q}$ & 
{$\{\alpha_1\}$,
 $\{\alpha_2\}$, $\{\alpha_3\}$,
 $\{\alpha_n\}$, $\{\alpha_1,\alpha_2\}$;} \\
$p+q=n$ & $1\leq p\leq q$ &
 {$\{\alpha_i\} (\forall i)$ if $p\leq 2$;}\\
%\rule[-5pt]{0pt}{0pt}
&& $\{\alpha_i, \alpha_j\} (\forall i, j)$ if $p=1$\\
\hline\hline

\multicolumn{1}{c}{
%\rule[0pt]{0pt}{12pt}
%\rule[-5pt]{0pt}{0pt}
$\lie{f}_4$} & \multicolumn{1}{c}{} & 
\multicolumn{1}{c}{
\begin{xy}
\ar@{-} (0,0) *++!D{\alpha_1} *{\circ}="A"; (10,0) *++!D{\alpha_2} 
 *{\circ}="B"
\ar@{=>} "B"; (20,0)*++!D{\alpha_3}  *{\circ}="C"
\ar@{-} "C"; (30,0) *++!D{\alpha_4}  *{\circ}="D"
\end{xy}
}\\[1ex]
\hline

%\rule[0pt]{0pt}{12pt}
%\rule[-5pt]{0pt}{0pt}
$\lie{f}_4$ & $\lie{so}_9$ & $\{\alpha_1\}$,
 $\{\alpha_2\}$, $\{\alpha_3\}$, $\{\alpha_4\}$,
 $\{\alpha_1,\alpha_4\}$ \\
\hline\hline

\multicolumn{1}{c}{
%\rule[0pt]{0pt}{12pt}
%\rule[-5pt]{0pt}{0pt}
$\lie{e}_6$} &\multicolumn{1}{c}{} & 
\multicolumn{1}{c}{
\begin{xy}
\ar@{-} (0,0) *++!D{\alpha_1} *{\circ}="A"; (10,0) *++!D{\alpha_3} 
 *{\circ}="B"
\ar@{-} "B"; (20,0)*++!U{\alpha_4}  *{\circ}="C"
\ar@{-} "C"; (30,0) *++!D{\alpha_5}  *{\circ}="D"
\ar@{-} "C"; (20,10) *++!D{\alpha_2}  *{\circ}="G"
\ar@{-} "D"; (40,0) *++!D{\alpha_6}  *{\circ}="E"
\end{xy}
}\\
\hline

%\rule[0pt]{0pt}{12pt}
%\rule[-5pt]{0pt}{0pt}
$\lie{e}_6$ & $\lie{sp}_8$ & $\{\alpha_1\}$, $\{\alpha_6\}$ \\
\hline

%\rule[0pt]{0pt}{12pt}
%\rule[-5pt]{0pt}{0pt}
$\lie{e}_6$ & $\lie{sl}_6\oplus \lie{sl}_2$ & $\{\alpha_1\}$, $\{\alpha_6\}$ \\
\hline

%\rule[0pt]{0pt}{12pt}
%\rule[-5pt]{0pt}{0pt}
$\lie{e}_6$ & $\lie{so}_{10}\oplus \C$
 & $\{\alpha_1\}$, $\{\alpha_2\}$,
 $\{\alpha_3\}$, $\{\alpha_5\}$, $\{\alpha_6\}$,
 $\{\alpha_1,\alpha_6\}$ \\
\hline

%\rule[0pt]{0pt}{12pt}
%\rule[-5pt]{0pt}{0pt}
$\lie{e}_6$ & $\lie{f}_4$ & $\{\alpha_1\}$, $\{\alpha_2\}$,
 $\{\alpha_3\}$, $\{\alpha_5\}$, $\{\alpha_6\}$,\\
& & 
 $\{\alpha_1,\alpha_2\}$, $\{\alpha_2,\alpha_6\}$, $\{\alpha_1,\alpha_3\}$,
 $\{\alpha_5,\alpha_6\}$\\
\hline\hline

\multicolumn{1}{c}{
%\rule[0pt]{0pt}{12pt}
%\rule[-5pt]{0pt}{0pt}
$\lie{e}_7$} &\multicolumn{1}{c}{} & 
\multicolumn{1}{c}{
\begin{xy}
\ar@{-} (0,0) *++!D{\alpha_1} *{\circ}="A"; (10,0) *++!D{\alpha_3} 
 *{\circ}="B"
\ar@{-} "B"; (20,0)*++!U{\alpha_4}  *{\circ}="C"
\ar@{-} "C"; (30,0) *++!D{\alpha_5}  *{\circ}="D"
\ar@{-} "C"; (20,10) *++!D{\alpha_2}  *{\circ}="G"
\ar@{-} "D"; (40,0) *++!D{\alpha_6}  *{\circ}="E"
\ar@{-} "E"; (50,0) *++!D{\alpha_7}  *{\circ}="F"
\end{xy}}\\
\hline

%\rule[0pt]{0pt}{12pt}
%\rule[-5pt]{0pt}{0pt}
$\lie{e}_7$ & $\lie{sl}_8$ & $\{\alpha_7\}$ \\
\hline

%\rule[0pt]{0pt}{12pt}
%\rule[-5pt]{0pt}{0pt}
$\lie{e}_7$ & $\lie{so}_{12}\oplus \lie{sl}_2$ & $\{\alpha_7\}$ \\
\hline

%\rule[0pt]{0pt}{12pt}
%\rule[-5pt]{0pt}{0pt}
$\lie{e}_7$ & $\lie{e}_6\oplus \C$ & $\{\alpha_1\}$, $\{\alpha_2\}$,
 $\{\alpha_7\}$ \\
\hline
\end{longtable}
}

\renewcommand{\thefootnote}{\fnsymbol{footnote}}\footnotetext[2]{
The subalgebra $\lie{k}$ of $\lie{g}$
 is not unique for
 $(\lie{g},\lie{k})\simeq(\lie{so}_8, \lie{so}_7)$,
 $(\lie{so}_8, \lie{so}_6 \oplus \C)$,
 $(\lie{so}_8, \lie{so}_5 \oplus \lie{so}_3)$.  
See \cite[Remark 5.3]{HNOO.2013} for details.
%%Remark~\ref{remark:choice.of.k}
}
\renewcommand{\thefootnote}{\fnsymbol{footnote}}\footnotetext[3]{
For $\lie{g}\simeq \lie{so}_{4n}$,
 a symmetric subalgebra $\lie{k}$ that is isomorphic to
 $\lie{sl}_{2n}\oplus \C$ is not unique up to inner automorphisms of $\lie{g}$.
Here we take $K$ as the Levi part $L_{\Pi\setminus\{\alpha_n\}}$
 of the parabolic subgroup $P_{\Pi\setminus\{\alpha_n\}}$.
}

\subsection{Triple flag variety of finite type, {where $ P_3 = B_G $}}\label{subsec:Table.3FV.Borel}

If one of the parabolic subgroup $ P_3 $ is a Borel subgroup $ B_G $ of $ G $, the triple flag varieties 
$ \dblFV = G/P_1 \times G/P_2 \times G/B_G $ 
are classified by Stembridge \cite{Stembridge.2003}.  
We reproduce his classification here.

\clearpage

%%table: triple flag variety of finite type
%
{
\setcounter{table}{3}
\begin{longtable}{c|c}
\caption{$G/P_1 \times G/P_2 \times G/B$ of finite type}
\label{table:group.spherical}
%\vphantom{$\Biggm|$}
\\
\hline
\rule[0pt]{0pt}{12pt}
$\lie{g}$ & $(\Pi \setminus J_1,\ \Pi \setminus J_2)$
 up to switching $J_1$ and $J_2$
\\
\hline\hline
\rule[0pt]{0pt}{12pt}
\rule[-5pt]{0pt}{0pt}
$\lie{sl}_{n+1}$ & 
\begin{xy}
\ar@{-} (0,0) *++!D{\alpha_1} *{\circ}="A"; (10,0) *++!D{\alpha_2} 
 *{\circ}="B"
\ar@{-} "B"; (20,0)="C" 
\ar@{.} "C"; (30,0)="D" 
\ar@{-} "D"; (40,0) *++!D{\alpha_{n}} *{\circ}="E"
\end{xy}\\

\rule[0pt]{0pt}{12pt}
 & $(\{\alpha_i\}, \{\alpha_j\}) (\forall i,j)$,
 $(\{\alpha_1, \alpha_i\}, \{\alpha_j\}) (\forall i,j)$,  \\

\rule[0pt]{0pt}{12pt}
 & $(\{\alpha_i, \alpha_n\}, \{\alpha_j\}) (\forall i,j)$, 
$(\{\alpha_i, \alpha_{i+1}\}, \{\alpha_j\}) (\forall i,j)$, \\

\rule[0pt]{0pt}{12pt}
 & $(\{\alpha_i, \alpha_j\}, \{\alpha_2\}) (\forall i,j)$,
   $(\{\alpha_i, \alpha_j\}, \{\alpha_{n-1}\}) (\forall i,j)$, \\
%
%\rule[0pt]{0pt}{12pt}
%\rule[-5pt]{0pt}{0pt}
 & 
$(\{\alpha_1\}, \text{any})$, $(\{\alpha_n\}, \text{any}) \smallvstrut$ \\

\hline

\rule[0pt]{0pt}{12pt}
\rule[-5pt]{0pt}{0pt}
$\lie{so}_{2n+1}$ & 
\begin{xy}
\ar@{-} (0,0) *++!D{\alpha_1} *{\circ}="A"; (10,0) *++!D{\alpha_2} 
 *{\circ}="B"
\ar@{-} "B"; (20,0)="C" 
\ar@{.} "C"; (30,0)="D" 
\ar@{-} "D"; (40,0) *++!D{\alpha_{n-1}} *{\circ}="E"
\ar@{=>} "E"; (50,0) *++!D{\alpha_n} *{\circ}="F"
\end{xy}\\

\rule[0pt]{0pt}{12pt}
\rule[-5pt]{0pt}{0pt}
 & $(\{\alpha_1\}, \{\alpha_i\}) (\forall i)$,
 $(\{\alpha_n\}, \{\alpha_n\}) \smallvstrut$ \\

\hline

\rule[0pt]{0pt}{12pt}
\rule[-5pt]{0pt}{0pt}
$\lie{so}_{2n}$  &
\begin{xy}
\ar@{-} (0,0) *++!D{\alpha_1} *{\circ}="A"; (10,0) *++!D{\alpha_2} 
 *{\circ}="B"
\ar@{-} "B"; (20,0)="C" 
\ar@{.} "C"; (30,0)="D" 
\ar@{-} "D"; (40,0) *+!DR{\alpha_{n-2}} *{\circ}="E"
\ar@{-} "E"; (45,8.6)  *+!L{\alpha_{n-1}} *{\circ}
\ar@{-} "E"; (45,-8.6)  *+!L{\alpha_n} *{\circ}
\end{xy}\\

\rule[0pt]{0pt}{12pt}
 $n\geq 4$
 & $(\{\alpha_1\}, \{\alpha_i\}) (\forall i)$,
 $(\{\alpha_i\}, \{\alpha_j\}) (i=1,2,3,\ j=n-1,n)$, \\

\rule[0pt]{0pt}{12pt}
 & $(\{\alpha_{n-1}\}, \{\alpha_{n-1}\})$,
 $(\{\alpha_{n-1}\}, \{\alpha_n\})$,
 $(\{\alpha_{n}\}, \{\alpha_n\})$, \\

\rule[0pt]{0pt}{12pt}
 & $(\{\alpha_1\}, \{\alpha_i,\alpha_{n-1}\}) (\forall i)$,
 $(\{\alpha_1\}, \{\alpha_i,\alpha_n\}) (\forall i)$, \\

\rule[0pt]{0pt}{12pt}
 & $(\{\alpha_{n-1}\}, \{\alpha_1,\alpha_2\})$,
 $(\{\alpha_n\}, \{\alpha_1,\alpha_2\})$,\\

\rule[0pt]{0pt}{12pt}
 & 
 {$(\{\alpha_i\}, \{\alpha_j,\alpha_k\})
 (i=n-1,n,\ \{j,k\}\subset\{1,n-1,n\})$;}
 \\

\rule[0pt]{0pt}{12pt}
\rule[-5pt]{0pt}{0pt}
 & 
 {$(\{\alpha_3\}, \{\alpha_2,\alpha_4\})$  if $n=4$;}\ 
 $(\{\alpha_4\}, \{\alpha_2,\alpha_3\})$ if $n=4 \smallvstrut$\\

\hline

\rule[0pt]{0pt}{12pt}
\rule[-5pt]{0pt}{0pt}
$\lie{sp}_{2n}$ &
\begin{xy}
\ar@{-} (0,0) *++!D{\alpha_1} *{\circ}="A"; (10,0) *++!D{\alpha_2} 
 *{\circ}="B"
\ar@{-} "B"; (20,0)="C" 
\ar@{.} "C"; (30,0)="D" 
\ar@{-} "D"; (40,0) *++!D{\alpha_{n-1}} *{\circ}="E"
\ar@{<=} "E"; (50,0) *++!D{\alpha_n} *{\circ}="F"
\end{xy}\\

\rule[0pt]{0pt}{12pt}
\rule[-5pt]{0pt}{0pt}
 & $(\{\alpha_1\}, \{\alpha_i\}) (\forall i)$,
 $(\{\alpha_n\}, \{\alpha_n\})$ \\

\hline

\rule[0pt]{0pt}{12pt}
\rule[-5pt]{0pt}{0pt}
$\lie{e}_6$ &
\begin{xy}
\ar@{-} (0,0) *++!D{\alpha_1} *{\circ}="A"; (10,0) *++!D{\alpha_3} 
 *{\circ}="B"
\ar@{-} "B"; (20,0)*++!U{\alpha_4}  *{\circ}="C"
\ar@{-} "C"; (30,0) *++!D{\alpha_5}  *{\circ}="D"
\ar@{-} "C"; (20,10) *++!D{\alpha_2}  *{\circ}="G"
\ar@{-} "D"; (40,0) *++!D{\alpha_6}  *{\circ}="E"
\end{xy}\\

\rule[0pt]{0pt}{12pt}
 & $(\{\alpha_i\}, \{\alpha_j\}) (i=1,6,\ j\neq 4)$,
\\
\rule[0pt]{0pt}{12pt}
\rule[-5pt]{0pt}{0pt}
&
 $(\{\alpha_1\}, \{\alpha_1,\alpha_6\})$,
 $(\{\alpha_6\}, \{\alpha_1,\alpha_6\})$ \\

\hline

\rule[0pt]{0pt}{12pt}
\rule[-5pt]{0pt}{0pt}
$\lie{e}_7$ &
\begin{xy}
\ar@{-} (0,0) *++!D{\alpha_1} *{\circ}="A"; (10,0) *++!D{\alpha_3} 
 *{\circ}="B"
\ar@{-} "B"; (20,0)*++!U{\alpha_4}  *{\circ}="C"
\ar@{-} "C"; (30,0) *++!D{\alpha_5}  *{\circ}="D"
\ar@{-} "C"; (20,10) *++!D{\alpha_2}  *{\circ}="G"
\ar@{-} "D"; (40,0) *++!D{\alpha_6}  *{\circ}="E"
\ar@{-} "E"; (50,0) *++!D{\alpha_7}  *{\circ}="F"
\end{xy}\\

\rule[0pt]{0pt}{12pt}
\rule[-5pt]{0pt}{0pt}
& $(\{\alpha_7\}, \{\alpha_1\})$,
 $(\{\alpha_7\}, \{\alpha_2\})$,
 $(\{\alpha_7\}, \{\alpha_7\})$ \\
\hline
\end{longtable}
}

\subsection{Triple flag variety of finite type, {where $ G $ is classical}}\label{subsec:Table.3FV.general}

%%table: triple flag variety of finite type

In this subsection, 
we list triple flag varieties 
$ \dblFV = G/P_1 \times G/P_2 \times G/P_3 $ 
of finite type, where $ G $ is classical.  
These classifications are quoted from  
the tables of \cite{MWZ.1999,MWZ.2000,Matsuki.2015,Matsuki.arXiv2019}.  
We only list the triples 
when $ P_i = P_{J_i} \;\; (i = 1, 2, 3) $ are proper parabolic subgroups, 
where $ J_i \subset \Pi $ is a subset of simple roots which generates the Levi part of $ P_i $.  

\subsubsection{Type A}\label{subsubsec:table.3FV.typeA}

We quote the results from \cite[Theorem~2.2]{MWZ.1999} except for the trivial case of double flag variety, which is denoted $ (A_{q,r}) $.  
In the following table, $ G = \GL_n $ and the numbering of simple roots $ \Pi $ is given as in \S~\ref{subsec:Table.3FV.Borel}.
(We use different labeling $ \alpha, \beta, \gamma $ to indicate the simple roots for $ P_1, P_2 $ and $ P_3 $ respectively.)
In the table, $ r $ denotes the number of simple roots in $ \Pi \setminus J_3 $.
\\[2ex]
%
%%\vskip -3ex
%
\begin{table}[H]
\caption{Type A triple flag varieties of finite type}\label{table:finite3FV.type.A}
\vphantom{$\Biggm|$}
%%\hfil{Table~\ref{table:finite3FV.type.A}: Type A triple flag varieties of finite type}\hfil
%%\\[1ex]
\resizebox{.98\linewidth}{!}{
%%\begin{longtable}{c|c|c|c}
\begin{tabular}{c|c|c|c}
\hline
\rule[0pt]{0pt}{12pt}
code in \cite{MWZ.1999} & $ \Pi \setminus J_1 $ & $ \Pi \setminus J_2 $ & $ \Pi \setminus J_3 $ \\
\hline\hline
$ (D_{r + 2}) $ & $ \{ \alpha_i \} \smallvstrut$ & $ \{ \beta_j \} $  & any subset 
%%($ |P_3| = r $) 
\\[1ex]
\cline{1-4}
$ (E_6) $ & $ \{ \alpha_i \} \smallvstrut$ & $ \{ \beta_{j_1}, \beta_{j_2} \} $  & $ \{ \gamma_{k_1}, \gamma_{k_2} \} $  \\[1ex]
\cline{1-4}
$ (E_7) \smallvstrut$ & $  \{ \alpha_i \} $ %%$ |P_1| = 2 $ 
& $ \{ \beta_{j_1}, \beta_{j_2} \} $ %%$ |P_2| = 3 $ 
& $ \{ \gamma_{k_1}, \gamma_{k_2} , \gamma_{k_3} \} $ %%$ |P_3| = 4 $ 
\\[1ex]
\cline{1-4}
$ (E_8) \smallvstrut$ & $  \{ \alpha_i \} $ %%$ |P_1| = 2 $ 
& $ \{ \beta_{j_1}, \beta_{j_2} \} $ %%$ |P_2| = 3 $ 
& $ \{ \gamma_{k_1}, \gamma_{k_2} , \gamma_{k_3} , \gamma_{k_4} \} $ %%$ |P_3| = 5 $ 
\\[1ex]
\cline{1-4}
$ (E_{r + 3}^{(a)}) \smallvstrut$ & $ \{ \alpha_2 \}, \{ \alpha_{n - 2} \} $ 
& $ \{ \beta_{j_1}, \beta_{j_2} \} $ %%$ |P_2| = 3 $ 
& any subset %%($ |P_3| = r $)  
\\[1ex]
\cline{1-4}
$ (E_{r + 3}^{(b)}) \smallvstrut$ & $ \{ \alpha_i \} $ & $ \{ \beta_j, \beta_{j + 1} \}, \{ \beta_1, \beta_j \}, \{ \beta_j, \beta_{n - 1} \} $ 
& any subset %%($ |P_3| = r $) 
\\[1ex]
\cline{1-4}
$ (S_{q, r}) \smallvstrut$ & $ \{ \alpha_1 \}, \{ \alpha_{n - 1} \} $ & any subset %%($ |P_2| = q $) 
& any subset %%($ |P_3| = r $) 
\\[1ex]
\cline{1-4}
%\hline
%\end{longtable}
\end{tabular}
}
\end{table}

\clearpage

\subsubsection{Type B}

We quote the results from \cite[Theorem~1.6]{Matsuki.2015}.  
In the following table, $ G = \SO_{2n+1} $ and the numbering of simple roots $ \Pi $ is given as in \S~\ref{subsec:Table.3FV.Borel}.
%%(We use different labeling $ \alpha, \beta, \gamma $ to indicate the simple roots for $ P_1, P_2 $ and $ P_3 $ respectively.)
%
\begin{longtable}{c|c|c|c}
\caption{Type B triple flag varieties of finite type}
\label{table:finite3FV.type.B}
\vphantom{$\Biggm|$}
\\
\hline
\rule[0pt]{0pt}{12pt}
code in \cite{Matsuki.2015} & $ \Pi \setminus J_1 $ & $ \Pi \setminus J_2 $ & $ \Pi \setminus J_3 $ \\
\hline\hline
(I) & $ \{ \alpha_n \} \smallvstrut$ & $ \{ \beta_n \} $  & any subset \\[1ex]
\cline{1-4}
(II) & $ \{ \alpha_1 \} \smallvstrut$ & $ \{ \beta_j \} $  & any subset \\[1ex]
\cline{1-4}
(III) & $ \{ \alpha_i \} \smallvstrut$ & $ \{ \beta_j \} $  & $ \{ \gamma_n \} $ \\[1ex]
\cline{1-4}
(IV) & $ \{ \alpha_i \} \smallvstrut$ & $ \{ \beta_n \} $  & $ \{ \gamma_i, \gamma_j \} $ \\[1ex]
\cline{1-4}
%\hline
\end{longtable}

\subsubsection{Type C}

We quote the results from \cite[Theorem~1.2]{MWZ.2000} except for the trivial case of double flag variety, which is denoted $ (\mathrm{Sp}A_{q,r}) $.
$ G = \Sp_{2n} $ and the numbering of simple roots $ \Pi $ is given as in \S~\ref{subsec:Table.3FV.Borel}.
(We use different labeling $ \alpha, \beta, \gamma $ to indicate the simple roots for $ P_1, P_2 $ and $ P_3 $ respectively.)
In the table, $ r $ denotes the number of simple roots in $ \Pi \setminus J_3 $.
%
%%\vskip -3ex
%
\begin{longtable}{c|c|c|c}
\caption{Type C triple flag varieties of finite type}
\label{table:finite3FV.type.C}
\vphantom{$\Biggm|$}
\\
\hline
\rule[0pt]{0pt}{12pt}
code in \cite{MWZ.2000} & $ \Pi \setminus J_1 $ & $ \Pi \setminus J_2 $ & $ \Pi \setminus J_3 $ \\
\hline\hline
$ (\mathrm{Sp}D_{r + 2}) $ & $ \{ \alpha_n \} \smallvstrut$ & $ \{ \beta_n \} $  & any subset 
%%($ |P_3| = r $) 
\\[1ex]
\cline{1-4}
$ (\mathrm{Sp}E_6) $ & $ \{ \alpha_n \} \smallvstrut$ & $ \{ \beta_i \} $  & $ \{ \gamma_j \} $  \\[1ex]
\cline{1-4}
$ (\mathrm{Sp}E_7) $ & $ \{ \alpha_n \} \smallvstrut$ & $ \{ \beta_i \} $  & $ \{ \gamma_j, \gamma_n \} $  \\[1ex]
\cline{1-4}
$ (\mathrm{Sp}E_8) $ & $ \{ \alpha_n \} \smallvstrut$ & $ \{ \beta_i \} $  & $ \{ \gamma_{j_1}, \gamma_{j_2} \} $  \\[1ex]
\cline{1-4}
$ (\mathrm{Sp}E_{r + 3}^{(b)}) \smallvstrut$ & $ \{ \alpha_n \} $ & $ \{ \beta_1 \} $ & any subset 
%%($ |P_3| = r $) 
\\[1ex]
\cline{1-4}
$ (\mathrm{Sp}Y_{r + 4}) \smallvstrut$ & $ \{ \alpha_1 \} $ & $ \{ \beta_i \} $ & any subset 
%%($ |P_3| = r $) 
\\[1ex]
\cline{1-4}
%\hline
\end{longtable}

\clearpage 

\subsubsection{Type D}

We quote the results from \cite[Theorem~1.7]{Matsuki.arXiv2019}.
In the following table, $ G = \SO_{2n} $ and the numbering of simple roots $ \Pi $ is given as in \S~\ref{subsec:Table.3FV.Borel}.  
(See Table~\ref{table:finite3FV.type.D} on the next page.)
%%(We use different labeling $ \alpha, \beta, \gamma $ to indicate the simple roots for $ P_1, P_2 $ and $ P_3 $ respectively.)
%
\begin{table}[hbtp]
\caption{Type D triple flag varieties of finite type}\label{table:finite3FV.type.D}
\vphantom{$\Biggm|$}
%%\hfil{Table~\ref{table:finite3FV.type.D}: Type D triple flag varieties of finite type}\hfil
%%\\[1ex]
%
\resizebox{.99\linewidth}{!}{
%%\begin{longtable}{c|c|c|c}
\begin{tabular}{c|c|c|c}
\hline
\rule[0pt]{0pt}{12pt}
code in \cite{Matsuki.arXiv2019} & $ \Pi \setminus J_1 $ & $ \Pi \setminus J_2 $ & $ \Pi \setminus J_3 $ \\
\hline\hline
(I-1) & $ \{ \alpha_1 \} \smallvstrut$ & $ \{ \beta_j \} $  & any subset \\[1ex]
\cline{1-4}
(I-2) & $ \{ \alpha_1 \} \smallvstrut$ & $ \{ \beta_j, \beta_{n -1} \}, \{ \beta_j, \beta_{n} \} $  & any subset \\[1ex]
\cline{1-4}
(II) & $ \{ \alpha_{n - 1} \}, \{ \alpha_n \} \smallvstrut$ & 
\begin{tabular}{c}
$ \{ \beta_1 \}, \{ \beta_2 \}, \{ \beta_3 \}, \{ \beta_1, \beta_2 \}, \smallvstrut$
\\
$ \{ \beta_1, \beta_{n - 1} \}, \{ \beta_1, \beta_n \},  \smallvstrut$  
\\
$ \{ \beta_{n -1} \}, \{ \beta_{n} \}, \{ \beta_{n -1}, \beta_{n} \} \smallvstrut$
\end{tabular}
& any subset \\[1ex]
\cline{1-4}
(III-1) & $ \{ \alpha_{n - 1} \}, \{ \alpha_n \} \smallvstrut$ & $ \{ \beta_j \} \; (4 \leq j \leq n - 2) $  & $ \{ \gamma_k \} $ \\[1ex]
\cline{1-4}
(III-2) & $ \{ \alpha_{n - 1} \}, \{ \alpha_n \} \smallvstrut$ & $ \{ \beta_j \} \; (4 \leq j \leq n - 2) $  & $ \{ \gamma_{k_1}, \gamma_{k_2} \} $ \\[1ex]
\cline{1-4}
(III-3) & $ \{ \alpha_{n - 1} \}, \{ \alpha_n \} \smallvstrut$ & $ \{ \beta_j \} \; (4 \leq j \leq n - 2) $  & 
\begin{tabular}{c}
%%$ \{ \gamma_1, \gamma_k, \gamma_{n - 1} \}, \{ \gamma_1, \gamma_k, \gamma_n \} $ 
$ \{ \gamma_1, \gamma_k, \gamma_{\nu} \} \; (\nu \in \{ n - 1, n \}) $, \\ 
%%$ \{ \gamma_k, \gamma_{k + 1}, \gamma_{n - 1} \}, \{ \gamma_k, \gamma_{k + 1}, \gamma_n \} $ 
$ \{ \gamma_k, \gamma_{k + 1}, \gamma_{\nu} \} \; (\nu \in \{ n - 1, n \}) $, \\
%%$ \{ \gamma_k, \gamma_{n - 2}, \gamma_{n - 1} \}, \{ \gamma_k, \gamma_{n - 2}, \gamma_n \} $ 
$ \{ \gamma_k, \gamma_{n - 1}, \gamma_n \} $, \\
$ \{ \gamma_1, \gamma_2, \gamma_{k + 2} \}, \{ \gamma_1, \gamma_{k + 1}, \gamma_{k + 2} \}, $ \\
$ \{ \gamma_k, \gamma_{k + 1}, \gamma_{k + 2} \} $ 
\end{tabular}
\\[1ex]
\cline{1-4}
(III-4) & $ \{ \alpha_{n - 1} \}, \{ \alpha_n \} \smallvstrut$ & $ \{ \beta_j \} \; (4 \leq j \leq n - 2) $  & 
\begin{tabular}{c}
$ \{ \gamma_1, \gamma_2, \gamma_3, \gamma_{\nu} \} \; (\nu \in \{ n - 1, n \}) $, \\
$ \{ \gamma_1, \gamma_2, \gamma_{n - 1}, \gamma_n \} $, \\
$ \{ \gamma_1, \gamma_{n - 2}, \gamma_{n - 1}, \gamma_n \} $, \\
$ \{ \gamma_{n - 3}, \gamma_{n - 2}, \gamma_{n - 1}, \gamma_n \} $, \\
$ \{ \gamma_1, \gamma_2, \gamma_3, \gamma_4 \} $ \\
\end{tabular}
\\[1ex]
\cline{1-4}
%\hline
\end{tabular}
%%\end{longtable}
}
\end{table}

%%\clearpage

\section{Quivers and multiple flag varieties}\label{sec:quiver.joint.variety}

As we explained in \S~\ref{subsec:3FV} and \S~\ref{subsec:DFl.symmetric.pair}, 
triple flag varieties are a special case of double flag varieties of symmetric pairs.  
For classical groups, those multiple flag varieties and their orbits have fruitful descriptions in terms of quiver representations.  To see it, let us explain the theory of Kac on quiver representations.  

\subsection{Orbits in quiver representations: Theory of Kac}

Let $ Q = (Q_0, Q_1) $ be a quiver, where $ Q_0 $ is the set of vertices, and $ Q_1 $ is the set of arrows.  
Each $ e \in Q_1 $ has the unique initial vertex $ s(e) $ and the terminal one $ t(e) $.  
So, in general,  $ Q $ is an oriented graph, which admits multiple edges and loops.  
However, in this paper, we only consider \emph{simply laced acyclic quivers} with single edges and no loop.
A representation $ \pi $ of the quiver $ Q $ is an assignment of a vector space $ V_i $ to each vertex $ i \in Q_0 $ 
together with an assignment of a linear map $ f_e : V_{s(e)} \to V_{t(e)} $ 
for each arrow $ e \in Q_1 $. 
The collection of dimensions $ d(\pi) = (\dim V_i)_{i \in Q_0} $ is called the dimension vector.  
Two representations $ \pi $ and $ \pi' $ are said to be isomorphic if there are isomorphisms 
$ \varphi_i : V_i \to V_i' \;\; (i \in Q_0) $ such that for each $ e \in Q_1 $, the diagram
\begin{equation*}
\vcenter{
\xymatrix @R-2ex @M+.5ex @C-3ex @L+.5ex @H+1ex {
V_{i} \ar[rr]^{f_e} \ar[d]^{\varphi_{i}} & & V_{j} \ar[d]_{\varphi_{j}}
\\
V_{i}' \ar[rr]^{f_e'} & & V_{j}'
}} \qquad 
(\text{where $ i = s(e) $ and $ j = t(e) $})
\end{equation*}
is commutative.

Fix a dimension vector $ \alpha \in \Z_{\geq 0}^{Q_0} $, and consider 
all the representations 
$ \pi $ of $ Q $ with the same dimension vector $ \alpha $.  Denote the space of such representations 
by 
\begin{equation*}
\ModQ^{\alpha}(Q) = \{ \pi \in \Rep Q \mid d(\pi) = \alpha \} .
\end{equation*}
For this variety, we fix vector spaces $ V_i $ for each $ i \in Q_0 $.  
Let $ G^{\alpha} = \prod_{i \in Q_0} \GL_{\alpha_i} $, which acts naturally on 
$ \ModQ^{\alpha}(Q) $ and it is easy to see that the orbits of $ G^{\alpha} $ represent isomorphism classes of representations of $ Q $ with dimension vector $ \alpha $. 

Obviously, every finite dimensional representation of a quiver $ Q $ decomposes 
into a direct sum of indecomposable representations.  
There exists an ingenious description of such indecomposable modules by Kac (\cite{Kac.Quiver.1980,Kac.QuiverII.1982}), 
which is originally due to Gabriel for the Dynkin quiver of finite type (\cite{Gabriel.1972}).
Kac's idea is to relate the quiver to a Cartan matrix $ A $ defined by 
\begin{equation*}
A = 2 I_{Q_0} - R, 
\qquad
\text{where } 
\begin{aligned}[t]
&
R = (r_{i,j})_{i,j \in Q_0} , \quad
\\
&
r_{i, j} = \# \{ e \in Q_1 \mid 
\{ s(e), t(e) \} = \{ i, j \} \}.
\end{aligned}
\end{equation*}
Thus, $ r_{i,j} $ denotes the number of edges which connect $ i $ and $ j $ so that $ R $ is a so-called adjacency matrix of the underlying graph of $ Q $.
We define a quadratic form in $ x = (x_i)_{i \in Q_0} $ by 
\begin{equation*}
(x|x) = \dfrac{1}{2} \transpose{x} A x = \norm{x}^2 - \dfrac{1}{2} \sum_{i, j \in Q_0} r_{i, j} x_i x_j ,
\end{equation*}
which is called a \emph{Tits form}.  
The following lemma is easy to prove, but the statement is crucial below.

\begin{lemma}
\label{lemma:dim.orbit}
$ \dim \ModQ^{\alpha}(Q) - (\dim G^{\alpha} - 1) = 1 - (\alpha|\alpha) $.  
\end{lemma}

Note that the action of $ G^{\alpha} $ on $ \ModQ^{\alpha}(Q) $ is not effective 
and $ G^{\alpha}/ \Cbatsu $ acts effectively.

Let $ \lie{g}(A) $ be the Kac-Moody Lie algebra defined by the Cartan matrix $ A = (a_{i,j})_{i, j \in Q_0} $ (cf.~\cite{Kac.KMalgebra.1990}), 
and $ \Pi = \{ \alpha_i \mid i \in Q_0 \} $ the set of simple roots so that it satisfies 
$ (\alpha_i|\alpha_j) = \frac{1}{2} a_{i,j} $.    
Let $ \Gamma = \{ \sum_{i \in Q_0} k_i \alpha_i \mid k_i \in \Z \} $ be the root lattice and 
define $ \Gamma^+ = \{ \sum_i k_i \alpha_i \mid k_i \geq 0 \} $.  
We can (and do) identify a dimension vector $ \alpha = (k_i)_{i \in Q_0} $ 
with an element $ \sum_{i \in Q_0} k_i \alpha_i $ 
in $ \Gamma^+ $.   
Denote $ \Delta $ and $ \Delta^+ $ its root system and the set of positive roots respectively.  
There are notions of real roots $ \Delta_{re}^+ $ and imaginary roots $ \Delta_{im}^+ $, 
the definitions of which we refer the readers to \cite{Kac.KMalgebra.1990}.  

\begin{theorem}[{Kac \cite[Theorem 3]{Kac.Quiver.1980}}]
\begin{penumerate}
\item
For any $ \alpha \in \Gamma^+ $, 
$ \ModQ^{\alpha}(Q) $ contains an indecomposable module if and only if 
$ \alpha \in \Delta^+ $.
\item
There exists a unique isomorphism class of indecomposable modules in $ \ModQ^{\alpha}(Q) $ if    and only if 
$ \alpha \in \Delta_{re}^+ $, which is equivalent to $ (\alpha|\alpha) = 1 $.  
In this case, indecomposable modules in $ \ModQ^{\alpha}(Q) $ consist of an open dense orbit of $ G^{\alpha} $.  
\item
If $ \alpha \in \Delta_{im}^+ $, then there are infinitely many isomorphism classes of indecomposable modules in $ \ModQ^{\alpha}(Q) $.  
There exists a family of orbits of indecomposable modules with $ 1 - (\alpha|\alpha) $ parameters.
\end{penumerate}
\end{theorem}

\subsection{Orbits in triple flag varieties}

Let us briefly introduce the classification after Magyar-Weyman-Zelevinsky \cite{MWZ.1999,MWZ.2000} 
in the case of type A.

Let 
$ \mbfa = (a_1, \dots, a_{\ell_{\mbfa}}) $ be a composition of $ n $.  
We allow zero for several components.  For this $ \mbfa $, we define 
a flag variety 
\begin{equation*}
\FlagVar_{\mbfa} = \{ F = (F_i)_{i = 1}^{\ell_{\mbfa}} \mid 
\dim {F_{i}/F_{i - 1} = a_i \ \text{for}\ 1\leq i\leq \ell_{\mbfa}} \} ,
\end{equation*}
where we put $ F_0 = \{ 0 \} $ for convenience.  
Note that $ F_{\ell_{\mbfa}} = \C^n =: V $.

Consider $ N $-tuple of compositions $ (\mbfa_1, \mbfa_2, \dots, \mbfa_N) $, 
where each $ \mbfa_j = (a_i^{(j)})_{i = 1}^{\ell_j} $ is a composition of $ n $.  
Define a multiple flag variety 
$ \multipleFV = \FlagVar_{\mbfa_1} \times \FlagVar_{\mbfa_2} \times \cdots \times \FlagVar_{\mbfa_N} 
= \prod_{j = 1}^N \FlagVar_{\mbfa_j} $.  

For $ \multipleFV $, consider a star shaped quiver $ Q $.  
The underlying graph of $ Q $ is a tree,
which has $ N $-branches of length $ \ell_1, \ell_2, \dots, \ell_N $ with a unique root.  
The orientation is given from the end vertex of each branch toward the root vertex.
We give an example of $ 3 $-branched quiver below.  
%%So let us consider the following star-shaped quiver with 3 arms.  
Here, 
each branch represents flags associated to the compositions $ \mbfa, \mbfb, \mbfc $, 
of length $ p, q $ and $ r $ respectively.
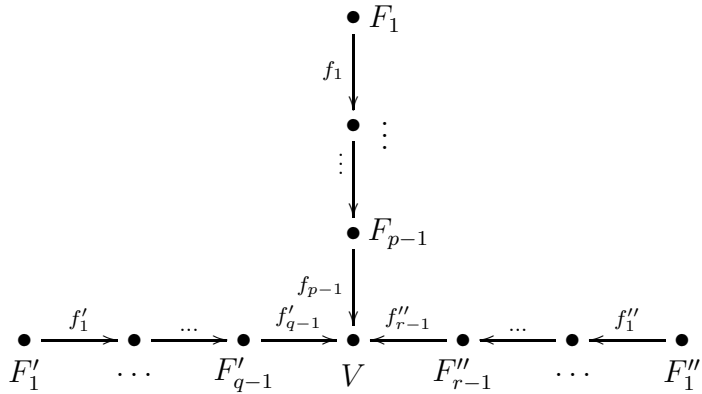
\begin{equation*}
\xymatrix@+=1cm{
& & & 
\bullet \save[]+<2.3ex,-.1ex>*{F_1}\restore \ar[d]_{f_1}
\\
& & & 
\bullet \save[]+<2.3ex,-.1ex>*{\vdots}\restore \ar[d]_(.3){\vdots}
\\
& & & 
\bullet \save[]+<3.3ex,-.1ex>*{F_{p- 1}}\restore \ar[d]_{f_{p-1}}
\\
\bullet \save[]+<0ex,-2.5ex>*{F_1'}\restore \ar[r]^{f_1'} & 
\bullet \save[]+<0ex,-2.5ex>*{\dots}\restore \ar[r]^{\dots} & 
\bullet \save[]+<0ex,-2.5ex>*{F_{q-1}'}\restore \ar[r]^{f_{q - 1}'} & 
\bullet \save[]+<0ex,-2.5ex>*{V}\restore & 
\bullet \save[]+<0ex,-2.5ex>*{F_{r-1}''}\restore \ar[l]_{f_{r - 1}''} & 
\bullet \save[]+<0ex,-2.5ex>*{\dots}\restore \ar[l]_{\dots} & 
\bullet \save[]+<0ex,-2.5ex>*{F_1''}\restore \ar[l]_{f_1''} & 
\\
}
\end{equation*}

We can associate a point in $ \multipleFV $
\begin{equation*}
(F^{(j)})_{j = 1}^N \in \prod_{j = 1}^N \FlagVar_{\mbfa_j} = \multipleFV
\end{equation*}
with a representation of $ Q $. 
Namely, in the $j$-th branch of $ Q $, we associate the flag $ F^{(j)} = (F_i^{(j)})_{i = 1}^{\ell_j} \in \FlagVar_{\mbfa_j} $ 
with a representation
\begin{equation*}
\xymatrix@+=1cm{
\bullet \save[]+<0ex,-2.5ex>*{F_1^{(j)}}\restore \ar[r]^{\varphi_1^{(j)}} & 
\bullet \save[]+<0ex,-2.5ex>*{F_2^{(j)}}\restore \ar[r]^{\varphi_2^{(j)}} & 
\bullet \save[]+<0ex,-2.5ex>*{\dots}\restore \ar[r]^{\dots} & 
\bullet \save[]+<0ex,-2.5ex>*{F_{\ell_j-1}^{(j)}}\restore \ar[r]^{\varphi_{\ell_j - 1}^{(j)}} & 
\bullet \save[]+<0ex,-2.5ex>*{V=F_{\ell_j}^{(j)}}\restore 
}
\end{equation*}
where $ \varphi_k^{(j)} : F_k^{(j)} \to F_{k + 1}^{(j)} $
is the inclusion map.  
In this way points in $ \multipleFV $ are embedded into
the representation space $ \Rep Q $ of the quiver.

\begin{lemma}
Let $ x, y \in \multipleFV $ be two $N$-tuples of flags
and $ \xi, \eta \in \Rep Q $ be the corresponding representations
of the quiver. 
Then $ x $ and $ y $ are in the same
$ \GL(V) $-orbit if and only if $ \xi $ and $ \eta $ are isomorphic
as representations of $ Q $.
\end{lemma}

To make use of this scheme ultimately,
we consider a full subcategory of $ \Rep Q $ consisting of 
those objects corresponding to $ N $-tuples of flags of length 
$ \ell_1, \ell_2, \dots, \ell_N $.  
Let $ \MultQRep(\mbfa_1, \dots, \mbfa_N) $ be the set of objects in $ \Rep Q $ 
which are isomorphic to representations 
corresponding to the multiple flags in $ \multipleFV = \prod_{j = 1}^N \FlagVar_{\mbfa_j} $.  
Then sum up all of these to get a full subcategory
\begin{equation*}
\MultQRep = \MultQRep_{\ell_1, \dots, \ell_N} = \bigcup \MultQRep(\mbfa_1, \dots, \mbfa_N) \subset \Rep Q,
\end{equation*}
where the union is taken over tuples of compositions of $n$ (of nonnegative integers including zero) with prescribed lengths, 
and $ n $ moves over all nonnegative integers. 
This subcategory is clearly additive and closed under the decomposition. 
To state it more precisely, let $ \pi \in \MultQRep $ and decompose it into indecomposable 
representations in $ \Rep Q $:
\begin{equation*}
\pi = \bigoplus_k \pi_k \qquad
(\pi_k \in \Rep Q : \text{indecomposable}). 
\end{equation*}
Because the category $ \Rep Q $ is equivalent to a category of modules of the path algebra of $ Q $, the Krull-Schmidt theorem applies. 
Thus a decomposition is always possible and it is indeed unique.  
It is easy to see the direct summands $ \pi_k $'s 
are belonging to $ \MultQRep $ again.  So the decomposition is indeed in the category $ \MultQRep $.

By Kac's theorem, the dimension vector
$ d_k = \dim \pi_k $ is a positive root of the Kac-Moody
algebra $ \lie{g}(A) $.  
Moreover if $ d_k $ is real, i.e. $ (d_k | d_k) = 1 $, then $ \pi_k $ is the unique indecomposable object (up to isomorphism) with the dimension vector $ d_k $.
If $ d_k $ is imaginary, i.e., $ (d_k | d_k) \leq 0 $, there are infinitely many non-isomorphic classes of indecomposable objects with dimension vector $ d_k $.

Let $ \multipleFV = \FlagVar_{\mbfa_1} \times \FlagVar_{\mbfa_2} \times \cdots \times \FlagVar_{\mbfa_N} $ 
be a multiple flag variety.  
We identify a multiple flag $ x \in \multipleFV $ with a representation $ \pi \in \Rep Q $. 
Then $ d = \dim \pi $ does not depend on the choice of $ x \in \multipleFV $.  
The decomposition of $ \pi $ into indecomposables
$ \pi = \bigoplus_k \pi_k $ induces a decomposition of $ d = \sum_k \dim \pi_k $.  

Let $ \Lambda = \Lambda_{\ell_1, \ell_2, \dots, \ell_N} = 
\{ \dim \sigma \mid \sigma \in \MultQRep \} $.  
Then the above considerations read as $ d = \sum_k d_k $ where $ d_k \in \Lambda $.

\begin{theorem}\label{thm:finiteness.criterion.via.quiver}
Let $ d $ be the dimension vector of any $ \pi \in \MultQRep $
corresponding to an $ N $-tuple of flags $ x \in \multipleFV = \prod_{j = 1}^N \FlagVar_{\mbfa_j} $.  
It is independent of the choice of $ x $, but only depending on $ (\mbfa_j)_{j = 1}^N $.   
\begin{penumerate}
\item\label{thm:finiteness.criterion.via.quiver:item:1}
The multiple flag variety $ \multipleFV $ is of finite type if and only if the following condition {\upshape(FT)} holds.

\medskip
\textbf{\upshape(FT)} \;
If there is a decomposition $ d = d' + d'' $ in $ \Lambda $ 
then $ (d'|d'), (d''|d'') \geq 1 $.  

\medskip

\item\label{thm:finiteness.criterion.via.quiver:item:2}
If it is of finite type, 
$ N $ must be less than or equal to $ 3 $ (at most a triple flag variety).  
Moreover, $ \GL(V) $-orbits in $ \multipleFV $ are in bijection with {decompositions} $ d = \sum_k d_k $ where $ d_k \in \Lambda $ and $ (d_k|d_k) = 1 $.
\end{penumerate}
\end{theorem}

\begin{proof}
\eqref{thm:finiteness.criterion.via.quiver:item:1}\;\;
If there is a summand $ d' $ where $ (d'|d') \leq 0 $, then 
$ d' $ corresponds to an imaginary root, and we know there are infinitely many orbits.  
Moreover, there are infinitely many such orbits that correspond to representations of the quiver with injective maps. (Indeed, the fact that $ d' $ is imaginary combined with Lemma \ref{lemma:dim.orbit}
implies that every orbit in $ \ModQ^{d'}(Q) $ has positive codimension therein; in addition, the subset of representations with injective maps is dense because of the reason of the rank, thus of the same dimension as $ \ModQ^{d'}(Q) $, hence {itself is a union of infinitely many such orbits.})
So direct summands should satisfy $ (d'|d') \geq 1 $.  
On the other hand if (FT) holds, 
we consider the unique decomposition by indecomposable representations.  
Pick a dimension vector $ d' $ of any indecomposable summand.  
Note that $ d' $ is a root, which is imaginary or real.
By (FT), we get $ (d'|d') \geq 1 $ and we know $ d' $ cannot be imaginary.  
Therefore it is a real root and there is only one isomorphism class.  
In this way, any decomposition produces only one isomorphism class {and there are only finitely many possibilities of the decompositions of dimension vectors.  }

\eqref{thm:finiteness.criterion.via.quiver:item:2}\;\;
Let us rewrite the Tits' form $ (d|d) $ in terms of $ ((\mbfa_j)_{j = 1}^N) $.  
It is easy to see
\begin{equation*}
  (d|d) = \dfrac{1}{2} \Bigl\{ \sum_{j = 1}^N \norm{\mbfa_j}^2 - (N - 2) n^2 \Bigr\} \qquad 
(n = \dim V) .
\end{equation*}
Using this formula, 
we can easily find that 
\begin{equation*}
(\mbfa_1, \dots, \mbfa_4) = ((1^2), (1^2), (1^2), (1^2)) \qquad
\text{(we omit $ 0 $'s)}
\end{equation*}
corresponds to an imaginary root. 
However, if $ N \geq 4 $, then the dimension vector corresponding to this particular $ (\mbfa_1, \dots, \mbfa_4) $ appears as a summand of \emph{any} dimension vector.  
Thus {multiple flag varieties with $N\geq 4$ factors} are not of finite type.
\end{proof}

\section{Finite type double flag varieties of type AIII}

Now we return back to the double flag variety 
$ \dblFV $ of type AIII.  

\subsection{Joint varieties and finiteness criterions}\label{subsection:JV}

To find out double flag varieties of finite type, 
we can also use the same idea along that of \cite{MWZ.1999}.  
For the double flag varieties of symmetric pairs of type AIII, 
this is carried out by Hiroki Homma \cite{Homma.2021}.  
Let us explain it below.

Consider a triple flag variety 
$$ 
\tripleFV(\mbfa,\mbfb,\mbfc) = \FlagVar_{\mbfa} \times \FlagVar_{\mbfb} \times \FlagVar_{\mbfc}
$$
where $ \mbfa, \mbfb, \mbfc $ are compositions of $ n $ that satisfy the condition 
\begin{equation*}
\sum_{i=1}^{\ell_{\mbfa}-1} a_i + \sum_{j=1}^{\ell_{\mbfb}-1} b_j = n .      
\end{equation*}
Sometimes we will write 
$ (\sum_{i=1}^{\ell_{\mbfa}-1} a_i, \sum_{j=1}^{\ell_{\mbfb}-1} b_j) = (p, q) $.  
Thus, if we write $\mbfa' = (a_i)_{i = 1}^{\ell_{\mbfa} - 1}$
for the sequence obtained from $\mbfa$ after removing the last part of the composition, 
then $ \mbfa' $ is a composition of $ p $ and 
$ \mbfb' $ is that of $ q $.  
Hereafter, we keep this notation $ \mbfa' $ for the above meaning.
Let us fix $ (p, q) $ and $ n = p + q $, 
and denote
$ K = \GL(V^+) \times \GL(V^-) $, where 
{$ V = V^+ \oplus V^- $ and }
$ V^+ = \C^p , V^- = \C^q $.  
Under this setting, we consider a double flag variety 
\begin{equation}\label{eq:doubleFVa'b'c}
\dblFV = \dblFV((\mbfa', \mbfb'); \mbfc) = 
K/Q \times G/P = \GL_p/Q_1 \times \GL_q/Q_2 \times \GL_n/P ,     
\end{equation}
where $ Q_1, Q_2, P $ are parabolic subgroups of $ \GL_p, \GL_q $ and $ \GL_n $ respectively such that
$ \GL_p/Q_1 = \FlagVar_{\mbfa'}$, $ \GL_q/Q_2 = \FlagVar_{\mbfb'}$, and $ \GL_n/P = \FlagVar_{\mbfc} $.

After Homma, we define a \emph{joint variety} $ \JVar(\mbfa, \mbfb, \mbfc) $  by
\begin{multline*}
\JVar(\mbfa, \mbfb, \mbfc) = 
\{ ((F_i^a)_{i = 1}^{\ell_{\mbfa}}, (F_i^b)_{i = 1}^{\ell_{\mbfb}}, (F_i^c)_{i = 1}^{\ell_{\mbfc}}) \in \FlagVar_{\mbfa} \times \FlagVar_{\mbfb} \times \FlagVar_{\mbfc} \mid 
\\
F_{\ell_{\mbfc}}^c = V = F_{\ell_{\mbfa} - 1}^a \oplus F_{\ell_{\mbfb} - 1}^b \}.
\end{multline*}
More generally, given a triple $ (\ell_1,\ell_2,\ell_3) $
of positive integers, we consider triples of compositions $ (\mbfa,\mbfb,\mbfc) $ such that
\begin{equation}
\label{condition:compositions}
    (\ell_{\mbfa},\ell_{\mbfb},\ell_{\mbfc}) = (\ell_1,\ell_2,\ell_3) \ \text{ and }\ 
    \sum_{i = 1}^{\ell_{\mbfa}-1} a_i + \sum_{j = 1}^{\ell_{\mbfb}-1} b_j = \sum_{k = 1}^{\ell_3} c_k
\end{equation}
and the semigroup
$$
\Lambda^J = \Lambda^J_{\ell_1,\ell_2,\ell_3}
:= \{ (\mbfa,\mbfb,\mbfc) \text{ satisfying (\ref{condition:compositions})}\}
$$
formed by such triples.
Moreover, we consider the quiver $ Q $ of triple branches of length 
$ \ell_1, \ell_2, \ell_3 $.  
Then any element of $\JVar(\mbfa,\mbfb,\mbfc)$ determines a representation of $ Q $ 
with the linear maps all injective as explained before.  
Let $ \JQRep(\mbfa,\mbfb,\mbfc) $
be the full subcategory of $ \Rep{Q} $ consisting of all the objects obtained in this way. 
%In summary, we have
%$$
%\dblFV(\mbfa,\mbfb,\mbfc)
%\subset
%\JVar(\mbfa,\mbfb,\mbfc)
%\subset
%\JQRep(\mbfa,\mbfb,\mbfc).
%$$
For studying the decompositions by indecomposables we need a larger subcategory.  
We sum up all $ \JQRep(\mbfa,\mbfb,\mbfc) $
when $ (\mbfa,\mbfb,\mbfc)\in \Lambda^J $ and denote the full subcategory of 
such objects by $ \JQRep $.

The next lemma clarifies the relation of the $ K $-orbits on the double flag varieties 
and the isomorphism classes of quiver representations.  
The proof of the lemma is straightforward and we will omit it.

\begin{lemma}\label{lemma:quiver-vs-orbits}
There are natural bijections between 
\begin{penumerate}
\item
isomorphism classes in $ \JQRep(\mbfa,\mbfb,\mbfc) $ as representations of quiver; 
\item
$ \GL(V) $-orbits on the joint flag variety $ \JVar(\mbfa,\mbfb,\mbfc) $;
\item
$ K $-orbits in the double flag variety $ \dblFV((\mbfa',\mbfb');\mbfc) $ 
defined in \eqref{eq:doubleFVa'b'c}.
\end{penumerate}
\end{lemma}

In the next lemma, we also point out some general properties of the orbits.
If $\mathcal{F},\mathcal{G}$ are objects in $\JQRep$, then we denote by $\Hom(\mathcal{F},\mathcal{G})$ the complex vector space of morphisms of quiver representations, and we set
\begin{equation}
\langle\mathcal{F},\mathcal{G}\rangle=\dim\Hom(\mathcal{F},\mathcal{G}).    
\end{equation}
Clearly $\langle\mathcal{F},\mathcal{G}\rangle$ only depends on $\mathcal{F},\mathcal{G}$ up to their isomorphism classes.
We actually focus on $K$-orbits of $\dblFV((\mbfa',\mbfb');\mbfc)$
and, if $\mathcal{F}$ is an element of $\JQRep(\mbfa,\mbfb,\mbfc)$,
then we denote the corresponding $K$-orbit by $\Xorbit_{\mathcal{F}}$ 
via Lemma~\ref{lemma:quiver-vs-orbits}.

\begin{lemma}\label{lemma:general.properties.orbits}
\begin{penumerate}
    \item $\dim\Xorbit_\mathcal{F}=p^2+q^2-\langle\mathcal{F},\mathcal{F}\rangle$.
    \item Given $\mathcal{F},\mathcal{F}'\in\JQRep(\mbfa,\mbfb,\mbfc)$
    the following implication holds:
    $$ \Xorbit_\mathcal{F}\subset\overline{\Xorbit_{\mathcal{F}'}} \implies \langle \mathcal{I},\mathcal{F}\rangle\geq \langle\mathcal{I},\mathcal{F}'\rangle\ \mbox{for all indecomposable object $\mathcal{I}\in\JQRep$.}
    $$
\end{penumerate}
\end{lemma}

\begin{proof}
(1) 
%%The formula follows by the following observations:
%%        \item $\dim K=p^2+q^2$.
Since $\mathcal{F} \in \JQRep(\mbfa,\mbfb,\mbfc)$, due to the definition of $\JQRep(\mbfa,\mbfb,\mbfc)$, every element in $\Hom(\mathcal{F},\mathcal{F})$ corresponds to an element $ f \in \End(V) $ such that $f(F)\subset F$ for all subspace $F$ in $\mathcal{F}$. Moreover, $V^+$ and $V^-$ themselves are among the subspaces of $\mathcal{F}$ hence we must have $f\in\gl(V^+)\times\gl(V^-)=\fk$, and finally
        $$\Hom(\mathcal{F},\mathcal{F})\cong\{f\in\fk:f(F)\subset F\ \forall F\in\mathcal{F}\}=\Lie(\Stab_K(\mathcal{F})).$$
        Thus we get $ \langle \mathcal{F}, \mathcal{F} \rangle = \dim \Stab_K(\mathcal{F})$, 
and consequently $ \dim \Xorbit_\mathcal{F} = \dim K/\Stab_K(\mathcal{F}) = p^2 + q^2 - \langle \mathcal{F}, \mathcal{F} \rangle $.

(2) follows from a general property of quiver representations; see e.g. \cite{Riedtmann.1986}.    
\end{proof}

\begin{remark}
\begin{penumerate}
\item
Note that 
    $$    p^2+q^2=\dim\dblFV((\mbfa', \mbfb'); \mbfc) + \frac{1}{\,2\,}\Bigl( \|\mbfa\|^2+\|\mbfb\|^2+\|\mbfc\|^2-n^2 \Bigr)  $$
    hence the formula in Lemma \ref{lemma:general.properties.orbits}\,(1) is an analogue of \cite[Proposition 4.6]{MWZ.1999}.

\item
The natural question is whether the converse of the implication in Lemma \ref{lemma:general.properties.orbits}\,(2) is valid. 
    There is a similar discussion in \cite[Proposition 4.5]{MWZ.1999}.
    In Theorem \ref{theorem:orbits.AIII} below we will show the equivalence 
    for one particular double flag variety.
%$\dblFV(\mbfa,\mbfb,\mbfc)$.
\end{penumerate}
\end{remark}

Now we concentrate on the study of the isomorphism classes in 
$ \JQRep $ as representations of the quiver $ Q $.  

\begin{lemma}
Any representation $ \pi \in \JQRep $ of $ Q $ decomposes into indecomposables in $ \Rep Q $ uniquely as 
    \begin{equation*}
        \pi = \bigoplus_k \pi_k ,
    \end{equation*}
where each indecomposable module $ \pi_k $ belongs to $ \JQRep $.
\end{lemma}

This lemma can be obtained in the same way as in the case of multiple flag varieties, and we omit the proof.
Note that in a composition $\mbfa$, we allow parts $a_i$ to be zero. Let $\mbfa^+$ be the list of nonzero parts of $\mbfa$ arranged in nonincreasing order.
For $(\mbfa,\mbfb,\mbfc)\in\Lambda^J$, let $d(\mbfa,\mbfb,\mbfc)$ be the corresponding dimension vector, i.e., the collection of dimensions of the subspaces in the triple of flags.

\begin{theorem}[{Homma \cite[Proposition 2.12 and Theorem 2.13]{Homma.2021}}]\label{thm:finite.type.AIII}
Let $ \boldsymbol{\lambda}=(\mbfa,\mbfb,\mbfc)\in \Lambda^J $.
The following conditions {\upshape(1)--(3)} are equivalent.  
\begin{penumerate}
\item
The double flag variety $ \dblFV=\dblFV((\mbfa',\mbfb');\mbfc) $
is of finite type.
\item 
Any $ \boldsymbol{\mu}=(\mbfa_0,\mbfb_0,\mbfc_0)\in\Lambda^J$ such that $ \boldsymbol{\lambda}-\boldsymbol{\mu}\in\Lambda^J$ satisfies $(d(\boldsymbol{\mu})|d(\boldsymbol{\mu}))\geq 1$.
\item 
There is no $\boldsymbol{\mu}=(\mbfa_0,\mbfb_0,\mbfc_0)\in\Lambda^J$ with $\boldsymbol{\lambda}-\boldsymbol{\mu}\in\Lambda^J$ such that $(\mbfa_0'^+,\mbfb_0'^+,\mbfc_0)$ is in the following list up to switching the first two partitions.  
\begin{eqnarray*}
& ((1^3),(1^3),(2^3)),\quad ((2^2),(1^5),(3^3)),\quad ((1^2),(1^2),(1^4)),\\ & ((3),(1^5),(2^4)),\quad
((2),(1^3),(1^5)),\quad ((3),(2^2),(1^7)).
\end{eqnarray*}
\end{penumerate}
\end{theorem}

The equivalence of (1) and (2) is already explained above.
For deriving the equivalence with (3), the proof is given in \cite[Theorem 2.13]{Homma.2021} based on case-by-case analysis.

We also point out that Homma's results include the complete list of indecomposable representations of the quiver in the case where $\dblFV$ is of finite type; see \cite[Theorem 2.15]{Homma.2021}. 
In principle this yields a parametrization of the orbits in terms of quiver representations.

\subsection{Classification of finite type double flag varieties of type AIII}\label{subsec:classification.table.typeAIII.DFV}

It is not immediate to read off an actual classification of double flag varieties of finite type from Theorem~\ref{thm:finite.type.AIII}.
Here we give a complete list describing the parabolic subgroups 
$ Q = Q_1 \times Q_2 \subset K $ and $ P \subset G $.

\begin{corollary}\label{cor:classification.AIII.finite.type}
The double flag variety $\dblFV=\dblFV((\mbfa',\mbfb');\mbfc)$ of type AIII is of finite type if and only if it appears (up to switching $ Q_1 $ and $ Q_2 $) among those in Table~\ref{table:AIII.finite.type}.
\end{corollary}

In the table, we use the above identification of $\dblFV$ with $\GL_p/Q_1\times \GL_q/Q_2\times \GL_n/P$.
Moreover we let $|P|=\ell_\mbfc$, $|Q_1|=\ell_\mbfa-1$, $|Q_2|=\ell_\mbfb-1$.

\begin{table}[htbp]
\caption{Double flag varieties of finite type in type AIII}\label{table:AIII.finite.type}
%
%%\hfil{Table~\ref{table:AIII.finite.type}: Double flag varieties of finite type in type AIII}\hfil
%%\\[2ex]
\hfil
\begin{tabular}{|c||c|c|c|}
\hline 
type &  $P$  &  $Q_1$  &  $Q_2$  \\ \hline \hline
(PA-1) & any & \multicolumn{2}{|c|}{$Q=K$ or mirabolic} \\ \hline
(PA-2) & any & any & $\GL_1$ \\ \hline
(PA-3) & any & \quad\quad maximal \quad\quad & $\GL_2$ \\ \hline\hline
(P6) & $|P|\leq 6$ & maximal & $\GL_q$ \\ \hline\hline
(P4-1) & $|P|\leq 4$ & $|Q_1|\leq 4$ & $\GL_q$ \\ \hline
(P4-2) & $|P|\leq 4$ \& $\min\,c_i=1$ & any & $\GL_q$  \\ \hline
{(P4-3)} & $|P|\leq 4$ & any & $\GL_2$ \\ \hline\hline
(P3-1) & $|P|\leq 3$ \& $\min\,c_i=1$ & any & any \\ \hline
(P3-2) & $|P|\leq 3$ \& $\min\,c_i=2$ & any & $|Q_2|\leq 2$ \\ \hline
(P3-3) & $|P|\leq 3$ & any & $\GL_q$ or mirabolic \\ \hline
(P3-4) & $|P|\leq 3$ & $|Q_1|\leq 4$ & maximal \\ \hline\hline
(P2) & \begin{tabular}{l}$|P|\leq 2$ \\ ($\GL_n$ or maximal) \end{tabular}
& any & any \\ \hline
\end{tabular}
\hfil
\end{table}

\vspace*{-2ex}

\begin{proof}
We first observe that if $\boldsymbol{\lambda}=(\mbfa,\mbfb,\mbfc)$ contains some summand $\boldsymbol{\mu}$ which is in the list given in Theorem \ref{thm:finite.type.AIII}\,(3), then we must have $|P|=\ell_\mbfc\geq 3$.
Thus, if $|P|=1$ or $2$ (that is, $P=\GL_n$ or maximal), then $\dblFV$ will be of finite type. This corresponds to the last row (P2) in Table \ref{table:AIII.finite.type}.

Next, we consider the case where $|P|=3$. 
Up to switching $\mbfa$ and $\mbfb$, we can assume that $|Q_1|=\ell_\mbfa-1\geq \ell_\mbfb-1=|Q_2|$.
\begin{itemize}
\item 
If $\min\,c_i=1$, then
$\boldsymbol{\lambda}$ has no summand arising in the list given in Theorem \ref{thm:finite.type.AIII}\,(3). 
This gives (P3-1). 
\end{itemize}
Now assume that $\min\,c_i \geq 2$. For avoiding summands corresponding to the case $ ((1^3),(1^3),(2^3)) $ 
in Theorem \ref{thm:finite.type.AIII}\,(3), it is necessary and 
sufficient to have $|Q_2| \leq 2$, so let us assume this condition below. 
\begin{itemize}
\item If $\min\,c_i=2$, then the other five triples in Theorem \ref{thm:finite.type.AIII}\,(3) are automatically avoided, and this yields (P3-2).
\item
If $\min\,c_i\geq 3$, then the triple $((2^2),(1^5),(3^3))$ is avoided if and only if $|Q_2|=1$; or $|Q_2|=2$ and $\min\{b_1,b_2\}=1$; or $|Q_1| \leq 4$.
(Note that, since $|Q_2| \leq 2$, $(1^5)$ cannot contribute to $\mbfb$.) Moreover, the remaining four triples of Theorem \ref{thm:finite.type.AIII}\,(3) are automatically avoided.  
This gives (P3-3) and (P3-4).
\end{itemize}

%Assuming $\min\,c_i=2$, it is of finite type exactly when it avoids summands corresponding to the case  $ ((1^3),(1^3),(2^3)) $ 
%in Theorem \ref{thm:finite.type.AIII}\,(3), which happens
%if and only if $|Q_2|\leq 2$.
%This gives (P3-2).
%
%Finally, we assume that $\min\,c_i\geq 3$.   
%If $|Q_2|\leq 3$, the summand $((1^3),(1^3),(2^3))$ will appear.  
%Thus we must have $|Q_2|\leq 2$
%
%for being of finite type $(\mbfa,\mbfb,\mbfc)$ has to avoid summands corresponding to $((1^3),(1^3),(2^3))$ and $((2^2),(1^5),(3^3))$,
%which happens if and only if 
%$|Q_2|\leq 2$ and $|Q_1|\leq 4$; or 
%$|Q_2|=1$; or 
%$|Q_2|=2$ and $\min\{b_1,b_2\}=1$. This gives (P3-3) and (P3-4).

Now suppose that $|P|=4$. 
Still we can assume that $|Q_1|\geq |Q_2|$.
Then $(\mbfa,\mbfb,\mbfc)$ will avoid summands corresponding to $((1^2),(1^2),(1^4))$
if and only if $|Q_2|=1$, that is $Q_2=\GL_q$.
In such a case, the first two triples in Theorem \ref{thm:finite.type.AIII}\,(3) are automatically avoided,
and to avoid the fourth one $((3),(1^5),(2^4))$, it is necessary and sufficient to have $q\leq 2$; or $|Q_1|\leq 4$; or $\min\,c_i=1$.
This gives (P4-1), (P4-2), (P4-3), and (PA-2) (in the special case $|P|=4$).

Next let $|P|\in\{5,6\}$. 
Still we assume $|Q_1|\geq |Q_2|$.
{Then, as above, $ (\mbfa, \mbfb, \mbfc) $ avoids $((1^2),(1^2),(1^4))$ if and only if $Q_2=\GL_q$. Now} the fifth triple $((2),(1^3),(1^5))$ will be avoided if and only if $q=1$ or $|Q_1|\leq 2$. In such a case, the other triples in Theorem \ref{thm:finite.type.AIII}\,(3) are also avoided. This corresponds to (P6) and {(PA-2)} (in the special case $|P|\in\{5,6\}$).

Finally, let $|P|\geq 7$.
Still we assume $|Q_1|\geq |Q_2|$.
First note that we must have $|Q_2|=1$ (i.e., $Q_2=\GL_q$), otherwise $((1^2),(1^2),(1^4))$ cannot be avoided.
If $q=1$ or $|Q_1|=1$, then the six triples are avoided.
If $q=2$, then for avoiding the fifth triple we must have $|Q_1|\leq 2$, and in this case all the triples will be avoided.
If $q\geq 3$, then for avoiding the fifth and sixth triples we must have $|Q_1|=1$ or ($|Q_1|=2$ and $\min\,a_i=1$) (i.e., $Q_1=\GL_p$ or mirabolic){, and then the other triples are also avoided}.
This yields (PA-1), (PA-2), (PA-3).
\end{proof}

\begin{remark}
The cases (PA-1)--(PA-3), and (P2) are already identified as finite type in \cite{NO.2011}.    
But all the other cases contain new series of finite type double flag varieties.  
\end{remark}

\subsection{Generalizations}

Theorem \ref{thm:finite.type.AIII} characterizes double flag
varieties of the form $\dblFV = (\GL_p\times\GL_q)/Q \times \GL_n/P$ with finitely many orbits of $K:=\GL_p\times\GL_q$. Actually, the fact that $K$ is a symmetric subgroup of $\GL_n$ does not play any role in the argument, and it is natural to consider the following more general problem.

\begin{problem}
Given a decomposition $V=\C^n=V_1\oplus\ldots\oplus V_N$
with $ p_i := \dim V_i $,
characterize the sequences of compositions $(\mbfa_1,\ldots,\mbfa_N; \mbfc) $ of $p_1,\ldots,p_N$ and $n$, respectively,
such that the multiple flag variety
\begin{equation}\label{X-general}
\Xfv = \bigl( \Flags_{\mbfa_1}(V_1)\times\cdots\times \Flags_{\mbfa_N}(V_N) \bigr)
\times 
\Flags_{\mbfc}(V) 
\end{equation}
has finitely many orbits for $L:=\GL_{p_1}\times\cdots\times\GL_{p_N}$.
\end{problem}

It is worth to note that, contrary to the case 
of usual multiple flag varieties of the form
$\Flags_{\mbfc_1}(V)\times\cdots\times\Flags_{\mbfc_N}(V)$,
for which it is necessary to have $N\leq 3$ for having a finite number of $\GL_n$-orbits (see \cite{MWZ.1999}), there exist multiple flag varieties $\Xfv$ of the form (\ref{X-general}) 
which are of finite type despite of an arbitrarily large number of factors $N$ as shown in the next example.  

\begin{example}
If we have $\mbfc=(1,n-1)$, so that $\Flags_\mbfc(V)=\mathbb{P}(V)$, then
the variety $\Xfv$ of (\ref{X-general}) will have a finite number of $L$-orbits for any choice of $N$ and $\mbfa_1,\ldots,\mbfa_N$.
Indeed, $\mathbb{P}(V)$ has a finite number of $T$-orbits for any maximal torus $T\subset \GL_n$, hence a fortiori it contains a finite number of $Q$-orbits, so that
$\Xfv \cong L/Q \times \mathbb{P}(V) $ has a finite number of $L$-orbits for any parabolic $Q\subset L$ (since $Q$ contains a maximal torus of $\GL_n$).
\end{example}

The next example is less immediate; the justification relies on a straightforward generalization of the criterion in Theorem \ref{thm:finite.type.AIII}\,(2) (we skip the details).

\begin{example}
Let $V=\C^n=\C^p\oplus \C^q\oplus \C$ with $n=p+q+1$, $p,q\geq 1$. Then the triple flag variety
\begin{align*}
&
\bigl( \mathbb{P}^{p-1}\times\mathbb{P}^{q-1} \bigr) \times \Grass_r(\C^n)
\\
& 
\simeq 
 \bigl( \GL_p/P_{(1, p - 1)} \times \GL_q/P_{(1, q- 1)} \times \GL_1/\GL_1 \bigr) \times \GL_n/P_{(r, n - r)}
\end{align*}
has a finite number of $\GL_p\times \GL_q\times \GL_1$-orbits.  
Here $ \Grass_r(V) $ denotes the Grassmann variety of $ r $-dimensional subspaces in $ V $.  

To check this, it is sufficient to show that every dimension vector of the form $d=(\mbfa_1,\mbfa_2,\mbfa_3, \mbfc)$ with 
$\mbfc=(r,n-r)$, $\mbfa_1=(1,p-1,q+1)$, $\mbfa_2=(1,q-1,p+1)$, $\mbfa_3=(1,n-1)$
satisfies $(d|d)\geq 1$.
Now, for seeing this, we compute
\begin{eqnarray*}
2(d|d) & = & \|\mbfc\|^2+\|\mbfa_1\|^2+\|\mbfa_2\|^2+\|\mbfa_3\|^2-2n^2 \\
 & = & r^2+(n-r)^2+3+(p-1)^2+(q+1)^2+(q-1)^2+(p+1)^2 \\
& & \qquad\hspace*{.45\textwidth} +(n-1)^2-2n^2 \\
 & = & (r-2)^2+2(p-\frac{r+1}{2})^2+2(q-\frac{r+1}{2})^2+1\geq 2.
\end{eqnarray*}
\end{example}

\section{Orbits in the double flag variety of type AIII}\label{section-7}
%%$\Grass_r(V^+\oplus V^-)\times\Flags(V^+)\times\Flags(V^-)$

In this section, we show a concrete example of classification of orbits in a double flag variety by using quiver representations.  
Thus we focus on the category of quiver representations $ \JQRep(\mbfa,\mbfb,\mbfc) $ and the double flag variety 
$ \dblFV((\mbfa',\mbfb');\mbfc) $ of type AIII (cf.~\cite{Fresse.N.2020,Fresse.N.2023}).  
We summarize the notation which will be used below:
\begin{itemize}
\item $n=p+q$, $r\leq n$ are positive integers, and we set
$$
\mbfa=(1^p,q),\quad \mbfb=(1^q,p),\quad \mbfc=(r,n-r);
$$
\item $V=\C^n=V^+\oplus V^-$ where $V^+=\C^p\times \{0\}^q$, $V^-=\{0\}^p\times \C^q$, 
\quad
with standard bases $(e_1,\ldots,e_n)=(e_1^+,\ldots,e_p^+,e_1^-,\ldots,e_q^-)$;
\smallskip
\item $G=\GL(V)=\GL_n$, \quad $K=\GL(V^+)\times \GL(V^-)=\GL_p\times\GL_q$, 
\\
$ P = P_{(r, n - r)} = \Stab_G(\C^r\times\{0\}^{n-r})$, \quad $B_K=B_p^+\times B_q^+$,  
\\
\smallskip
where $B_p^+\subset\GL_p$, $B_q^+\subset \GL_q$ are subgroups of upper-triangular matrices corresponding to $ \mbfa' = (1^p) $ and $ \mbfb' = (1^q) $ respectively;
\smallskip
\item $ \dblFV = \dblFV((\mbfa',\mbfb');\mbfc) = K/B_K \times G/P \simeq (\Flags(V^+)\times\Flags(V^-)) \times \Grass_r(V) $.
\end{itemize}

\subsection{Category of joint representations}

We consider the quiver $ Q = T_{2, p+1, q + 1} $ indicated below, 
which is also denoted as $ S(p+1, q+1) $ in \cite{MWZ.1999}.  
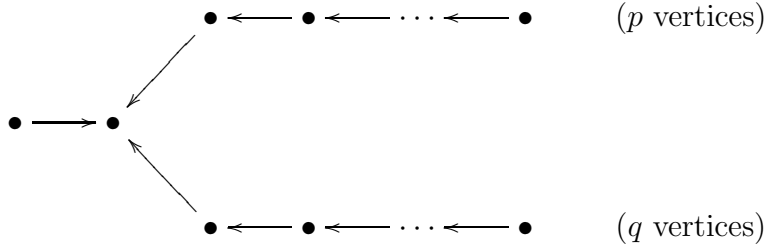
$$
\xymatrix{
&& \bullet \ar[ld] & \bullet \ar[l] & \cdots\ar[l] & \bullet \ar[l] & \text{($p$ vertices)} \\
\bullet \ar[r] & \bullet & \\
&&\bullet \ar[lu] & \bullet \ar[l] & \cdots\ar[l] & \bullet\ar[l]  & \text{($q$ vertices)}
}
$$
We denote by $ \JQRep $ the full subcategory of quiver representations $ \Rep Q $ consisting of 
the following representations 
$$
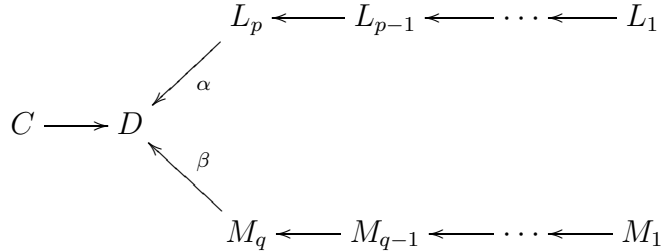
\xymatrix{
&&L_p \ar[ld]^\alpha & L_{p-1} \ar[l] & \cdots\ar[l] & L_1\ar[l] \\
C \ar[r] & D & \\
&&M_q\ar[lu]_\beta & M_{q-1} \ar[l] & \cdots\ar[l] & M_1\ar[l]
}
$$
where all maps are injective and $\Im\alpha\oplus\Im\beta=D$. 
Then $ \JQRep(\mbfa,\mbfb,\mbfc) $ is the subset of such representations satisfying $\dim C=r$, {$\dim D=n$}, $\dim L_i=i$, $\dim M_j=j$ for all $i,j$,
so that the points in the double flag variety $\dblFV = \dblFV((\mbfa',\mbfb');\mbfc) $ 
correspond to the objects in $\JQRep(\mbfa,\mbfb,\mbfc)$, for which $D=V$, $L_p=V^+$, and $M_q=V^-$ hold.  

In order to classify the $K$-orbits of $\dblFV$
by applying Lemma \ref{lemma:quiver-vs-orbits},
we describe the category $ \JQRep $ in more detail.
From \cite{Homma.2021}, we can describe the indecomposable objects which arise in the decomposition of elements in $ \JQRep $.
For $i\in\{1,\ldots,p\}$, we define
$$
L_\bullet^i=(L_1^i,\ldots,L_p^i)\quad\text{such that}\quad L_k^i=\left\{\begin{array}{ll}
\langle e_i^+\rangle & \mbox{if $k\geq i$} \\
0 & \mbox{if $k<i$.}
\end{array}\right.
$$
Furthermore, we set $L^0_\bullet=(L_1^0,\ldots,L_p^0)$ with $L_1^0=\ldots=L_p^0=0$.
We define in the same way $M_\bullet^j=(M_1^j,\ldots,M_q^j)$ with $e_j^-$ instead of $e_i^+$, and $M_\bullet^0=(M_1^0,\ldots,M_q^0)=(0,\ldots,0)$. 
We deduce the indecomposable objects $\I_i^+$, $\I_i'^+$, $\I_j^-$, $\I_j'^-$, $\I_{i,j}$ depicted in Figure \ref{F1}.
%on page \pageref{fig:indecomp.typeAIII.quiver}. 
For each one
we include (in parentheses) a graphical representation, which will be explained in \S~\ref{subsec:orbit.by.quivers.typeAIII} below.
\begin{figure}\label{fig:indecomp.typeAIII.quiver}
\resizebox{.9\linewidth}{!}{
\begin{minipage}{\textwidth}
\begin{eqnarray*}
&
\I_i^+\,\left(\ \GIiplus\ \right):\ 
\vcenter{\vbox{
\xymatrix{&&L_\bullet^i\ar[dl] \\ 
\langle e_i^+\rangle\ar[r] & \langle e_i^+\rangle & \\
&&M_\bullet^0\ar[lu]}
}}
\quad
\I_i'^+\,\left(\ \GIiprimeplus\ \right):\ 
\vcenter{\vbox{
\xymatrix{&&L_\bullet^i\ar[dl] \\ 
0\ar[r] & \langle e_i^+\rangle & \\
&&M_\bullet^0\ar[lu]}
}}
\\
&
\I_j^-\,\left(\ \GIjminus\ \right):\ 
\vcenter{\vbox{
\xymatrix{&&L_\bullet^0\ar[dl] \\ 
\langle e_j^-\rangle\ar[r] & \langle e_j^-\rangle & \\
&&M_\bullet^j\ar[lu]}
}}
\quad
\I_j'^-\,\left(\ \GIjprimeminus\ \right):\ 
\vcenter{\vbox{
\xymatrix{&&L_\bullet^0\ar[dl] \\ 
0\ar[r] & \langle e_j^-\rangle & \\
&&M_\bullet^j\ar[lu]}
}}
\\
&
\I_{i,j}\,\left(\ \GIij\ \right):\ 
\vcenter{\vbox{
\xymatrix{&&L_\bullet^i\ar[dl] \\ 
\langle e_i^++e_j^-\rangle\ar[r] & \langle e_i^+,e_j^-\rangle & \\
&&M_\bullet^j\ar[lu]}
}}
\end{eqnarray*}
\end{minipage}
}
\caption{The indecomposable objects $\mathcal{I}_i^+$, $\mathcal{I}_i'^+$ $\mathcal{I}_j^-$, $\mathcal{I}_j'^-$, $\mathcal{I}_{i,j}$ (for $1\leq i\leq p$ and $1\leq j\leq q$) of the category {$ \JQRep $}.}\label{F1}
\end{figure}
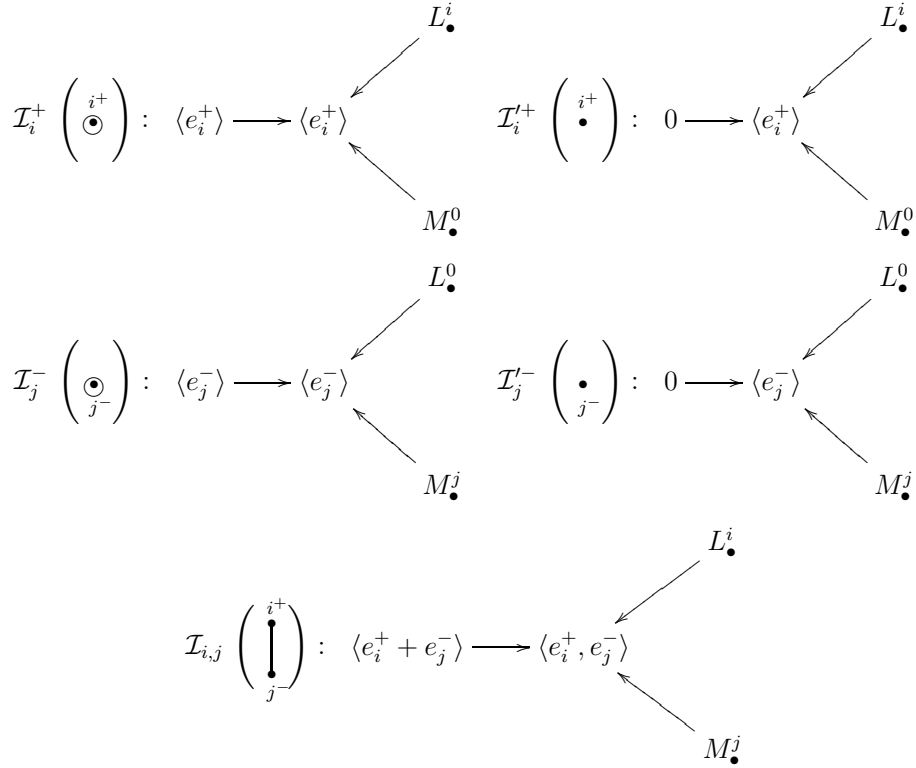

Recall that if $\mathcal{F},\mathcal{G}$ are objects in $\JQRep$, we have denoted by $\Hom(\mathcal{F},\mathcal{G})$ the complex vector space of morphisms of quiver representations, and set
$\langle\mathcal{F},\mathcal{G}\rangle=\dim\Hom(\mathcal{F},\mathcal{G})$.

\begin{lemma}\label{lemma:hom.indecomposables}
Let $ \mathcal{F} \in \JQRep(\mbfa,\mbfb,\mbfc) $ be a quiver representation corresponding to the point  
$ ((F_k^+)_{k=1}^p,(F_\ell^-)_{\ell=1}^q,W) \in \dblFV((\mbfa',\mbfb');\mbfc) $.
For all $(i,j)\in\{1,\ldots,p\}\times\{1,\ldots,q\}$, we have
\begin{eqnarray*}
&\Hom(\mathcal{I}_i^+,\mathcal{F})= F_i^+\cap W,\qquad 
\Hom(\mathcal{I}_i'^+,\mathcal{F}) =F_i^+, \\
&\Hom(\mathcal{I}_j^-,\mathcal{F})=F_j^-\cap W,\qquad
\Hom(\mathcal{I}_j'^-,\mathcal{F})=F_j^-, \\
&\Hom(\mathcal{I}_{i,j},\mathcal{F})=(F_i^+\oplus F_j^-)\cap W.
\end{eqnarray*}
\end{lemma}

\begin{proof}
Straightforward.
\end{proof}

\subsection{Orbit decomposition of $\dblFV$}\label{subsec:orbit.by.quivers.typeAIII}

We recall some results and notation from \cite{Fresse.N.2020,Fresse.N.2021,Fresse.N.2023}.  
Since 
\begin{equation*}
\dblFV = K/B_K \times G/P \simeq (\Flags(V^+)\times\Flags(V^-)) \times \Grass_r(V) ,
\end{equation*}
it is easy to see $ \dblFV / K \simeq B_K \backslash \Grass_r(V) $.  
On the other hand, a point in $ \Grass_r(V) $ represents an $ r $-dimensional subspace of $ V $, 
which can be represented as an image of rank $ r $ matrix $ A \in \Mat_{n, r} \; (\dim V = n = p + q) $.
It is well known that $ \Mat^{\circ}_{n, r}/\GL_r \simeq \Grass_r(V) $, where $ \Mat^{\circ}_{n, r} $ denotes the set of full rank (i.e., rank $ r $) matrices.
Thus we get $ \dblFV / K \simeq B_K \backslash \Mat^{\circ}_{p + q, r}/\GL_r $.  
In this way, as a representative of an orbit in $ \dblFV / K $, we can take a matrix from $ \Mat^{\circ}_{p + q, r} $.  
In fact, we can take ``partial permutations'' as representatives.

By $\ppermutations_{p,r}$ we denote the set of \emph{partial permutation matrices} of size $(p,r)$, i.e., matrices $\tau$ of this shape with coefficients in $\{0,1\}$ and at most one $1$ in each row and each column (consequently all the other entries are $ 0 $).  
We put 
\begin{equation*}
\ppermutations = \Bigl\{ \omega=\begin{pmatrix} \tau_1 \\ \tau_2 \end{pmatrix} \Bigm| \tau_1\in\ppermutations_{p,r},\tau_2\in\ppermutations_{q,r} \Bigr\}, 
\;\;
\ppermutations^{\circ} = \{ \omega \in \ppermutations \mid \rank \omega = r \} 
\subset\ppermutations .
\end{equation*}
The latter subset $ \ppermutations^{\circ} $ is equipped with a natural right action of the symmetric group $\permutationsof{r}$, and we finally 
put $\parameters=\ppermutations^{\circ}/\permutationsof{r}$.  

The set $ \parameters $ is in bijection with $ \dblFV / K $ (see \cite[Theorem 8.1]{Fresse.N.2020}).  
More precisely, we can describe a representative of any $ K $-orbit as follows.  
The range $[\omega]:=\mathrm{Im}\,\omega$ is well defined whenever $\omega\in\parameters$.
Let $\flag^+$ and $\flag^-$ denote the standard flags of $V^+$ and $V^-$ respectively.  
Every orbit in the double flag variety $ \dblFV/K $ has a representative of the form
$$
\flag_\omega:=(\flag^+,\flag^-,[\omega]) \in \dblFV \qquad (\omega\in\parameters).
$$

\begin{notation}
\label{notation:graph}
(a)
It is convenient to represent $\omega=\begin{pmatrix} \tau_1 \\ \tau_2 \end{pmatrix}\in\parameters$ by the graph $\graphic(\omega)$ given below.
\begin{itemize}
    \item $\graphic(\omega)$ has $p+q$ vertices labeled with
    $1^+,\ldots,p^+$ and $1^-,\ldots,q^-$, respectively.
    \item For every column of $\omega$ which contains two $1$'s, in the $i$-th and $j$-th rows of $\tau_1$ and $\tau_2$, respectively, we join the vertices $i^+$ and $j^-$ of $\graphic(\omega)$ with an edge.
    \item For every column of $\omega$ which contains only one $1$, situated in the $i$-th row of $\tau_1$ (resp. in the $j$-th row of $\tau_2$), we circle the corresponding vertex $i^+$ (resp. $j^-$) of $\graphic(\omega)$, and call it a \emph{marked vertex}.
    \item In particular, the vertex $i^+$ (resp. $j^-$) is neither marked nor incident with an edge (we will call this ``free point'') if and only if the corresponding row of $\omega$ contains only $0$'s.
\end{itemize}

For instance,
\begin{equation}
\omega=\mbox{\tiny $\begin{pmatrix} 0 & 0 & 0 & 0 \\ 1 & 0 & 0 & 0 \\ 0 & 0 & 0 & 0 \\ 0 & 1 & 0 & 0 \\ 0 & 0 & 1 & 0 \\ \hline 0 & 1 & 0 & 0 \\ 0 & 0 & 0 & 1 \\ 1 & 0 & 0 & 0 \end{pmatrix}$}
\;\;
\rightsquigarrow
\;\;
\graphic(\omega)=
\mbox{\tiny
$\begin{picture}(100,20)(0,0)
\put(0,11){$\bullet$}\put(20,11){$\bullet$}\put(40,11){$\bullet$}\put(60,11){$\bullet$}\put(80,11){$\bullet$}
\put(0,18){$1^+$}\put(20,18){$2^+$}\put(40,18){$3^+$}\put(60,18){$4^+$}\put(80,18){$5^+$}
\put(0,-11){$\bullet$}\put(20,-11){$\bullet$}\put(40,-11){$\bullet$}
\put(0,-20){$1^-$}\put(20,-20){$2^-$}\put(40,-20){$3^-$}
\put(21,13){\line(1,-1){21}}
\put(61,12){\line(-3,-1){60}}
\put(82,13){\circle{8}}
\put(22.5,-9){\circle{8}}
\end{picture}$}.
\label{formula:example.graph}
\end{equation}

(b)
Note that, according to the representation of the indecomposable objects of the category $\JQRep$ given in Figure \ref{F1}, 
the connected components of $\graphic(\omega)$ give rise to a collection of indecomposable objects which we denote $\Indecomposable(\graphic(\omega))$.  
For instance, for the above $\omega$, we have
$$
\Indecomposable(\graphic(\omega)) 
= \{ \mathcal{I}_5^+, \; \mathcal{I}_1'^+, \; \mathcal{I}_3'^+, \; {\mathcal{I}_2^-,} \; \mathcal{I}_{2,3}, \; \mathcal{I}_{4,1} \}.
$$
\end{notation}

The next lemma easily follows from the definitions.

\begin{lemma}\label{lemma:decomposition.omega}
In the category $\JQRep$, 
we have a multiplicity free decomposition 
\begin{equation*}
\mathcal{F}_\omega=\bigoplus\nolimits_{\mathcal{I}\in \Indecomposable(\graphic(\omega))}\mathcal{I}
\end{equation*}
by indecomposable representations.  
\end{lemma}

\begin{theorem}\label{theorem:orbits.AIII}
\begin{penumerate}
    \item The map $\parameters\to \dblFV/K$, $\omega\mapsto \Xorbit_\omega:=K\cdot \mathcal{F}_\omega$,
    is a bijection.
    \item We have 
    \begin{align*}
        \dim\Xorbit_\omega
        &= p^2+q^2- 
        \sum_{\text{\makebox[10ex][c]{\upshape $ x=i^\pm $: free point}}} i
        \\
        &
         - \sum_{\text{\makebox[8ex][c]{\upshape $ x $: marked point}}} \quad
         \#\left\{
         \text{\upshape
         \begin{tabular}{l}
         marked points on the same row
         \\
         on the left of $x$, including $x$
         \end{tabular}
         }\right\} \\
         & - \sum_{\text{\upshape $ \alpha $: edge}} \;\;
         \#\left\{\text{\upshape
         \begin{tabular}{l}
         marked points and edges 
         \\
         on the left of $\alpha$, including $\alpha$
         \end{tabular}
         }\right\}. 
    \end{align*}

    \item For all $\omega,\omega'\in\parameters$, we have
    $$ \Xorbit_\omega\subset\overline{\Xorbit_{\omega'}}
    \iff \langle\mathcal{I},\mathcal{F}_\omega\rangle\geq     \langle\mathcal{I},\mathcal{F}_{\omega'}\rangle
    \quad \text{for all indecomposable $\mathcal{I}\in\JQRep$}.
    $$
\end{penumerate}
\end{theorem}

\begin{proof}
(1) follows from \cite[Theorem 2.2\,(1)]{Fresse.N.2023}.

(2) follows from a combinatorial interpretation of 
$ \langle \mathcal{F}_{\omega}, \mathcal{F}_{\omega} \rangle $ 
by using Lemmas \ref{lemma:general.properties.orbits}\,(1), \ref{lemma:hom.indecomposables}, and \ref{lemma:decomposition.omega}.

(3) Following the notation of \cite[\S2.1]{Fresse.N.2023}, 
for $(i,j)\in\{0,1,\ldots,p\}\times\{0,1,\ldots,q\}$,
we denote by $r_{i,j}(\omega)$ the number of marked points or edges of the subgraph of $\graphic(\omega)$ formed by the vertices $k^+$, $\ell^-$ with $1\leq k\leq i$ and $1\leq \ell\leq j$.
Then, Lemma \ref{lemma:hom.indecomposables} yields the equalities
\begin{eqnarray*}    
 & \langle\I_i^+,\mathcal{F}_\omega\rangle=r_{i,0}(\omega),\quad 
\langle\I_i'^+,\mathcal{F}_\omega\rangle=i,\\
& \langle\I_j^-,\mathcal{F}_\omega\rangle=r_{0,j}(\omega),\quad
\langle\I_j'^+,\mathcal{F}_\omega\rangle=j,\quad
\langle\I_{i,j},\mathcal{F}_\omega\rangle=r_{i,j}(\omega).
\end{eqnarray*}
This observation, combined with \cite[Theorem 2.2\,(4)]{Fresse.N.2023}, yields
part (3) of the theorem.
\end{proof}

\begin{remark}
The obtained dimension formula in Theorem \ref{theorem:orbits.AIII}\,(2)
is much simpler than the one given in \cite[Theorem 2.2\,(2)]{Fresse.N.2023}.
\end{remark}

\part{\textbf{Steinberg theory for double flag varieties}}\label{part:Steinberg.theory}
%%\section*{\textbf{Part III. Steinberg theory for double flag varieties}}\partlabel{part:Steinberg.theory}
%%\addcontentsline{toc}{section}{\textbf{Part III. Steinberg theory for double flag varieties}\partstrut}

\section{Definition of Steinberg maps}\label{sec:Definition.Steinberg.maps}

\subsection{Symmetrized and exotic Steinberg maps for  double flag varieties}\label{subsec:two.Steinberg.maps.nil-projection.assumption}

Recall the following general situation:
\begin{itemize}
    \item $G$ is a connected reductive group endowed with an involution $\theta\in\Aut(G)$;
    \item $K:=G^\theta$ is the corresponding symmetric subgroup and we assume that it is connected;
    \item $P\subset G$ and $Q\subset K$ are parabolic subgroups, and we consider the double flag variety
    $$
    \dblFV:= K/Q \times G/P.
    $$
\end{itemize}
%%We denote by $\fg$ and $\fk$ the Lie algebras of $G$ and $K$, respectively. 
By abuse of notation, let $\theta:\fg\to\fg$ denote the differential of the involution $ \theta $, which is an involution of the Lie algebra.  
We also consider the Cartan decomposition 
\begin{equation*}
\fg=\fk\oplus\fs    , \quad
\text{ where \quad $\fs= 
%%\fg^{-\theta} = 
\{x\in\fg:\theta(x)=-x\}$.}
\end{equation*}
Let $\fg\to\fk$, $x\mapsto x^\theta$ and $\fg\to\fs$, $x\mapsto x^{-\theta}$ be the projections along this decomposition.

We apply the constructions of \S~\ref{subsec:moment.map.for.FlB} to the double flag variety $\dblFV$.
In particular, we consider $\dblFV$ as a $ K $-variety with diagonal $ K $-action. 
This gives rise to the moment map:
$$
\mu:T^*\dblFV\to\fk\cong\fk^*,\ ((x,\fq_1),(y,\fp_1))\mapsto x+y^\theta
$$
\skipover{
and
$$
\phi:T^*\dblFV\to\fg\cong\fg^*,\ ((x,\fq_1),(y,\fp_1))\mapsto y.
$$
}
and we get the conormal variety 
$$
\conormal:=\mu^{-1}(0)\cong\{(y,\fq_1,\fp_1)\in\fg\times K/Q\times G/P:y\in\fn_{\fp_1},\ y^\theta\in\fn_{\fq_1}\}.
$$
Here we recall that $\fn_{\fp_1}$ and $\fn_{\fq_1}$ respectively denote the nilradicals of the parabolic subalgebras $\fp_1\subset\fg$ and $\fq_1\subset\fk$.
By Lemma \ref{lemma:conormal}, $\conormal$ is the disjoint union of the conormal bundles
$$
\conormal = {\coprod_{\Xorbit \in \dblFV/K} T^*_\Xorbit\dblFV} .
$$
On the other hand, by using the moment map for the $ G $-action on the second factor, 
we get a map
$$
\phi: \conormal \to G\,\fn_{\fp} \subset \nilpotentsof{\fg}, \qquad  (y, \fq_1, \fp_1) \mapsto y,
$$
which is simply {the} projection to the first factor, and 
it is clearly $K$-equivariant.
By composing with the projections along the Cartan decomposition, this finally yields two $K$-equivariant maps
\begin{align*}   
&
\phi_\fk: \conormal \to K\,\fn_{\fq} \subset \nilpotentsof{\fk}, 
&&
(y, \fq_1, \fp_1) \mapsto y^\theta
\qquad\text{ and }\qquad
\\
&
\phi_\fs: \conormal \to \fs, 
&&
(y, \fq_1, \fp_1) \mapsto y^{-\theta} .
\end{align*}
The image of $ \phi_\fk $ is clearly contained in the nilpotent variety $ \nilpotentsof{\fk} $.  
For $\phi_\fs$, we also expect that the image of $\phi_\fs$ is contained in the nilpotent variety $ \nilpotentsof{\fs} $ as well, but it is not always the case.  
See Propositions~\ref{proposition:image:nilpotent.classical}, \ref{proposition:image:nilpotent.AIII} and Conjecture~\ref{conj:Im.exotic.mm.in.nilpotent.variety} below.  

From now on, we assume the following ``nil-projection assumption''.

\begin{assumption}\label{assumption:phis}
The image of the projection $ \phi_{\fs}: \conormal \to \lie{s} $ 
is contained in the nilpotent variety $ \nilpotentsof{\lie{s}} $.
\end{assumption}

Note that the nilpotent variety $\nilpotentsof{\fs}$ has {a finite number}  of $ K $-orbits due to \cite{Kostant.Rallis.1971}.

\medskip

For every orbit $\Xorbit\in\dblFV/K$, the conormal bundle $T^*_\Xorbit\dblFV$ is $K$-stable and irreducible. Hence its images by $\phi_{\fk}$ and $\phi_{\fs}$ are $K$-stable irreducible subsets of $\nilpotentsof{\fk}$ and $ \nilpotentsof{\fs} $ respectively.  
This means the closure of the image of $\Xorbit\in\dblFV/K$ by $\phi_{\fk}$ (respectively $\phi_{\fs}$) coincides with the closure of a single nilpotent $ K $-orbit in  $\nilpotentsof{\fk}$ (respectively $\nilpotentsof{\fs}$).  

\begin{definition}\label{def:definition.of.Steinberg.maps}
\begin{penumerate}
\item 
%%Recall that $\nilpotentsof{\fk}$ consists of finitely many $K$-orbits.
The map $\phi_{\fk}$ gives rise to a well-defined map
$$
\Phi_\fk:\dblFV/K\to\nilpotentsof{\fk}/K,\quad \Xorbit\mapsto \big(\text{the unique $K$-orbit dense in }\phi_\fk(T^*_\Xorbit\dblFV)\big),
$$
which we call the {\em symmetrized Steinberg map}.

\item 
%%For deriving a similar map from $\phi_{\fs}$, 
Similarly, 
under Assumption~\ref{assumption:phis}, 
we define the {\em exotic Steinberg map}, using $\phi_{\fs}$:
$$
\Phi_\fs:\dblFV/K\to\nilpotentsof{\fs}/K,\quad  \Xorbit\mapsto \big(\text{the unique $K$-orbit dense in }\phi_\fs(T^*_\Xorbit\dblFV)\big).
$$
\end{penumerate}
\end{definition}

\subsection{Assumption~\ref{assumption:phis} for the exotic Steinberg map}

To define the exotic Steinberg map, we need Assumption~\ref{assumption:phis} above.  
In fact, when the group $ G $ is classical and the parabolic subgroup $ P $ of $ G $ has abelian unipotent radical, this condition always holds.  
More generally, we have the following proposition.  

\begin{proposition}\label{proposition:image:nilpotent.classical}
Assumption~\ref{assumption:phis} holds
in the following situations.
\begin{penumerate}
\item\label{prop:P.abelian.unipotent.radical:1} 
$G=\GL_n$ and $P\subset G$ is a maximal
parabolic subgroup; 
\item\label{prop:P.abelian.unipotent.radical:2}
$G=\Sp_{2n}$ and $P\subset G$ is the stabilizer of a Lagrangian subspace or of a line;
\item\label{prop:P.abelian.unipotent.radical:3}
$G=\SO_{n}$ and $P\subset G$ is the stabilizer of an isotropic line or of a Lagrangian subspace (if $n$ is even), or of a maximal isotropic subspace if $n=5$.
\end{penumerate}
\medskip 
In particular, this implies that Assumption~\ref{assumption:phis} holds whenever $G$ is classical and $P$ has an abelian unipotent radical (see \cite{Richardson.Roehrle.1992}).
\end{proposition}

\begin{proposition}\label{proposition:image:nilpotent.AIII}
Let $(G,K)=(\GL_{p+q},\GL_p\times\GL_q) \; (p,q\geq 1) $ be a symmetric pair of type {\upshape AIII}. Then,
Assumption~\ref{assumption:phis} holds whenever we are in {one of the following cases that appear in Table \ref{table:AIII.finite.type}:}
{
\begin{penumerate}
\item 
{\upshape (PA-1), (PA-2), (PA-3), (P2)};
\item 
{\upshape(P3-1)},  
%%with $ q = 2; $ 
{\upshape(P3-3)} with $ q = 2 $;
\item 
{\upshape (P3-4)} with $Q_1$ maximal. 
\end{penumerate}
}

In particular, this implies that Assumption~\ref{assumption:phis} holds
whenever the double flag variety $\dblFV=G/P\times K/Q$ is of finite type and $P$ is a Borel subgroup of $G$.  
\end{proposition}

The propositions are shown case-by-case. In fact, they follow from the next technical lemma.
(It would be interesting to have a more conceptual proof of the result.)

\begin{lemma}\label{lemma:image:nilpotent}
\begin{penumerate}
    \item Let $ y \in \Mat_{p+q} $ be nilpotent and put 
\begin{equation*}
y=\begin{pmatrix} a \;&\; b \\ c \;&\; d \end{pmatrix}\in\nilpotentsof{\gl_{p+q}}, \qquad
v:=\begin{pmatrix} a \;&\; 0 \\ 0 \;&\; d \end{pmatrix}, \quad
w:=\begin{pmatrix} 0 \;&\; b \\ c \;&\; 0 \end{pmatrix}, 
\end{equation*}
with $(a,d)\in\gl_p\times \gl_q$, $b\in\Mat_{p,q}$, $c\in\Mat_{q,p}$.
    Assume that $ v $ is nilpotent and one of the following conditions {\upshape (A)--(G)}  occurs.
%%$v:=\begin{pmatrix} a & 0 \\ 0 & d \end{pmatrix}$ 
    Then $ w $ is also nilpotent.
%%$w:=\begin{pmatrix} 0 & b \\ c & 0 \end{pmatrix}$ 
%
    \begin{itemize}
        \item[{\makebox[3em][r]{\upshape(A)}}] $y^2=0$;
        \item[{\makebox[3em][r]{\upshape(B)}}] $y^3=0$ and {$\rk y \leq 2$};
        \item[{\makebox[3em][r]{\upshape(C)}}] $y^3=0$, $a^2=0$, and $d^2=0$;
        \item[{\makebox[3em][r]{\upshape(D)}}] $y^3=0$ and $q=2$;
        \item[{\makebox[3em][r]{\upshape(E)}}] $\rk a\leq 1$ and $d=0$;
        \item[{\makebox[3em][r]{\upshape(F)}}] $q=2$, $a^2=0$, and $d=0$;
        \item[{\makebox[3em][r]{\upshape(G)}}] $q=1$.
    \end{itemize}
    \item Let $y\in\gl_n$ and let $y^*$ be the adjoint of $y$ with respect to a given orthogonal or symplectic form on $\C^n$. 
    Assume that $v:=\frac{1}{2}(y-y^*) \smallvstrut$ is nilpotent and $y^2=0$.
    Then $w:=\frac{1}{2}(y+y^*) \smallvstrut$ is also nilpotent.
\end{penumerate}
\end{lemma}

\begin{proof}
(1) We calculate 
\begin{equation}\label{ysquare}
y^2=\begin{pmatrix} a^2 + b c \;&\; a b + b d \\ c a + d c \;&\; d^2 + c b \end{pmatrix}, 
\end{equation}
and get $ \Tr y^2 = \Tr a^2 + \Tr d^2 + 2 \Tr bc $.  
Since $y$, $a$, and $d$ are all nilpotent, we obtain 
\begin{equation}\label{trace.cb}
\Tr cb=0.
\end{equation}

(G) When $q=1$, \eqref{trace.cb} already implies that $cb=0$, thus $(bc)^2=b(cb)c=0$, and so $w$ is nilpotent.

(F) When $d=0$ and $a^2=0$, we see that
    $$
    y^4=\begin{pmatrix}
    (bc)^2+abca \;&\; bcab+abcb \\ cabc+cbca \;&\; (cb)^2
    \end{pmatrix}.
    $$
    Whence
    $
    0=\Tr y^4=\Tr (bc)^2+\Tr abca+\Tr (cb)^2=2\Tr (cb)^2
    $
because 
%%$\Tr (bc)^2=\Tr(cb)^2$ and 
$\Tr abca=\Tr bca^2=0$.
    If $q=2$, since $cb$ is a matrix of size $2$ such that $\Tr(cb)=\Tr(cb)^2=0$, we conclude that $(cb)^2=0$,
    so that {$(bc)^2b=b(cb)^2=0$} holds. Therefore, {$w^5=0$}.

(D) If $y^3=0$ and $d^2=0$ (which holds in particular if $q=2$), then we have
\begin{equation} \label{equations:y3=0}
0=y^3=\begin{pmatrix}
a^3 + a b c + b c a + b d c \;&\; a^2 b + a b d + b c b \\
c a^2 + c b c + d c a \;&\; c a b + c b d + d c b
\end{pmatrix}.
\end{equation}
Multiplying the upper right coefficient by $c$ on the left, we get
\begin{equation} \label{formula:cbcb}
(cb)^2 = - c a^2 b - c a b d.
\end{equation}
Multiplying the lower diagonal coefficient by $d$ on the right, we obtain 
\begin{equation} \label{formula:tracecabd}
\Tr (c a b d) = - \Tr (c b d^2) - \Tr (d c b d) = - \Tr (c b d^2) = 0, 
\end{equation}
since $d^2=0$. Finally, multiplying the upper diagonal coefficient of \eqref{equations:y3=0} by $a$ and using $ \Tr a^4 = 0 $, we get
\begin{equation} \label{formula:traceca2b}
2 \Tr (c a^2 b) = -\Tr (b d c a) = 0
\end{equation}
because of \eqref{formula:tracecabd}. Altogether, \eqref{formula:cbcb}, \eqref{formula:tracecabd}, and \eqref{formula:traceca2b} imply that
$\Tr(cb)^2=0$. When $q=2$, this fact combined with \eqref{trace.cb}
is sufficient for ensuring that $(cb)^2=0$,
{whence $(bc)^2b=0$. Therefore, $w^5=0$.}

(C) Here, in addition to $y^3=0$ and $d^2=0$, we assume that $a^2=0$. Hence we can replace $a^3$, $a^2b$, and $ca^2$ by zero in \eqref{equations:y3=0}.
Multiplying the lower left (resp. upper right) block of the above matrix by $b$ on the left (resp. by $c$ on the right), we obtain
$$
- b d c a = (bc)^2 = - a b d c,
$$
whence
$$
(bc)^4 = b d c a^2 b d c = 0
$$
(since $a^2=0$). 
Similarly it can be shown that $(cb)^4=0$, hence $w$ is nilpotent.

(A) If $y^2=0$, then from \eqref{ysquare} we have
\begin{equation*}
w^2=\begin{pmatrix} bc \;&\; 0 \\ 0 \;&\; cb \end{pmatrix}
=-\begin{pmatrix} a^2 \;&\; 0 \\ 0 \;&\; d^2 \end{pmatrix}=-v^2.    
\end{equation*}
Since $v$ is nilpotent, we conclude that $w$ is also nilpotent in this case.

(E)
Up to enlarging the size of the matrices by adding rows/columns of zeros, we can assume that $p=q\geq 2$.
By $E_{i,j}^{(p)}$ we will denote the elementary matrix of size $(p,p)$ with $1$ at the position $(i,j)$.
Also, we can assume that $\rk a=1$ and $d=0$ and, up to conjugation by $\GL_p\times \GL_p$, we can assume that
$a=E_{1,2}^{(p)}$.
Fix $\epsilon\in\{0,1\}$.
Then, we have
\begin{eqnarray*}
\det(\epsilon v+w-t \, \unitmatrix_{2p}) & = & 
\det\begin{pmatrix} \epsilon a - t \, \unitmatrix_p & b \\ c & - t \, \unitmatrix_p \end{pmatrix} = \det((\epsilon a - t \, \unitmatrix_p)(-t \, \unitmatrix_p) - b c) \\
 & = & \det(t^2 \, \unitmatrix_p - b c - \epsilon t \, E_{1,2}^{(p)}) = f(t^2)+\epsilon t g(t^2)
\end{eqnarray*}
with certain polynomials $f,g$. 
Since $t^{2p} = \det (v + w - t\, \unitmatrix_{2p}) 
= f(t^2) + t \, g(t^2) $ 
(because $y=v+w$ is nilpotent), we conclude that $g\equiv 0$, hence $\det(w-t\unitmatrix_{2p})=t^{2p}$, and so $w$ is nilpotent.

(B) 
Finally we assume that $y^3=0$ and $\rk y\leq 2$.
For this remaining case, we reason by induction on $p+q$, with immediate initialization if $p=0$ or $q=0$. 
We have $2\geq \rk y\geq \max\{\rk a,\rk d\}$.
Since we have already dealt with (A) and (C), 
we can assume that $\rk y=2$ and one of $a^2$, $d^2$ is nonzero, say $a^2\not=0$, and so $\rk a=2$.
Then up to conjugating under $\GL_p\times\GL_q$, we can assume that
$$
a=E_{1,2}^{(p)}+E_{2,3}^{(p)}.
$$
Since $\rk y=2$, this implies in particular that the first column of $y$ is zero.
Let $y',v',w'$ be the submatrices of $y,v,w$ obtained by removing the first row and column.
We can apply the induction hypothesis to $y'$, and
we get that $w'$ is nilpotent. Since the first column of $w$ (like the one of $y$) is zero, we conclude that $w$ is nilpotent.

(2) We have
$$
0=y^2=(v+w)^2=(v^2+w^2)+(vw+wv).
$$
Since $v^*=-v$ and $w^*=w$,
we get that $(v^2+w^2)^*=v^2+w^2$
and $(vw+wv)^*=-(vw+wv)$.
Due to the uniqueness of the decomposition of a matrix as sum of self-adjoint and skew-adjoint elements, we derive that $v^2+w^2=vw+wv=0$. In particular, knowing that $v$ is nilpotent, we conclude that $w$ is nilpotent.
\end{proof}

\begin{example}\label{example:imagenotins}
In general, Assumption~\ref{assumption:phis} will be valid
if and only if the implication
\begin{equation}\label{condition:phifs:implication}
(y \in G \fn_{\fp} \quad \text{and} \quad y^\theta \in K \fn_{\fq}) \quad \Longrightarrow \quad y^{-\theta} \in \nilpotentsof{\fs}
\end{equation}
holds.
When $P$ and $Q$ are Borel subgroups, 
the left-hand side of \eqref{condition:phifs:implication} just means that $y$ and $y^\theta$ are nilpotent,
and the implication of \eqref{condition:phifs:implication} is \emph{not true} in general.
Consider for instance the symmetric pair $(G,K)=(\GL_6,\GL_3\times\GL_3)$ of type AIII, and
$$
y=\begin{pmatrix}
    a \;&\; \unitmatrix_3 \\ b \;& -a
\end{pmatrix}
\quad\text{with}\quad a=\begin{pmatrix}
    0 \;&\; 1 \;& 0 \\ 0 \;&\; 0 \;& -1 \\ 0 \;&\; 0 \;& 0
\end{pmatrix}\quad\text{and}\quad b=\begin{pmatrix}
    0 \;&\; 0 \;&\; 1 \\ 0 \;&\; 0 \;&\; 0 \\ 1 \;&\; 0 \;&\; 0
\end{pmatrix}.
$$
Then we have $y^4=0$, $(y^\theta)^3=0$, but $y^{-\theta}\notin\nilpotentsof{\fs}$.
However, note that the corresponding double flag variety is \emph{not} of finite type.
\skipover{
\begin{penumerate}
\item
Also in the case of the symmetric pair $(G,K)=(\GL_n\times \GL_n,\GL_n)$, condition \eqref{condition:phifs:implication} does not hold in general. Take for instance $ n = 3 $, 
and consider 
\begin{eqnarray*}
y & = & \left(\begin{pmatrix}
    0 & 1 & 1 \\ 0 & 0 & 1 \\ 0 & 0 & 0
\end{pmatrix},\begin{pmatrix}
    0 & -1 & 0 \\ 0 & 0 & 0 \\ 1 & -1 & 0
\end{pmatrix}\right) \\
 & = & \underbrace{\frac{1}{2}\left(\begin{pmatrix}
    0 & 0 & 1 \\ 0 & 0 & 1 \\ 1 & -1 & 0
\end{pmatrix},\begin{pmatrix}
    0 & 0 & 1 \\ 0 & 0 & 1 \\ 1 & -1 & 0
\end{pmatrix}\right)}_{y^\theta}+
\underbrace{\frac{1}{2}\left(\begin{pmatrix}
    0 & 2 & 1 \\ 0 & 0 & 1 \\ -1 & 1 & 0
\end{pmatrix},\begin{pmatrix}
    0 & -2 & -1 \\ 0 & 0 & -1 \\ 1 & -1 & 0
\end{pmatrix}\right)}_{y^{-\theta}}.
\end{eqnarray*}
Then, $y$ and $y^\theta$ are nilpotent, unlike $y^{-\theta}$.
Note that this case corresponds to the triple flag variety 
$ (\GL_3/B_3^+)^3 $ which is \emph{not} of finite type again.  
\end{penumerate}
}
\end{example}

We do not have a counterexample where 
a double flag variety is \emph{of finite type} 
and the image of the projection $ \phi_{\fs}: \conormal \to \lie{s} $ is not contained in the nilpotent variety.  
Thus, based on the cases already solved by Propositions \ref{proposition:image:nilpotent.classical} and \ref{proposition:image:nilpotent.AIII}, we propose a conjecture here.

\begin{conjecture}[{Nil-projection Conjecture}]\label{conj:Im.exotic.mm.in.nilpotent.variety}
If a double flag variety $ \dblFV $ is of finite type, 
then the image of the projection $ \phi_{\fs}: \conormal \to \lie{s} $ 
is contained in the nilpotent variety $ \nilpotentsof{\lie{s}} $.
\end{conjecture}

\section{Generalization of Robinson-Schensted correspondence}\label{section:generalization.RS}

We keep the setting and notation of the previous section.  
Let us assume 
\begin{equation}\label{condition:finite.type}
    \text{$\Xfv$ has finitely many $K$-orbits, that is, $\Xfv$ is of finite type}, 
\end{equation}
and we further assume Assumption~\ref{assumption:phis}, i.e., the image of the projection $ \phi_{\fs} $ is contained in the nilpotent variety so that we have  
$ \phi_{\fs}: \conormal \to \nilpotentsof{\lie{s}} $.

%Note that this is a favorable situation for dealing with the maps $\Phi_\fk$ and $\Phi_\fs$ is when the following condition is satisfied:
Because $ \Xfv $ is of finite type, as already pointed out in \S~\ref{subsec:moment.map.for.FlB},
the conormal variety $\conormal$ is equidimensional
and the closures of the conormal bundles $T^*_\Xorbit\Xfv$,
for $\Xorbit\in\Xfv/K$, are the irreducible components of $\conormal$.
Moreover, in this case, the $K$-orbits in $\Xfv$ can be described combinatorially,
whereas nilpotent orbits have a well-known combinatorial description (see for example, \cite{Collingwood.McGovern.1993}), hence
the maps $\Phi_\fk$ and $\Phi_\fs$ can be explicitly described by the combinatorial data.
%(By the way, we know no example where \eqref{assumption:phis} does not hold when \eqref{condition:finite.type} is fulfilled.) 

In the previous section, we have extended the setting of the classical work of Steinberg \cite{Steinberg.1976}, which deals with the standard double flag variety $ \dblFV=G/B\times G/B $ on which $ G $ acts diagonally. This standard theory yields the classical {\em Steinberg map} $ \Phi:W\to\nilpotentsof{\fg}/G $, also considered in \cite{Hinich.Joseph.2005}.
In the next subsections, first we explain how classical Steinberg theory can be obtained as a particular case of the construction given above
and recall from \cite{Steinberg.1976,Steinberg.1988} the calculation of $\Phi$ in the case where $G=K=\GL_n$.
Then we calculate $\Phi_\fk$ and $\Phi_\fs$ for the symmetric pair $(G,K)=(\GL_n,\GL_p\times \GL_q)$ of type AIII 
and the double flag variety $\Xfv$ considered in \S~\ref{section-7}.

\subsection{Classical Steinberg map and Robinson--Schensted correspondence}

Let us consider 
$ \bbG = G \times G $ 
and define $ \theta \in \Aut \bbG $ as 
$ \theta(g_1, g_2) = (g_2, g_1) \; ((g_1, g_2 ) \in \bbG) $, which is the ``flip'' of two components.  
Put $ \bbK = \bbG^{\theta} = \Delta G $, which is the diagonal embedding $ G \overset{\Delta}{\hookrightarrow} \bbG $.  
We identify $ \bbK $ and $ G $ sometimes.  
Now consider the parabolic subgroup $ \bbQ = \bbK $ and a Borel subgroup $ \bbB = B \times B \subset G \times G = \bbG $.  
Then the double flag variety associated with these data is 
\begin{equation*}
\dblFV = \bbK/\bbQ \times \bbG/\bbB = G/B \times G/B ,
\end{equation*}
on which $ G (=\bbK) $ acts diagonally.  

\skipover{
In this section, we assume that $\theta=\mathrm{id}_G:G\to G$, so that $(G,K)=(G,G)$ is a trivial symmetric pair.
Moreover, let $P=Q=B\supset T$, a standard Borel subgroup of $G$, so that $\Xfv=G/B\times G/B$ is a full double flag variety.
}

By virtue of Bruhat decomposition, $\Xfv$ has finitely many $G$-orbits ($ G = \bbK $ as noted above), parametrized by the elements of the Weyl group $ W = N_G(T)/T $, 
where $ T $ is a maximal torus in $ B $.  So let us write the decomposition as 
$$
\Xfv=G/B\times G/B=\coprod_{w\in W}\Xorbit_w
\quad\text{where }\quad
\Xorbit_w=G\cdot(B,wB).
$$
Note that the conormal direction at $ (B, wB) \in \dblFV $ is 
{\begin{equation*}
%\bigl((\Delta \lie{g}) + (\lie{g}/\fb, \lie{g}/{}^w\fb) \bigr)
\{(x+\fb,x+{}^w\fb):x\in \lie{g}\}^{\bot}  = \fn\cap {}^w\fn ,
\end{equation*}}
where $\fn\subset\fb=\mathrm{Lie}(B)$ stands for the nilpotent radical of our Borel subalgebra,
and the notation ${}^w\fb$ or ${}^w\fn$ stands for the conjugate by $ w \in W $.
Therefore the conormal bundles of orbits are given by
$$
T^*_{\Xorbit_w}\Xfv=G\cdot \{(\fb,{}^w\fb,x):x\in\fn\cap {}^w\fn\}.
$$
\indent
In the present situation, the nilpotent varieties are 
\begin{equation*}
{
\nilpotentsof{\fk} = \nilpotents^{\theta} = \{ (x, x) : x \in \nilpotentsof{\lie{g}} \} 
\;\;
\text{ and }
\;\;
\nilpotentsof{\fs} = \nilpotents^{-\theta} = \{ (x, -x) : x \in \nilpotentsof{\lie{g}} \} 
}
\end{equation*}
and both are identified with $ \nilpotentsof{\lie{g}} $.  
Since we have chosen $ \bbQ = \bbK $, the symmetrized Steinberg map $ \Phi_\fk $ is the trivial map which assigns every orbit $ \Xorbit_w $ to the trivial orbit, and 
Assumption~\ref{assumption:phis} is trivially satisfied.  
%%fulfilled, as we have $\fs=0$. Hence both Steinberg maps $\Phi_\fk$ and $\Phi_\fs$ are well defined.
Then we focus on
%%Actually we disregard the exotic map $\Phi_\fs$, which is the trivial map, and 
the (exotic) Steinberg map 
$$
\Phi=\Phi_\fs:\Xfv/G\cong W\to\nilpotentsof{\fg}/G
$$
which is explicitly given as 
$$
\Phi: w \mapsto (\mbox{the unique nilpotent $G$-orbit $\mathcal{O}$ which intersects $\fn\cap{}^w\fn$ densely}).
$$

Let us further assume $G=\GL_n$.  In this case, the Steinberg map $\Phi$ can be computed explicitly, in terms of the \emph{Robinson-Schensted correspondence} (RS correspondence) 
$$
\RS:\permutationsof{n}
\stackrel{\sim}{\longrightarrow}
\coprod\nolimits_{\lambda\vdash n}\ST{\lambda}\times\ST{\lambda},\ w\mapsto (\RS_1(w),\RS_2(w)).
$$
Here, the union on the right-hand side is over partitions of $n$ (equivalently, Young diagrams of size $n$) and $\ST{\lambda}$ denotes the set of 
\emph{standard Young tableaux} of shape $\lambda$ (i.e., the boxes of $ \lambda $ are filled with integers $ 1, 2, \dots, n $ and 
the filling is increasing in both directions, horizontal and vertical).  
Let us recall the definition and some properties of the RS correspondence briefly (see \cite{Fulton} for more details).

\begin{definition}
\begin{penumerate}
\item If $T$ is a Young tableau and $a$ an integer, let $T\leftarrow a$ be the tableau obtained by inserting $a$ into $T$ as follows:
\begin{itemize}
    \item If $a$ is greater than or equal to every entry in the first row of $T$, then we insert $a$ at the end of the first row of $T$.
    \item Otherwise, let $b$ be the biggest entry of the first row which is $>a$; we substitute $b$ by $a$, and insert $b$ into the subtableau formed by the remaining rows of $T$, according to the same rule.
\end{itemize}
\item If $a_1,\ldots,a_k$ is a list of integers, let $\Rowinsert(a_1,\ldots,a_k)$ be the tableau obtained by inserting successively $a_1,\ldots,a_k$ according to the rule described above, that is,
$$
\Rowinsert(a_1,\ldots,a_k)=(\emptyset\leftarrow a_1\leftarrow\cdots\leftarrow a_k).
$$
\item 
The tableau $\RS_1(w)$ is defined by
$$
\RS_1(w)=\Rowinsert(w_1,\ldots,w_n)
$$
while we define $\RS_2(w)$ as the unique tableau
such that the shape of the subtableau of $\RS_2(w)$ of entries $1,\ldots,k$ coincides with the shape of $\Rowinsert(w_1,\ldots,w_k)$.
\end{penumerate}
\end{definition}

For later use, we mention some properties of the Robinson--Schensted correspondence:

\begin{proposition}[{See, e.g., \cite{Fulton,vanLeeuwen.1996}}]\label{proposition:classical.RS}
\begin{penumerate}
    \item The map 
    $$
    \RS=(\RS_1,\RS_2):\permutationsof{n}\to \coprod\nolimits_{\lambda\vdash n}\ST{\lambda}\times \ST{\lambda}
    $$ 
    so obtained is a bijection.
    \item If $w\in \permutationsof{n}$, then $(\RS_1(w^{-1}),\RS_2(w^{-1}))=(\RS_2(w),\RS_1(w))$.
    \item If $w_0:k\mapsto n-k+1$ is the longest element of $\permutationsof{n}$, then $\RS_1(ww_0)={}^t\RS_1(w)$, 
    where $ \transpose{T} $ denotes the transposition of a tableau $ T $, and
    $\RS_1(w_0w)= \transpose{\mathbf{S}(\RS_1(w))} $,
    where $\mathbf{S}$ stands for the \emph{Sch\"utzenberger involution} {\upshape(\emph{see} \cite[\S 2]{vanLeeuwen.1996})}.
\end{penumerate}
\end{proposition}

\begin{remark}
    More generally, if $\sigma:\{j_1<\ldots<j_n\}\to \{i_1<\ldots<i_n\}$ is a bijection, then we define
    $$    \RS_1(\sigma)=\Rowinsert(\sigma(j_1),\ldots,\sigma(j_n)),
    $$
    which is a Young tableau with entries $i_1,\ldots,i_n$, and 
    we can define
    $\RS_2(\sigma)$ as the unique Young tableau of entries $j_1,\ldots,j_n$ such that
    $$
    \shape \RS_2(\sigma)|_{\{j_1,\ldots,j_k\}}
    =
    \shape \RS_1(\sigma|_{\{j_1,\ldots,j_k\}})\quad\text{for all $k=1,\ldots,n$},
    $$
    where the notation on the left hand side stands for the shape of the subtableau of entries {$j_1,\ldots,j_k$}.
    Moreover, from Proposition {\ref{proposition:classical.RS}\,(2)} we easily have $\RS_2(\sigma)=\RS_1(\sigma^{-1})$.
\end{remark}

Finally, we recall that nilpotent orbits of $\fg=\gl_n$ are parametrized by the set $\partitionsof{n}=\{\lambda\vdash n\}$ of partitions of $n$ through Jordan normal form:
$$
\nilpotentsof{\gl_n}=\coprod\nolimits_{\lambda\vdash n}\mathcal{O}_\lambda.
$$
Let us denote the shape of a tableau $ T $ by $ \shape T $, which is also interpreted as a partition (or a Young diagram).

\begin{theorem}[{Steinberg, \cite{Steinberg.1976,Steinberg.1988}}]\label{theorem:Steinberg}
For $G=\GL_n$ and $\Xfv=G/B\times G/B$,
the Steinberg map $\Phi: \Xfv/G\cong W\to\nilpotentsof{\fg}/G$ is given by, for $ w \in W $, 
$$
\Phi:\Xorbit_w\mapsto \mathcal{O}_\lambda\quad\mbox{such that $\lambda=\shape(\RS_1(w))=\shape(\RS_2(w))$}.
$$
\end{theorem}

\subsection{Description of Steinberg maps for a double flag variety of type AIII}

In this section, we use the same notation as in the beginning of \S~\ref{section-7}.  
\skipover{
we consider the same situation as in \S~\ref{subsec:orbit.by.quivers.typeAIII}, namely
\begin{itemize}
\item $n=p+q$, $r\leq n$ are positive integers;
\item $V=\C^n=V^+\oplus V^-$ where $V^+=\C^p\times \{0\}^q$, $V^-=\{0\}^p\times \C^q$, with standard bases $(e_1,\ldots,e_n)=(e_1^+,\ldots,e_p^+,e_1^-,\ldots,e_q^-)$;
\smallskip
\item $G=\GL(V)=\GL_n$, $K=\GL(V^+)\times \GL(V^-)=\GL_p\times\GL_q$, $ P = \Stab_G(\C^r\times\{0\}^{n-r})$, $B_K=B_p^+\times B_q^+$ where $B_p^+\subset\GL_p$, $B_q^+\subset \GL_q$ are subgroups of upper-triangular matrices.
\end{itemize}
}
Thus our double flag variety is
\begin{align*}
\dblFV &= K/B_K \times G/P =  \bigl( \GL_p/ B_p^+ \times \GL_q/ B_q^+ \bigr) \times \GL_n/P_{(r, n - r)} 
\\
&\simeq \Bigl( \Flags(V^+)\times\Flags(V^-) \Bigr) \times \Grass_r(V).
\end{align*}
The purpose of this section is to give a combinatorial description of the Steinberg maps
$$
\Phi_\fk:\Xfv/K \to \nilpotentsof{\fk}/K
\quad\text{and}\quad
\Phi_\fs:\Xfv/K \to \nilpotentsof{\fs}/K.
$$
To this end, we first need to review the combinatorial parametrization of the various orbit sets that appear in the picture.

\subsubsection{Parametrization of $\nilpotentsof{\fk}/K$}
In the present case, the Cartan decomposition is
$\fg=\fk\oplus\fs$
with
\begin{align*}
\fk&=\left\{\begin{pmatrix}
a \;&\; 0 \\ 0 \;&\; d
\end{pmatrix}:a\in\gl_p,\ d\in\gl_q\right\},\quad 
\\
\fs&=\left\{\begin{pmatrix}
0 \;&\; b \\ c \;&\; 0
\end{pmatrix}:b\in \Mat_{p,q},\ c\in\Mat_{q,p}\right\}.
\end{align*}

    We have natural bijections
    $\nilpotentsof{\fk}\cong \nilpotentsof{\gl_p}\times \nilpotentsof{\gl_q}$
    and so 
    $\nilpotentsof{\fk}/K\cong \nilpotentsof{\gl_p}/\GL_p\times \nilpotentsof{\gl_q}/\GL_q\cong 
    \partitionsof{p}\times \partitionsof{q}$
    where the last identification is through Jordan normal form.

    \subsubsection{Parametrization of $\nilpotentsof{\fs}/K$}

    Nilpotent orbits of $\fs$ are parametrized by a set of signed Young diagrams:

    \begin{definition}
        \label{def:SYD}
{\rm (a)} A {\em signed Young diagram} (of signature $(p,q)$)
is a Young diagram of size $p+q$ whose boxes are filled in with $p$ symbols $+$ and $q$ symbols $-$
so that:
\begin{itemize}
\item two consecutive boxes of the same row have opposite signs, so that each row is a sequence of alternating signs;
\item we identify two such fillings up to permutation of rows, in particular 
we can standardize the filling in such a way that,
among rows which have the same length, the rows starting with a $+$ are above those starting with a $-$ (if there is any).
\end{itemize}

Let $\signedyd{p,q}$ denote the set of signed Young diagrams of signature $(p,q)$.
For $\Lambda\in\signedyd{p,q}$, the shape of $\Lambda$ is an element of $\partitionsof{p+q}$, also denoted by $\shape(\Lambda)$.
{By $\#\Lambda_{\leq k}(\pm)$ we will denote the number of $\pm$'s within the first $k$ columns of $\Lambda$.}
    \end{definition}

\begin{proposition}[{see, e.g., \cite{Collingwood.McGovern.1993,Ohta.1991.TMJ}}]
We have a bijection
$$
\signedyd{p,q}\to \nilpotentsof{\fs}/K,\ \Lambda\mapsto \{x\in\nilpotentsof{\fs}:
\forall k\geq 1,\ \dim \ker x^k\cap V^\pm=\#\Lambda_{\leq k}(\pm)\}.
$$
\end{proposition}

\subsubsection{Parametrization of $\dblFV/K$ by partial permutations} \label{section:combi.quotient}
    
We recall from \S~\ref{subsec:orbit.by.quivers.typeAIII} that the $K$-orbits of $\dblFV$ are parametrized by the set
$$
\parameters=\left\{
\omega=\begin{pmatrix}\tau_1 \\ \tau_2\end{pmatrix}
:\tau_1\in
\ppermutations_{p,r},\ 
\tau_2\in \ppermutations_{q,r},\ \rk \omega=r
\right\}/\permutationsof{r}.
$$
In fact, we have a bijection
$$
\parameters\stackrel{\sim}{\to} \dblFV/K,\
\omega\mapsto K(\flag^+,\flag^-,[\omega])
$$
where $[\omega]=\mathrm{Im}\,\omega$ and
$\flag^+,\flag^-$ are the standard flags of $V^+,V^-$, respectively
(see Theorem \ref{theorem:orbits.AIII}).
Note also that we have a graphical expression
$\graphic(\omega)$
of the elements $\omega\in\parameters$
(see Notation \ref{notation:graph}).  

\subsubsection{Combinatorial Steinberg maps}

Now, according to the parametrization of orbits  from the previous subsections, 
the symmetrized and exotic Steinberg maps $\Phi_\fk$ and $\Phi_\fs$ can be considered as combinatorial maps
\begin{equation}
\Phi_\fk:\parameters\to \{(\lambda,\mu):\lambda\vdash p,\ \mu\vdash q\}
\quad\text{and}\quad 
\Phi_\fs:\parameters\to \signedyd{p,q}.    
\end{equation}
By using these maps, 
we will finally reformulate another parametrization of $ \dblFV/K $ in terms of standard tableaux in Theorem~\ref{theorem:fibers.Phik}, 
which clarifies the irreducible components of the conormal fibers of Steinberg maps. 

Let us describe these maps $\Phi_\fk$ and $\Phi_\fs$ explicitly. 

\subsection{Generalized Robinson-Schensted correspondence related to $\Phi_\fk$} \label{section:generalized-RS}

The following theorem gives a complete algorithm to calculate the map $ \Phi_\fk:\parameters\to \{(\lambda,\mu):\lambda\vdash p,\ \mu\vdash q\} $.

\begin{theorem} \label{theorem:description.Phik}
Let $\omega\in\parameters$. We use the following notation, relative to the graph $\graphic(\omega)$:
\begin{itemize}
    \item Let $i_1<\ldots<i_k$, $\ell_1<\ldots<\ell_s$, and {$\ell'_1<\ldots<\ell'_{s'}$} be respectively the numbers of the end points, marked points, and free points, within the line of positive vertices of~$\graphic(\omega)$.
    Let $j_1<\ldots<j_k$, $m_1<\ldots<m_t$, and {$m'_1<\ldots<m'_{t'}$} be similarly the numbers of the end points, marked points, and free points, within the negative vertices of $\graphic(\omega)$.
    \item 
    Let $L$, $L'$, $M$, $M'$ be the vertical Young tableaux with entries 
    $\{\ell_i\}_{i=1}^s$, $\{\ell'_i\}_{i=1}^{s'}$, $\{m_i\}_{i=1}^t$, $\{m'_i\}_{i=1}^{t'}$, respectively.
    Let $\sigma:\{j_1,\ldots,j_k\}\to\{i_1,\ldots,i_k\}$
    be the bijection such that $\sigma(j)=i$ whenever $(i^+,j^-)$ is an edge of $\graphic(\omega)$.
\end{itemize}
Then, the image of $\omega$ by the symmetrized Steinberg map is the pair of partitions 
\begin{equation}
\Phi_\fk(\omega)=(\lambda,\mu)=\big(\shape( L\cdot \RS_1(\sigma)\cdot L')\,,\,\shape(M\cdot \RS_2(\sigma)\cdot M')\big)
\end{equation}
where $\cdot$ is the sliding operation of the plactic monoid (see \cite[\S1.2--2.1]{Fulton}).
\end{theorem}

\begin{example}
\label{example:Phik}
For $\omega$ as in (\ref{formula:example.graph}), we obtain
$$
L\cdot \RS_1(\sigma)\cdot L'=\mbox{\scriptsize $\young(5)$}\cdot\mbox{\scriptsize $\young(2,4)$}\cdot\mbox{\scriptsize $\young(1,3)$}=\mbox{\scriptsize $\young(13,2,4,5)$},\quad 
M\cdot \RS_2(\sigma)\cdot M'=\mbox{\scriptsize $\young(2)$}\cdot\mbox{\scriptsize $\young(1,3)$}=\mbox{\scriptsize $\young(13,2)$},
$$
hence
$$
\Phi_\fk(\omega)=(\lambda,\mu)\quad\text{with}\quad \lambda=\mbox{\scriptsize $\yng(2,1,1,1)$},\ \mu=\mbox{\scriptsize $\yng(2,1)$}.
$$
\end{example}

\begin{remark} \label{remark:description.Phik}
The properties of the row insertion algorithm also yield the following rephrasing of the formula giving $\Phi_\fk(\omega)=(\lambda,\mu)$ in Theorem \ref{theorem:description.Phik}.
Specifically, if we introduce the permutations
\begin{eqnarray*}
 & w_{\fk,+}=\begin{pmatrix} 1 & \cdots & s & s+1 & \cdots & s+k & s+k+1 & \cdots & p \\
\ell_s & \cdots & \ell_1 & \sigma(j_1) & \cdots & \sigma(j_k) & \ell'_{s'} & \cdots & \ell'_1\end{pmatrix}\in\permutationsof{p}, \\
 & w_{\fk,-}=\begin{pmatrix} 1 & \cdots & t & t+1 & \cdots & t+k & t+k+1 & \cdots & q \\
m_t & \cdots & m_1 & \sigma^{-1}(i_1) & \cdots & \sigma^{-1}(i_k) & m'_{t'} & \cdots & m'_1\end{pmatrix}\in\permutationsof{q}, 
\end{eqnarray*}
then we obtain
$$
(\lambda,\mu)=\big(
\shape(\RS_1(w_{\fk,+})),\shape(\RS_1(w_{\fk,-}))\big).
$$
\end{remark}

\begin{proof}[Sketch of proof of Theorem \ref{theorem:description.Phik}]
According to {Definition \ref{def:definition.of.Steinberg.maps}}, we need to consider the conormal bundle
$$
T_{\Xorbit_\omega}^*\dblFV=K\cdot \{(\flag_p^+,\flag_q^+,[\omega],y):y\in\mathcal{D}_\omega\}
$$
with the conormal direction
$$
\mathcal{D}_\omega=\{y\in \fn_\omega: y^\theta\in \fn_p^+\times \fn_q^+\}.
$$
Here $\fn_\omega=\{y:\mathrm{Im}\,y\subset [\omega]\subset \ker y\}$ is the nilradical of the parabolic subalgebra determined by $[\omega]$ and $\fn_p^+,\fn_q^+$ stand for the subalgebras of strictly upper triangular matrices of $\gl_p$ and $\gl_q$, respectively. 

For $a\in\gl_p$, it can be shown that
$$
\exists b,c,d\ \text{such that}\ \begin{pmatrix} a & b \\ c & d \end{pmatrix} 
\in \mathcal{D}_\omega \quad \Longleftrightarrow \quad a \in \fn_p^+ \cap \Ad  (w_{\fk,+}) \fn_p^+
$$
where $w_{\fk,+}$ is as in Remark \ref{remark:description.Phik} above.
A reformulation of
Steinberg's Theorem \ref{theorem:Steinberg} gives that the unique nilpotent $\GL_p$-orbit which intersects the linear space
$\fn_p^+ \cap \Ad(w_{\fk,+}) \fn_p^+$
is the one corresponding to the partition
$$
\shape \RS_1(w_{\fk,+})=\shape \RS_2(w_{\fk,+}).
$$
Arguing similarly with the block $d$
and the permutation $w_{\fk,-}$, we finally get that the product of nilpotent orbits that intersects the projection $\phi_\fk(T_{\Xorbit_\omega}^*\dblFV)=K\cdot\{y^\theta:y\in\mathcal{D}_\omega\}$ densely is $\mathcal{O}_\lambda\times\mathcal{O}_\mu$
with
$$
(\lambda,\mu)=(\shape \RS_1(w_{\fk,+}),\shape \RS_1(w_{\fk,-})).
$$
This (combined with Remark \ref{remark:description.Phik}) yields the claimed assertion.
See \cite[\S\ 4.3]{Fresse.N.2023} for a complete proof.
\end{proof}

As a byproduct of Theorem \ref{theorem:description.Phik}, we can determine the fibers of $\Phi_\fk$ in terms of a combinatorial correspondence which extends the Robinson--Schensted correspondence.
It involves the following sets of $5$-tuples:
for $(\lambda,\mu)\in\partitionsof{p}\times\partitionsof{q}$, let 
\begin{equation}
\Tquintuples{\lambda,\mu}=\{(T_1,T_2;\lambda',\mu';\nu) \mid 
T_1 \in \ST{\lambda}, \; T_2 \in \ST{\mu}, 
\text{satisfying}\ (\star) \}
\end{equation}
where
\begin{itemize}
%    \item[$(\star)$] $T_1$ and $T_2$ are standard Young tableaux of shapes $\lambda$ and $\mu$, respectively;
    \item[$(\star)$] $\nu\columnstrip\lambda'\columnstrip\lambda$,
$\nu\columnstrip\mu'\columnstrip\mu$, and $|\lambda'|+|\mu'|=|\nu|+r$.
\end{itemize} 
Hereafter, $|\lambda|$ stands for the size of a partition $\lambda$, and the notation
$\lambda'\columnstrip \lambda$ means that $\lambda'$ is a Young subdiagram of $\lambda$ such that the skew diagram $\lambda\setminus\lambda'$ is {\em column strip} (i.e., it contains at most one box in each row).

\begin{theorem}\label{theorem:fibers.Phik}
We use the notation of Theorem \ref{theorem:description.Phik} and Remark \ref{remark:description.Phik}.
We have a commutative diagram
\[
\xymatrix @R 10ex {
\parameters
\ar[dr]_{\Phi_\fk} \ar@{->}[rr]^-{\sim}_-{\genRS} & & \coprod\limits_{(\lambda, \mu) \in \partitionsof{p} \times \partitionsof{q}}
\Tquintuples{\lambda,\mu}
\ar[dl]^{\pi} \\
& \partitionsof{p} \times \partitionsof{q}}
\]
where $\pi$ is the natural {map}
and $\genRS$
is a bijection defined by
\begin{eqnarray}
\genRS(\omega)&=&\big(L\cdot \RS_1(\sigma)\cdot L',M\cdot\RS_2(\sigma)\cdot M'; \nonumber \\
&&\shape(L\cdot \RS_1(\sigma)),\shape(M\cdot \RS_2(\sigma));\ \shape(\RS_1(\sigma))\big). \label{formula:theorem.genRS}
\end{eqnarray}
In particular, $\genRS$ restricts to a bijection
$$
\Phi_\fk^{-1}(\lambda,\mu)\stackrel{\sim}{\longrightarrow} \Tquintuples{\lambda,\mu}
$$
for all $(\lambda,\mu)\in\partitionsof{p}\times\partitionsof{q}$.
\end{theorem}

\begin{proof}
We only need to justify that the map $\genRS$ is a bijection; the rest of the statement will then follow from Theorem \ref{theorem:description.Phik}.

First, we note that the considered map is well defined. Let $(T_1,T_2;\lambda',\mu';\nu)$ denote the right-hand side of \eqref{formula:theorem.genRS}.
The fact that $\lambda\setminus\lambda'$, $\lambda'\setminus\nu$, $\mu\setminus\mu'$, and $\mu'\setminus\nu$ are column strips follows from \cite[Proposition in \S1.1]{Fulton},
and by construction we have
$|\lambda'|+|\mu'|=2k+s+t=|\nu|+r$.

For showing the bijectivity, let $(T_1,T_2;\lambda',\mu';\nu)\in\mathcal{T}_{\lambda,\mu}$.
Two applications of \cite[Proposition in \S1.1]{Fulton} yield a unique 6-tuple
$(S_1,S_2,L,L',M,M')$, where $S_1,S_2$ are Young tableaux of shape $\nu$
and $L,L',M,M'$ are vertical Young tableaux, all contents being disjoint, such that $T_1=L\cdot S_1\cdot L'$, $T_2=M\cdot S_2\cdot M'$,
$\lambda'=\shape(L\cdot S_1)$, and $\mu'=\shape(M\cdot S_2)$.
Letting $\sigma:J\to I$ be the bijection associated with the pair $(S_1,S_2)$ via Robinson--Schensted correspondence,
the data $I,L,L',J,M,M',\sigma$ determine a unique element $\omega\in\parameters$
fitting with the statement of Theorem \ref{theorem:description.Phik},
and by construction the map $(T_1,T_2;\lambda',\mu';\nu)\mapsto \omega$
so obtained is the inverse procedure of $\genRS$.
\end{proof}

\begin{example} \label{example:genRS}
\begin{penumerate}
\item Let $ (p, q, r) = (5, 3, 4) $.  
For $\omega$ as in (\ref{formula:example.graph}) and Example \ref{example:Phik}, we obtain
%\\[-.1ex]
\begin{gather*}
\omega=
\begin{pmatrix}
\eb_2 & \eb_4 & \eb_5 & 0 \\ \hline
\eb_3 & \eb_1 & 0     & \eb_2 
\end{pmatrix}
\;\;
\rightsquigarrow
\;\;
\graphic(\omega)=
\mbox{\tiny
$\begin{picture}(100,20)(0,0)
\put(0,11){$\bullet$}\put(20,11){$\bullet$}\put(40,11){$\bullet$}\put(60,11){$\bullet$}\put(80,11){$\bullet$}
\put(0,18){$1^+$}\put(20,18){$2^+$}\put(40,18){$3^+$}\put(60,18){$4^+$}\put(80,18){$5^+$}
\put(0,-11){$\bullet$}\put(20,-11){$\bullet$}\put(40,-11){$\bullet$}
\put(0,-20){$1^-$}\put(20,-20){$2^-$}\put(40,-20){$3^-$}
\put(21,13){\line(1,-1){21}}
\put(61,12){\line(-3,-1){60}}
\put(82,13){\circle{8}}
\put(22.5,-9){\circle{8}}
\end{picture}$}, 
\\[2ex]
\genRS(\omega)=
\left(\;
{\tiny \young(13,2,4,5)\,,\ \young(13,2)\,;\ \yng(1,1,1)\,,\ \yng(2,1)\,;\ \yng(1,1) }
%\ytableaushort{13,2,4,5}, \ytableaushort{13,2} ;
%\ydiagram{1,1,1}, \ydiagram{2,1}; \ydiagram{1,1} }
\right).%=
%\left(
%{\tiny \ytableaushort{13,2,4,5} *[*(red)]{1,1} *[*(yellow)]{1,1,1}, \ytableaushort{13,2} *[*(red)]{1,1} *[*(yellow)]{2,1} }
%\right)
\end{gather*}
Here, to describe $ \omega $ as a pair of partial permutations, 
we use $ \eb_k $ for the $ k $-th elementary basis vector, i.e., a column vector whose $ k $-th entry is $ 1 $ and any other entries are $ 0 $ (we do not distinguish whether it is belonging to $ \C^p $ or $ \C^q $ for simplicity).  
%(we encode the tuple $\genRS(\omega)=(T_1,T_2;\lambda',\mu';\nu')$ tuple as the pair of tableaux $(T_1,T_2)$
%where the boxes of $\lambda',\mu',\nu$ are suitably colored).
\item
We give another example for $ (p, q, r) = (5, 4, 4) $.  
%\\
\begin{gather*}
\omega
=\begin{pmatrix}
\eb_1 & \eb_2 & \eb_3 & \eb_5 \\ \hline
\eb_2 & \eb_4 & 0     & \eb_3
\end{pmatrix}
\;\;
\rightsquigarrow
\;\;
\graphic(\omega)=
\mbox{\tiny
$\begin{picture}(100,20)(0,0)
\put(0,11){$\bullet$}\put(20,11){$\bullet$}\put(40,11){$\bullet$}\put(60,11){$\bullet$}\put(80,11){$\bullet$}
\put(0,18){$1^+$}\put(20,18){$2^+$}\put(40,18){$3^+$}\put(60,18){$4^+$}\put(80,18){$5^+$}
\put(0,-11){$\bullet$}\put(20,-11){$\bullet$}\put(40,-11){$\bullet$}\put(60,-11){$\bullet$}
\put(0,-20){$1^-$}\put(20,-20){$2^-$}\put(40,-20){$3^-$}\put(60,-20){$4^-$}
\put(01,13){\line(1,-1){21}}
\put(21,12){\line(2,-1){41}}
\put(82,13){\line(-2,-1){41}}
\put(42.2,12.6){\circle{8}}
\end{picture}$}, 
\\[2ex]
\genRS(\omega)=
\left(
{\tiny \young(124,35)\,,\ \young(13,2,4)\,;\ \yng(2,2)\,,\ \yng(2,1)\,;\ \yng(2,1) }
%\ytableaushort{13,2,4,5}, \ytableaushort{13,2} ;
%\ydiagram{1,1,1}, \ydiagram{2,1}; \ydiagram{1,1} }
\right).%=
\end{gather*}

\item Let $p=q=r=n$. In this case, $\parameters$ contains the subset
$$
\parameters_0:=\left\{\omega_\sigma=\begin{pmatrix} 
\sigma \\
\unitmatrix_n \end{pmatrix}:\sigma\in\permutationsof{n}\right\}.
$$
Then we have
$$
\genRS(\omega_\sigma)=(\RS_1(\sigma),\RS_2(\sigma);\lambda,\mu;\nu)
$$
with $\lambda=\mu=\nu$ is the common shape of $\RS_1(\sigma)$ and $\RS_2(\sigma)$.
In this way, $\genRS$ is indeed an extension of Robinson--Schensted correspondence.
\item \label{example:genRS.3}
Let $p=r\geq 2$ and $q=1$. Then $\parameters$ contains three families of elements $\{\omega_i^{\mathrm{I}}\}_{i=1}^p$, 
$\{\omega_i^{\mathrm{II}}\}_{i=1}^p$, and
$\{\omega^{\mathrm{III}}\}$, with respective graphical representations as below. 
\begin{align*}
\graphic(\omega_i^{\mathrm{I}}) &= \;\omegaI\ \quad\;\;
%%\\[7ex]
& & 
\graphic(\omega_i^{\mathrm{II}}) &= \;\omegaII\ \quad
\\[7ex]
\graphic(\omega^{\mathrm{III}}) &= \;\omegaIII
\\[.1ex]
\end{align*}
Their images by $\genRS$ are as follows.
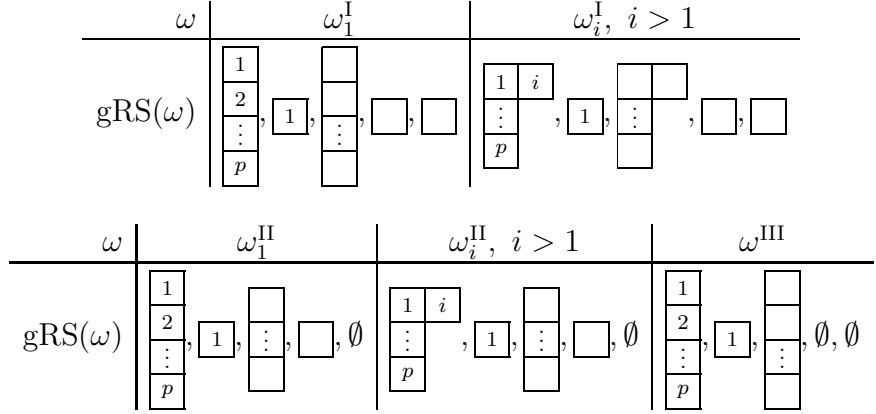
\begin{eqnarray*}
 & \begin{array}{r|c|c}    
\omega & \omega_1^\mathrm{I} & \omega_i^\mathrm{I},\ i>1 \\ \hline & & \\[-4mm]
\genRS(\omega) 
 & {\tiny \ytableaushort{1,2,\vdotsshiftedup,p} , \ytableaushort{1} , \ytableaushort{\ ,\ ,\vdotsshiftedup,\ } , \ytableaushort{\ } , \ytableaushort{\ } } 
 & {\tiny \ytableaushort{1i,\vdotsshiftedup,p} , \ytableaushort{1} , \ytableaushort{\ \ ,\vdotsshiftedup,\ } , \ytableaushort{\ } , \ytableaushort{\ } }
\end{array} \\[2ex]
 & {\begin{array}{r|c|c|c}    
\omega & \omega_1^{\mathrm{II}} & \omega_i^{\mathrm{II}},\ i>1 & \omega^{\mathrm{III}} \\ \hline
 & & & \\[-4mm]
\genRS(\omega) 
 & {\tiny \ytableaushort{1,2,\vdotsshiftedup,p} , \ytableaushort{1} , \ytableaushort{\ ,\vdotsshiftedup,\ } , \ytableaushort{\ } , \emptyset } 
 & {\tiny \ytableaushort{1i,\vdotsshiftedup,p} , \ytableaushort{1} , \ytableaushort{\ ,\vdotsshiftedup,\ } , \ytableaushort{\ } , \emptyset } 
 & {\tiny \ytableaushort{1,2,\vdotsshiftedup,p} , \ytableaushort{1} , \ytableaushort{\ ,\ ,\vdotsshiftedup,\ } , \emptyset , \emptyset } 
\end{array}}
\end{eqnarray*}
\end{penumerate}
\end{example}

\subsubsection{Relation to the Travkin's mirabolic Robinson-Schensted correspondence}

In \cite{Travkin.2009}, Travkin considers the variety
$$
X=\Flags(V)\times \Flags(V)\times V,
$$
with diagonal action of $\GL(V)$,
which is a slight variation of the triple flag variety 
$$ 
\Flags(V)\times \Flags(V)\times \mathbb{P}(V) = \GL_n/B_n^+ \times \GL_n/B_n^+ \times \GL_n/P_\mathrm{mir} , 
$$
where $ P_\mathrm{mir}\subset \GL_n $ is mirabolic --- the stabilizer of a line in $ V=\C^n $.
Then, using an explicit description of the conormal bundles to the orbits, he derives a \emph{mirabolic} Robinson--Schensted correspondence 
(\cite[\S 3.3]{Travkin.2009}) 
\begin{align}
& X/\GL_n \simeq 
\{(w,\sigma) \mid  w \in \permutationsof{n},\ \sigma\text{ a descent of $w$}\}
\notag
\\
& \xrightarrow{\;\sim\;} 
\coprod_{(\lambda,\mu)\in\partitionsof{n}^2}
\{ (T_1,T_2,\nu) \mid T_1 \in \ST{\lambda}, \; T_2 \in \ST{\mu}, 
\text{ satisfying $(\star\star)$}\} ,     
\label{eq:Travkin.mirabolic.RS.corr}
\end{align}
where a \emph{descent} of $w$ means a (possibly empty) sequence $i_1<\ldots<i_s$ with $w(i_1)>\ldots>w(i_s)$, and the property 
$(\star\star)$ is given below (compare it with the condition $ (\star) $ in \S~\ref{section:generalized-RS}).
%%the same condition as in \S~\ref{section:generalized-RS} (claiming that $T_1,T_2$ are standard Young tableaux of respective shapes $\lambda,\mu$) and
\smallskip
\begin{itemize}
    \item[$(\star\star)$] 
    $\lambda\setminus\nu$ and $\mu\setminus\nu$ are column strip, that is, $\nu\columnstrip\lambda$ and $\nu\columnstrip\mu$.
\end{itemize}

\medskip

On the other hand, in \cite{Fresse.N.2020}, 
we consider an ``open big cell'' $ \dblFV^{\circ}  $ in the double flag variety $ \dblFV $ of type AIII below,
\begin{equation*}
    \dblFV 
    = (\GL_n/B_n^+ \times \GL_n/B_n^+) \times \GL_{2n}/P_{(n,n)}
    = K/B_K \times G/P_S \quad (P_S := P_{(n,n)}).
\end{equation*}
If we denote by $ U_S^{\circ} $ the unipotent radical of the opposite parabolic subgroup $ P_S^{\circ} $, 
the open big cell is defined as a $ K $-stable subvariety 
\begin{equation*}
\dblFV^{\circ} = K/B_K \times U_S^{\circ} P_S/P_S 
\simeq (\GL_n/B_n^+ \times \GL_n/B_n^+) \times \Mat_n , 
\end{equation*}
where $ (k_1, k_2) \in K = \GL_n \times \GL_n $ acts on $ \Mat_n $ by 
$ k_1 A k_2^{-1} \;\; (A \in \Mat_n) $.  
It is obvious that 
$ \dblFV^{\circ}/K \simeq B_n^+ \backslash \Mat_n / B_n^+ $, and 
the representatives of $ B_n^+ \times B_n^+ $-orbits in $ \Mat_n $ are precisely given by 
partial permutations $ \ppermutations_{n,n} $ (see \S~\ref{subsec:orbit.by.quivers.typeAIII}); for the proof, we refer to \cite{Fulton}, and also \cite[Lemma 3.4]{Fresse.N.2023}.  
{In \cite{Fresse.N.2023}, the representatives are described as 
$ \omega = \begin{pmatrix} \tau_1 \\ \unitmatrix_n \end{pmatrix} $ 
with $ \tau_1 \in \ppermutations_{n,n} $}.  
If we apply the algorithm for the generalized Robinson--Schensted correspondence explained in \S~\ref{section:generalized-RS}, 
we precisely obtain {a triple $ (T_1, T_2, \nu) $ belonging to the set} in the right hand side of the Travkin's mirabolic RS correspondence \eqref{eq:Travkin.mirabolic.RS.corr} 
(note that $ L = M' = \emptyset $ in Theorem~\ref{theorem:fibers.Phik} so that we do not need intermediate $ \lambda' $ and $ \mu' $ because $ \lambda'=\nu $ and $ \mu'=\mu$).
Thus we get  a bijection
\begin{align}
& \dblFV^{\circ}/ K \simeq 
\ppermutations_{n,n} \notag
\\
& \xrightarrow{\;\sim\;} 
\coprod_{(\lambda,\mu)\in\partitionsof{n}^2}
\{ (T_1,T_2,\nu) \mid T_1 \in \ST{\lambda}, \; T_2 \in \ST{\mu}, 
\text{ satisfying $(\star\star)$}\}. 
\label{eq:big.cell.DFV.K}
\end{align}
Consequently all of these five sets that appear in \eqref{eq:Travkin.mirabolic.RS.corr} and \eqref{eq:big.cell.DFV.K}
are in bijection.  
In particular we get

\begin{proposition}
There is a bijection between orbits in 
$ (\GL_n/B_n^+ \times \GL_n/B_n^+ \times \C^n)/\GL_n $
and 
$ (\GL_n/B_n^+ \times \GL_n/B_n^+ \times \Mat_n)/(\GL_n \times \GL_n) $.  
\end{proposition}

We get this proposition {by} comparing mirabolic/generalized RS correspondence.  
However, constructing bijection between Travkin's set 
$ \{ (w, \sigma) \mid w \in S_n, \sigma \subset {\{1,\ldots,n\}} \text{ is a {descent} in $ w $} \} $ 
and the partial permutations $ \ppermutations_{n,n} $ is somewhat straightforward.  
In fact, for $ \tau \in \ppermutations_{n,n} $, we get a permutation matrix $ w\in\permutationsof{n} $ by adding $1$'s in some columns $ i_1,\ldots,i_k $ in such a way that 
$ i_1 < i_2 < \dots < i_k $ and $ w(i_1) > w(i_2) > \dots > w(i_k) $, 
where $ k = n - \rank \tau $. 
Then $ \sigma = \{ w(i_1) , w(i_2) , \dots , w(i_k) \} $ is a descent for $ w $, and this actually gives a bijection between orbital representatives.  

We thank Anthony Henderson for pointing out the bijection to us.  

\begin{remark}
\begin{penumerate}
\item
In the mirabolic case, 
there is a natural map from the conormal variety to the \emph{enhanced nilpotent cone} 
$ \nilpotentsof{\lie{gl}_n} \times V $ (see \cite{Achar.Henderson.2008}), which is $ \GL_n $-equivariant.  
It is interesting that the orbits $ (\nilpotentsof{\lie{gl}_n} \times V)/ \GL_n $ are classified by the bipartitions 
$ Q_n = \{ (\mu, \nu) \mid |\mu| + |\nu| = n \} $ which is further in bijection with 
$ \mathfrak{P} = \{ (\nu, \lambda) \mid |\lambda| = n , \nu\columnstrip\lambda \} $.  
Note that the latter set $ \mathfrak{P} $ also classifies the {pairs} $ (U, W) $, where 
$ U $ is an irreducible unipotent representation of $ \GL_n(\bbF_q) $, and 
$ W $ is an irreducible constituent of the restriction of $ U $ to the mirabolic subgroup $ P_n(\bbF_q) $ 
(\cite{Zelevinsky.LNM1981}; see also \cite[\S~2.2]{Travkin.2009}).)

\item
For the symplectic group, Syu Kato developed a generalization of Deligne-Langlands theory over the exotic nilpotent cone (\cite{Kato.2009}).  
In this case, Henderson-Trapa \cite{Henderson.Trapa.2012} and 
Nandakumar-Rosso-Saunders \cite{NRS.I.2018} establish a kind of Robinson--Schensted correspondence called the exotic RS correspondence.  
They are also related to the set of bipartitions $ Q_n $ or $ \mathfrak{P} $.  

\item 
There is yet another bijection for $ \ppermutations_{n,n} $ by using so-called ``admissible signed Young diagrams''.  For this, see \cite{Singh.arXiv.2020}. 
\end{penumerate}

\smallskip
It {would be} interesting to develop a general theory which explains all of these related results.  
\end{remark}

\subsection{Combinatorial description of the exotic Steinberg map $\Phi_\fs$}

The description of the exotic Steinberg map in terms of signed Young diagrams is more involved.

\begin{theorem}\label{thm:exotic.Steinberg.map}
Let $\omega\in \parameters$. With the notation of Theorem \ref{theorem:description.Phik},
we consider the bijections
\begin{eqnarray}
\label{ws+} & w_{\fs,+} = \begin{pmatrix} m_1 & \cdots & m_t & j_1 & \cdots & j_k & q+1 & \cdots & q+s' \\
-1 & \cdots & -t & \sigma(j_1) & \cdots & \sigma(j_k) & \ell'_{s'} & \cdots & \ell'_1 \end{pmatrix}, \\
\label{ws-} & w_{\fs,-} = \begin{pmatrix} \ell_1 & \cdots & \ell_s & i_1 & \cdots & i_k & p+1 & \cdots & p+t' \\
-1 & \cdots & -s & \sigma^{-1}(i_1) & \cdots & \sigma^{-1}(i_k) & m'_{t'} & \cdots & m'_1 \end{pmatrix}.
\end{eqnarray}
Then, the image of $\omega$ by the exotic Steinberg map is the signed Young diagram $\Phi_\fs(\omega)=\Lambda$  determined as follows. 
Here, {$\nrplus{\Lambda}{c}$, resp. $\nrminus{\Lambda}{c}$ (resp. $\nrboxes{\lambda}{c}$)
denotes the number of $+$'s, resp. $-$'s (resp. boxes)} in the first $c$ columns of $\Lambda$ (respectively $\lambda$).
\begin{penumerate}
\item\label{thm:ex.St.map:part-a}%%[\rm (a)] 
For every $c\geq 1$ even,
\[
\nrplus{\Lambda}{c}=\nrboxes{\lambda}{c}\quad\mbox{and}\quad \nrminus{\Lambda}{c}=\nrboxes{\mu}{c},
\]
where $(\lambda,\mu)=\Phi_\fk(\omega)$.
\item\label{thm:ex.St.map:part-b}%%[\rm (b)] 
For every $c\geq 1$ odd,
\[
\nrplus{\Lambda}{c}=s-t+\nrboxes{\lambda'}{c}\quad\mbox{and}\quad \nrminus{\Lambda}{c}=t-s+\nrboxes{\mu'}{c},
\]
where $(\lambda',\mu')$ is the pair of Young diagrams given by
\[(\lambda',\mu')=\big(\shape(\RS_1(w_{\fs,+})),\shape(\RS_1(w_{\fs,-}))\big).\]
\end{penumerate}
\end{theorem}

\begin{proof}[Sketch of proof]
In view of the definition of $\Phi_\fs$ in Definition \ref{sec:Definition.Steinberg.maps},
we need to determine the nilpotent $\GL_p\times\GL_q$-orbit of $\fs$ which intersects the projection $\{y^{-\theta}:y\in\mathcal{D}_\omega\}$ densely, where
$\mathcal{D}_\omega$ is the conormal direction already considered in the proof of Theorem \ref{theorem:description.Phik}.

Since every $y\in\mathcal{D}_\omega$ satisfies $y^2=0$, we already have
$$
(y^{-\theta})^{2m}=(-1)^m(y^\theta)^{2m}\quad\text{for all}\ m\geq 1,\ \text{all}\ y\in\mathcal{D}_\omega.
$$
This, combined with Theorem \ref{theorem:description.Phik}, proves \eqref{thm:ex.St.map:part-a}.
%%part (a) of the theorem.

For showing \eqref{thm:ex.St.map:part-b}, 
%%part (b), 
we need to characterize the blocks $c\in\Mat_{q,p}$ satisfying the condition
$$
\exists a,b,d\ \text{such that}\ \begin{pmatrix}
a & b \\ c & d
\end{pmatrix}\in\mathcal{D}_\omega.
$$
This is done, again, by relying on Steinberg's result (Theorem \ref{theorem:Steinberg}), though the construction here is less direct than in Theorem \ref{theorem:description.Phik}; see \cite[\S4.4]{Fresse.N.2023} for a complete proof.
\end{proof}

\begin{example}
\begin{penumerate}
\item
Taking again $\omega$ as in (\ref{formula:example.graph}) and Example \ref{example:Phik}, we have
$$
w_{\fs,+}=\begin{pmatrix} 1 & 2 & 3 & 4 & 5 \\ 4 & -1 & 2 & 3 & 1 \end{pmatrix}\quad\text{and}\quad w_{\fs,-}=\begin{pmatrix} 2 & 4 & 5 \\ 3 & 1 & -1 \end{pmatrix},
$$
which (combined with the calculation of $\Phi_\fk(\omega)$ in Example \ref{example:Phik}) yields
$$
\lambda'=\mbox{\scriptsize $\yng(3,1,1)$},\quad\mu'=\mbox{\scriptsize $\yng(1,1,1)$},\quad \mbox{and}\quad\Lambda=\mbox{\scriptsize $\young(-+,-+,+,+,+,-)$}.
$$

\item 
Assume $p=r\geq 2$ and $q=1$
as in Example \ref{example:genRS}\,\eqref{example:genRS.3}.
With the same notation as in Example \ref{example:genRS}, we obtain the following images by the exotic Steinberg map $\Phi_\fs$.
\begin{eqnarray*}
{\begin{array}{r|c|c|c|c}    
\omega & \ \ \omega_1^\mathrm{I}\ \ & \omega_i^\mathrm{I},\ i\geq 2 & \omega_i^{\mathrm{II}},\ i\geq 1 & \omega^{\mathrm{III}} \\ \hline  & & & \\[-4.8mm]
\Phi_\fs(\omega) 
 & {\tiny \ytableaushort{+,\vdotsshiftedup,+,-} } 
 & {\tiny \ytableaushort{+-,+,\vdotsshiftedup} } 
 & {\tiny \ytableaushort{-+,+,\vdotsshiftedup} }
 & {\tiny \ytableaushort{+-,+,\vdotsshiftedup} } 
 \end{array}}
\end{eqnarray*}
\end{penumerate}
\end{example}

Unlike the case of the symmetrized Steinberg map, we have almost no information about the fibers of the exotic Steinberg map $\Phi_\fs$. 
It would be interesting to identify the fibers, which will give different parametrization of the orbits $ \dblFV/K $ in terms of a generalized Robinson--Schensted correspondence analogous to Theorem \ref{theorem:fibers.Phik}.

\part{\textbf{Embedding theory for double flag varieties}}\label{part:Orbit.Embedding}
%%\section*{\textbf{Part IV. Embedding theory for double flag varieties}}\partlabel{part:Orbit.Embedding}
%%\addcontentsline{toc}{section}{\textbf{Part IV. Embedding theory for double flag varieties}\partstrut}

Let us return back to a general double flag variety associated with a symmetric pair.
Thus $ (G, K) $ is a symmetric pair, $ P, Q $ being parabolic subgroups 
of $ G $ and $ K $ respectively, and 
we consider a double flag variety $ \dblFV = K/Q \times G/P $.  

{Here by ``embedding theory'' we mean general results which guarantee an embedding between orbit sets, and general consequences of these results.}
There {exist general results on orbit embedding 
(see \cite{Ohta.2008}, \cite{Nishiyama.2014}), 
and one can apply them} to the case of double flag varieties 
(\cite{Fresse.N.2021}).
If $ \dblFV $ can be embedded into a larger, well-understood double flag variety, 
we can deduce finiteness of orbits or even orbit classifications from those of the larger double flag variety.  

\section{Embedding of double flag varieties and their orbits}\label{sec:embedding.theory.DFV}

Let $ \bbG $ be a (larger) connected reductive algebraic group over $ \C $, 
and let $ \theta, \sigma \in \Aut \bbG $ be two involutions which are commutative.  
Let $ \bbK = \bbG^{\theta} $ be the fixed point subgroup of $ \theta $, 
and assume that our $ G $ is the fixed point subgroup of $ \sigma $, i.e., $ G = \bbG^{\sigma} $.
Then $ ( \bbG, \bbK ) $ is a symmetric pair, and 
$ K = G^{\theta} = \bbK^{\sigma} = G \cap \bbK $ is a symmetric subgroup of both $ G $ and $ \bbK $.
We assume all the subgroups $ G, K $ and $ \bbK $ are connected.
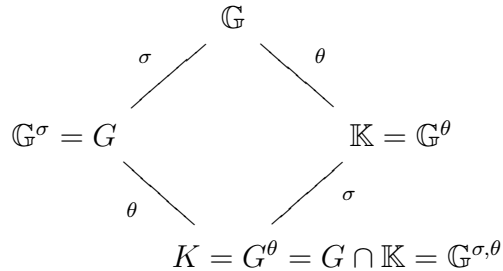
\begin{equation*}
\vcenter{
\xymatrix @R-.3ex @M+.5ex @C-3ex @L+.5ex @H+1ex {
 & \ar@{-}[ld]_{\sigma} \makebox[3ex][c]{$\bbG$} \ar@{-}[rd]^{\theta} &
\\
\ar@{-}[rd]_(.4){\theta} \makebox[3ex][r]{$ \bbG^{\sigma} = G $}  &  & \makebox[3ex][l]{$ \bbK = \bbG^{\theta} $} \ar@{-}[ld]^(.3){\sigma}
%%\ar@{}@<1ex>[ld]_{\theta} \ar@{-}[ld]
\\
& \makebox[9ex][l]{$K = G^{\theta} = G \cap \bbK = \bbG^{\sigma, \theta}$} & 
}}
\end{equation*}
\skipover{
\begin{equation*}
\begin{array}{ccc}
\bbK & \subset & \bbG 
\\
\text{\rotatebox[origin=c]{-90}{$\supset$}} & & \text{\rotatebox[origin=c]{-90}{$\supset$}} 
\\
K & \subset & G 
\end{array}
\end{equation*}
}
\indent
For any parabolic subgroup $ P $ of $ G $, 
there exists a $ \sigma $-stable parabolic subgroup $ \bbP $ of $ \bbG $ 
satisfying 
$ P = \bbP^{\sigma} = \bbP \cap G $.  
Similarly, for a parabolic subgroup $ Q $ of $ K $, 
we choose a $ \sigma $-stable parabolic subgroup 
\skipover{
In most of the literature, the letter $ \bbQ $ is reserved to denote 
the field of rational numbers.  However, in this paper, we consider everything over 
the complex number field $ \C $ and actually the rational number field does play no rule.  
For this reason, we choose the letter $ \bbQ $ for a parabolic subgroup of $ \bbK $.}
$ \bbQ $ of $ \bbK $ 
such that 
$ Q = \bbQ^{\sigma} = \bbQ \cap K $.  
We denote 
$ \bbdblFV = \bbG / \bbP \times \bbK / \bbQ $ a double flag variety 
for $ ( \bbG, \bbK ) $, and   
$ \dblFV = G/P \times K/Q $ as before.  
Then we get a natural embedding 
\begin{equation*}
\dblFV \hookrightarrow \bbdblFV, \qquad
(g P, k Q)  \mapsto (g \bbP, k \bbQ) \quad\; (g \in G, k \in K)
\end{equation*}
and we will use the same notation 
$ \dblFV $ for its embedded image in $ \bbdblFV $.  

By abuse of notation, let $ \sigma $ also denote {the} automorphism of $ \bbdblFV $ defined by 
$ 
\sigma (\bbg \cdot \bbP, \bbk \cdot \bbQ) 
= (\sigma(\bbg) \cdot \bbP, \sigma(\bbk) \cdot \bbQ) 
$ so that $ \sigma $ is a {$ \bbK $}-equivariant automorphism on $ \bbdblFV $, i.e., 
$ \sigma(h \cdot \bbx) = \sigma(h) \cdot \sigma(\bbx) $ holds 
for any $ h \in \bbK $ and $ \bbx \in \bbdblFV $.

\begin{definition}\label{def:-sigma.square.root}
For a $ \sigma $-stable subgroup $ \bbH $ of $ \bbG $, 
let us denote 
\begin{equation*}
\bbH^{-\sigma} = \{ h \in \bbH \mid \sigma(h) = h^{-1} \} .
\end{equation*}
We say that the subgroup \emph{$ \bbH $ admits $ (-\sigma) $-square roots}  
if for any $ h \in \bbH^{- \sigma} $ there exists an $ f \in \bbH^{-\sigma} $ 
which satisfies $ h = f^2 $.  
In this case $ f $ is called a $ (-\sigma) $-square root of $ h $.  
\end{definition}

\begin{lemma}
If both $ \bbP $ and $ \bbQ $ admit $ (-\sigma) $-square roots, then
$ \dblFV $ 
coincides with the fixed point set 
$ \bbdblFV^{\sigma} := \{ \bbx \in \bbdblFV \mid \sigma(\bbx) = \bbx \} $ 
of $ \sigma $.  
\end{lemma}

\begin{proof}
Take $ \bbg \bbP \in (\bbG/\bbP)^{\sigma} $.  
Then we get 
$ \bbg^{-1} \sigma(\bbg) \in \bbP^{-\sigma} $ since $ \bbP $ is $ \sigma $-stable.  
Therefore there exists a $ (-\sigma) $-square root $ f \in \bbP^{-\sigma} $ of $ \bbg^{-1} \sigma(\bbg) $.  
By this definition $ \bbg f = \sigma(\bbg) f^{-1} = \sigma(\bbg f) $ holds, 
and hence $ \bbg f \in G $, which in turn implies $ \bbg \bbP = (\bbg f) \bbP $ belongs to $ G/P = G \cdot \bbP \subset \bbG / \bbP $.  
The proof of $ \bbk \bbQ \in K/ Q $ in $ \bbK/\bbQ $ is the same.  
\end{proof}

The natural $ K $-equivariant embedding 
$ \dblFV \hookrightarrow \bbdblFV $ 
induces an orbit map 
\begin{equation*}
\xymatrix @R -5ex @C -6ex @L 5pt { \iota : & \dblFV / K \ar[rr] & \qquad \qquad & \bbdblFV / \bbK
\\
& \text{\rotatebox[origin=c]{-90}{$\ni$}} &  & \text{\rotatebox[origin=c]{-90}{$\ni$}} 
\\
& \calorbit \ar@{|->}[rr] & & \makebox[1em][l]{\,$ \orbit = \bbK \cdot \calorbit $}} ,
\end{equation*}
which is not injective in general.  
(In the previous sections, we often used the notation $ \orbit $ to denote a $ K $-orbit in $ \dblFV $, 
but here we change it to $ \calorbit $.)  
However, in a nice situation (like the following theorem), the map $ \iota $ becomes an injection and 
we get $ \calorbit = \orbit \cap \dblFV $ is a single $ K $-orbit.  

\begin{theorem}\label{thm:embedding.orbits}
Let us consider the double flag varieties 
$ \dblFV = K/Q \times G/P $ and $ \bbdblFV = \bbK / \bbQ \times \bbG / \bbP $ above.  
We assume 
\begin{itemize}
\item[\upshape(E)]
The parabolic subgroups $ \bbP $ and $ \bbQ $ admit $ (-\sigma) $-square roots, 
and,  
for any $ \sigma $-stable parabolic subgroups $ \bbP_1 \subset \bbG $ and 
$ \bbQ_1 \subset \bbK $ which are conjugate to $ \bbP $ and $ \bbQ $ respectively, 
the intersection $ \bbP_1 \cap \bbQ_1 $ also admits $ (-\sigma) $-square roots.  
\end{itemize}
Under the condition {\upshape(E)}, the following statements hold.  
\begin{penumerate}
\item
The natural orbit map 
$ \iota : \dblFV / K \to \bbdblFV / \bbK $ defined by 
$ \iota(\calorbit) = \bbK \cdot \calorbit $  for $ \calorbit \in \dblFV / K $ is injective, 
i.e., if we put $ \orbit = \iota(\calorbit) $, then it holds $ \calorbit = \orbit \cap \dblFV = \orbit^{\sigma} $.

\item
For any $ \bbK $-orbit $ \orbit $ in $ \bbdblFV $, 
the intersection $ \orbit \cap \dblFV $ is either empty or a single $ K $-orbit.
\end{penumerate}
\end{theorem}

For the proof, we refer the readers to \cite[Theorem~3.3]{Fresse.N.2021}, where Theorem~3.2 in \cite{Nishiyama.2014} is used in crucial way.  
Note that we do not need finiteness of the orbits for this theorem.  

\begin{corollary}
Under the same notation and the condition {\upshape (E)} above, 
if the double flag variety $ \bbdblFV = \bbK / \bbQ \times \bbG / \bbP $ is of finite type, then 
$ \dblFV = K/Q \times G/P $ is also of finite type.
\end{corollary}

\section{Embedding of conormal bundles}

From now on, we always assume the condition {\upshape (E)} in Theorem~\ref{thm:embedding.orbits}.

Let us denote by $ \lie{G} $ the Lie algebra of $ \bbG $.  
Similarly we will use German capital letters for the Lie algebras of the groups denoted by black board bold letters.  
We fix a nondegenerate invariant bilinear form on 
$ \lie{G} $ which is also invariant under the involutions $ \sigma $ and $ \theta $ once and for all, 
and we identify the duals of $ \lie{g}, \lie{k} $ and $ \lie{K} $ with themselves using this bilinear form.  
Let us consider the cotangent bundles 
$ T^*\dblFV $ and $ T^*\bbdblFV $.  
Since 
$ T^*\bbdblFV 
\simeq (\bbG \times_{\bbP} \nilradical{\lie{P}}) 
\times (\bbK \times_{\bbQ} \nilradical{\lie{Q}}) $, 
and $ \nilradical{\lie{P}} $ and $ \nilradical{\lie{Q}} $ are $ \sigma $-stable, 
%%by the identification $ \lie{G} \simeq \lie{G}^{\ast} $ and $ \lie{K} \simeq \lie{K}^{\ast} $, 
we have a natural extension of $ \sigma \in \Aut \bbdblFV $ to the whole $ T^*\bbdblFV $.  
We prove

\begin{lemma}
%%Let us assume the condition {\upshape(A)} in {\upshape\S~\ref{subsec:embedding.orbits}}. 
$ T^*\dblFV $ is the set of $ \sigma $-fixed points in 
$ T^*\bbdblFV $, i.e., 
$ (T^*\bbdblFV)^{\sigma} = T^*\dblFV $ holds.
\end{lemma}

\begin{proof}
%%It is obvious that $ T^*\dblFV \subset (T^*\bbdblFV)^{\sigma} $.  Let us prove the reversed inclusion.  
Let us prove $ (\bbG \times_{\bbP} \nilradical{\lie{P}})^{\sigma} = G \times_{P} \nilradical{\lie{p}} $.  
Take $ [\bbg, \bbu] \in \bbG \times_{\bbP} \nilradical{\lie{P}} $.  
Then $ \sigma([\bbg, \bbu]) = [\sigma(\bbg), \sigma(\bbu)] = [\bbg, \bbu] $ 
if and only if there exists $ \bbp \in \bbP $ such that 
$ \sigma(\bbg) = \bbg \bbp $ and $ \sigma(\bbu) = {\bbp}^{-1} {\bbu} $.
So we get $ \bbp = \bbg^{-1} \sigma(\bbg) \in \bbP^{-\sigma} $, and it has a $ (-\sigma) $-square root $ f \in \bbP^{-\sigma} $.  
From this, we know $ \sigma(\bbg f) = \bbg f $, which means $ \bbg f \in G = \bbG^{\sigma} $.  
Also we see 
\begin{equation*}
\sigma(f^{-1} \bbu) = f \sigma(\bbu) = f \bbp^{-1} \bbu = f f^{-2} \bbu = f^{-1} \bbu ,
\end{equation*}
which proves $ f^{-1} \bbu \in (\nilradical{\lie{P}})^{\sigma} = \nilradical{\lie{p}} $.  
Thus $ [\bbg, \bbu] = [\bbg f, f^{-1} \bbu] \in G \times_{P} \nilradical{\lie{p}} $.

The proof for 
$ (\bbK \times_{\bbQ} \nilradical{\lie{Q}})^{\sigma} 
= K \times_Q \nilradical{\lie{q}} $ is the same.
\end{proof}

Let us denote the moment maps by 
\begin{equation}
\mu_{\bbdblFV} : T^*\bbdblFV \to \lie{K}^{\ast} \simeq \lie{K} , 
\qquad
\mu_{\dblFV} : T^*\dblFV \to \lie{k}^{\ast} \simeq \lie{k} .
\end{equation}
Clearly $ \mu_{\bbdblFV} $ commutes with $ \sigma $ and 
$ \mu_{\bbdblFV} \restrict_{T^*\dblFV} = \mu_{\dblFV} $.  
Let $ \conormal_{\bbdblFV} = \mu_{\bbdblFV}^{-1}(0) $ 
and $ \conormal_{\dblFV} = \mu_{\dblFV}^{-1}(0) $ 
be the corresponding conormal varieties.  

\begin{theorem}\label{thm:embedding.conormal.bundles}
Assume the condition 
{\upshape(E)} in Theorem~{\upshape\ref{thm:embedding.orbits}}.
\begin{penumerate}
\item\label{thm:embedding.conormal.bundles:item:01}
The set of fixed points in $ \conormal_{\bbdblFV} $ is equal to $ \conormal_{\dblFV} $.  
Thus we get $ (\conormal_{\bbdblFV})^{\sigma} = \conormal_{\bbdblFV} \cap T^*\dblFV = \conormal_{\dblFV} $.  
\item\label{thm:embedding.conormal.bundles:item:02}
By Theorem~\ref{thm:embedding.orbits}, 
there is an embedding $ \iota : \dblFV / K \hookrightarrow \bbdblFV / \bbK $.  
If we denote $ \orbit = \iota(\calorbit) $ for $ \calorbit \in \dblFV/K $, we have $ (T^*_{\orbit}\bbdblFV)^{\sigma} = T^*_\calorbit\dblFV $.
\end{penumerate}
\end{theorem}

\begin{proof}
\eqref{thm:embedding.conormal.bundles:item:01}\ 
$ (\conormal_{\bbdblFV})^{\sigma} 
= (T^*\bbdblFV)^{\sigma} \cap \mu_{\bbdblFV}^{-1}(0)
= T^*\dblFV \cap \mu_{\bbdblFV}^{-1}(0)
= T^*\dblFV \cap \mu_{\dblFV}^{-1}(0)
= \conormal_{\dblFV} $.  

{\eqref{thm:embedding.conormal.bundles:item:02}\ 
We have
\begin{eqnarray*}
(T^*_{\Xorbit}\bbX)^\sigma & = & \{([\mathbb{P}_1,u],[\mathbb{Q}_1,v])\in\conormal_{\bbX}^\sigma:(\mathbb{P}_1,\mathbb{Q}_1)\in\Xorbit \}\\
 & = & \{([P_1,u],[Q_1,v])\in\conormal_\Xfv:(P_1,Q_1)\in\Xorbit\cap \Xfv\} \\ 
 & = & T^*_\calorbit \Xfv
\end{eqnarray*}
where we use part (1) and the fact that $\Xorbit \cap\Xfv=\calorbit$.
}
\skipover{\eqref{thm:embedding.conormal.bundles:item:02}\ 
Since $ \orbit^{\sigma} = \calorbit $, we prove the equality for the fiber.  
Take $ x = (Q_1, P_1) \in \calorbit $.  
Then as a point in $ \bbK/\bbQ \times \bbG/\bbP $, 
it is also written as $ x = (\bbQ_1, \bbP_1) $, where $ Q_1 = \bbQ_1^{\sigma} $ and $ P_1 = \bbP_1^{\sigma} $.  
The fiber of the conormal bundle at $ x $ is given by 
$ (T^*_{\orbit}\bbdblFV)_x = \bigl( (\lie{Q}_1, \lie{P}_1) + \Delta \lie{K} \bigr)^{\bot} $, 
where $ \Delta \lie{K} $ denotes the diagonal embedding and $ \bot $ refers to the orthogonal space.  
Since all the subspaces appearing here are $ \sigma $-stable, we see 
\begin{equation*}
(T^*_{\orbit}\bbdblFV)_x^{\sigma} 
= \bigl( (\lie{Q}_1^{\sigma}, \lie{P}_1^{\sigma}) + \Delta \lie{K}^{\sigma}\bigr)^{\bot} 
= \bigl( (\lie{q}_1, \lie{p}_1) + \Delta \lie{k}\bigr)^{\bot} 
= (T^*_{\calorbit}\dblFV)_x
.
\end{equation*}}
\end{proof}

\section{Compatibility with Steinberg maps}\label{section:compatibility.Steinberg.maps}

In this section we keep the assumption (E) in Theorem~\ref{thm:embedding.orbits} 
so that we have an orbit embedding $\iota:\dblFV / K \hookrightarrow \bbdblFV / \bbK $, 
which induces the embedding of the conormal bundle $ T^*_\calorbit\dblFV \hookrightarrow T^*_{\orbit}\bbdblFV $.  

Recall the nilpotent varieties $ \nilpotentsof{\lie{k}} = \nilpotentsof{\lie{g}}^{\theta} $ 
and $ \nilpotentsof{\lie{s}} = \nilpotentsof{\lie{g}}^{-\theta} $.  
Similarly, we get the nilpotent varieties 
$ \nilpotentsof{\lie{K}} $ and 
$ \nilpotentsof{\lie{S}} $ 
for $ \lie{G} = \Lie (\bbG) $, 
where $ \nilpotentsof{\lie{S}} = \nilpotentsof{\lie{G}}^{-\theta} $.
Since $ \nilpotentsof{\lie{g}}^{\pm\theta} $ are varieties of $ \sigma $-fixed points in $ \nilpotentsof{\lie{G}}^{\pm\theta} $, 
we have natural orbit maps 
\begin{equation*}
\iota^{\pm\theta} : \nilpotentsof{\lie{g}}^{\pm\theta} / K \to \nilpotentsof{\lie{G}}^{\pm\theta} / \bbK , 
\qquad 
\iota^{\pm\theta}(\calorbit) = \bbK \cdot \calorbit 
\quad (\calorbit \in \nilpotentsof{\lie{g}}^{\pm\theta} / K).
\end{equation*}
The maps $ \iota^{\pm\theta} $ are not injective in general.

The Assumption~\ref{assumption:phis} on the image of the Steinberg map 
clearly holds for $\dblFV$ whenever it holds for $\bbdblFV$, which we assume below.
Under this assumption, we get 
the Steinberg maps $\Phi^{\pm\theta}:\dblFV/K\to \nilpotentsof{\lie{g}}^{\pm\theta}/K$ relative to $\dblFV$
and also $\bbPhi^{\pm\theta}:\bbdblFV/\bbK \to \nilpotentsof{\lie{G}}^{\pm\theta}/\bbK$ relative to $\bbdblFV$ 
{(see Definition \ref{def:definition.of.Steinberg.maps}; here we use $ \Phi^{\pm\theta} $ instead of $ \Phi_{\lie{k}} $ or $ \Phi_{\lie{s}} $, since we want to have notation which keep track of the involution)}. 
It is natural to expect the following.

\begin{conjecture}\label{conj:question.compatibility}
The orbit embedding $\iota:X/K\hookrightarrow \bbX/\bbK$ is compatible with the Steinberg maps $\Phi^{\pm\theta}$ and $\bbPhi^{\pm\theta}$, i.e., 
the diagram below is commutative.
\begin{equation}\label{eq:embedding.and.Steinberg.maps.commute}
\vcenter{
\xymatrix %@R -6ex @C -6ex @L 5pt
{ \dblFV / K \ar@{^{(}->}[d]^\iota \ar[rr]^{\text{\normalsize$\Phi^{\pm\theta}$}} & \qquad \qquad & \nilpotentsof{\lie{g}}^{\pm\theta} / K \ar[d]^{\iota^{\pm\theta}} \\
 \bbdblFV / \bbK \ar[rr]^{\text{\normalsize$\bbPhi^{\pm\theta}$}} & \qquad \qquad & \nilpotentsof{\lie{G}}^{\pm\theta} / \bbK} 
}
\end{equation}
\end{conjecture}

Unfortunately, the formalism of functoriality of taking $ \sigma $-invariants does not imply Conjecture~\ref{conj:question.compatibility}.  
Let $\calorbit\subset \dblFV$ be a $K$-orbit
corresponding to the $\bbK$-orbit $ \orbit = \iota(\calorbit) $ 
in $\bbdblFV$.
By the definition of $\bbPhi^{\pm\theta}$, 
there exists a non-empty $ \bbK $- and $ \sigma $-stable dense open subset 
$\bbU^{\pm} \subset T^*_\orbit\bbdblFV$ which satisfies
\begin{equation}
\bbU^{\pm} \subset (\bbphi^{\pm\theta})^{-1}\bigl(\bbPhi^{\pm\theta}(\orbit)\bigr) \cap T^*_\orbit\bbdblFV , 
%%\cap (\bbphi^{-\theta})^{-1}\bigl(\bbPhi^{-\theta}(\orbit)\bigr), 
\end{equation}
where we denote the Steinberg maps by $ \bbphi^{\pm \theta}: \conormal_{\bbdblFV} \to \nilpotentsof{\lie{G}}^{\pm \theta} $.
%%since $T^*_\orbit\bbdblFV$ satisfies these conditions. 
Note that $T^*_\calorbit\dblFV=(T^*_\orbit\bbdblFV)^\sigma$ by Theorem \ref{thm:embedding.conormal.bundles}.

\begin{lemma}\label{lemma:dense.U.proves.commutativity}
The equality 
{$\bbPhi^{\pm\theta}(\iota(\calorbit))=\iota^{\pm\theta}(\Phi^{\pm\theta}(\calorbit))$} 
holds (so that the diagram \eqref{eq:embedding.and.Steinberg.maps.commute} commutes) if and only if $\bbU^{\pm}$ intersects with $T^*_\calorbit\dblFV$.
\end{lemma}

{Note that, in Lemma \ref{lemma:dense.U.proves.commutativity}, the commutativity of the diagram for $ +\theta $ follows from the condition that $ \bbU^{+} \cap T^*_\calorbit\dblFV \neq \emptyset $ (and similarly for $ - \theta $).  
If the condition holds for both of $ \pm \theta $, 
the open set $ \bbU := \bbU^{+} \cap \bbU^{-} $ intersects with $ T^*_\calorbit\dblFV $, 
since $ T^*_\calorbit\dblFV $ is irreducible.}  

\skipover{
However, in general, it seems difficult to check the condition for $ \bbU $ in the above lemma.  
In the next section, we will discuss an explicit embedding of type CI into type AIII 
(see \eqref{eq:bbK.G.K.typeCI.embedding} below).  
Yet it is still difficult to prove the commutativity of the Steinberg maps in full generality.  
Direct calculations tell us that up to $ n = 3 $ there always exist such $ \bbU $ in 
Lemma \ref{lemma:dense.U.proves.commutativity}, so that the commutativity holds.
}

\section{Finiteness of orbits via embedding theory}

In this section, we provide useful examples of the embedding theory explained above.  
In fact, it turns out to be a powerful tool for finding out new double flag varieties of finite type.  
We use this embedding theory in two different ways.  
One is for the embedding to triple flag varieties, and the other is that into double flag varieties of type AIII.  
We begin with the embedding into type AIII, and then come back to 
the triple flag varieties.  
%%For this purpose, we first return back to type A setting.

\subsection{Embedding of double flag varieties into type AIII}\label{subsec:Embedding.AIII}

Let $ \bbG = \GL_n $ be a general linear group, 
and put $ \Mat_n = \Mat_n(\C) $, which is an ambient vector space of $ \bbG $
\footnote{Of course $ \Mat_n(\C) $ is the Lie algebra of $ \bbG $, but its structure of a finite dimensional associative algebra is more important.}.
Let us consider two commuting involutions $ \sigma, \theta \in \Aut \bbG $ defined by 
\begin{equation}
\theta(g) = I\, g\, I^{-1}, 
\qquad
\sigma(g) = J\,\transpose{g}^{-1} J^{-1}, 
\end{equation}
where $ I, J \in \GL_n $ are certain regular matrices. 
Since $ \sigma $ and $ \theta $ are involutions, we must have $ I^2 = J^2 = \unitmatrix_n $ 
up to nonzero constant multiple.   
Put
\begin{align*}\label{eq:bbK.embedding}
&
\bbK = \bbG^{\theta} = \{ g \in \GL_n \mid I g I^{-1} = g \} ,
\\
&
G = \bbG^{\sigma}, \qquad
K = G^{\theta} = \bbK^{\sigma}
\end{align*}
as before.  
Then the commutativity of $ \theta $ and $ \sigma $ enforces $ J \transpose{I}^{-1} = \alpha \, I J $, i.e., $ \sigma(I) = \alpha I $,  
for some $ \alpha \in \Cbatsu $.
We allow $ I = \unitmatrix_n $ (the identity matrix), 
and in that case we have 
$ \bbK = \bbG $ and $ K = G $.  

\begin{proposition}\label{prop:condition.E}
Under the above setting, the condition {\upshape(E)} in Theorem~\ref{thm:embedding.orbits} 
is satisfied for any choice of $ \sigma $-stable parabolic subgroups 
$ \bbP \subset \bbG $ and $ \bbQ \subset \bbK $.  
In particular, we get an embedding of orbits $ \dblFV/K \hookrightarrow \bbdblFV / \bbK $ 
of double flag varieties.
\end{proposition}

\begin{proof}
Let us prove that $ \bbP $ admits $ (-\sigma) $-square roots (Definition~\ref{def:-sigma.square.root}).  
To prove it, we follow the arguments in \cite[Theorem~1]{Ohta.2008}.  
So, we define an anti-involution 
$ \tau : \Mat_n \to \Mat_n $ by 
$ \tau(g) = J \transpose{g} J^{-1} = \sigma(g^{-1}) $.  
This is a linear anti-algebra homomorphism.  

Take $ g \in \bbP^{-\sigma} $, which means $ \sigma(g) = g^{-1} $.    
Note that $ \sigma(g) = g^{-1} $ is equivalent to $ \tau(g) = g $.  

By an easy linear algebra argument, 
we get a polynomial $ f(T) \in \C[T] $ in a variable $ T $, which satisfies 
$ f(g)^2 = g $.  Since $ g $ is invertible, $ f(g) $ is also invertible and hence in $ \bbG $.  
Since $ \tau(f(g)) = f(\tau(g)) = f(g) $, we get $ f(g) \in \bbG^{-\sigma} $.  
To prove that $ f(g) \in \bbP $, we use a characterization of parabolic subgroups 
due to Kempf \cite[p.~305]{Kempf.1978}   
and Mumford \cite[\S~2.2]{MFK.1994}.
Namely, there exists a one parameter subgroup $ \lambda : \Cbatsu \to \bbG $ such that
\begin{equation}
\bbP = \{ g \in \bbG \mid \text{the limit } \, \lim_{t \to 0} \lambda(t) g \lambda(t)^{-1} =: g_0 \text{ exists} \} .
\end{equation}
Using this, 
we see $ \lambda(t) f(g) \lambda(t)^{-1} = f(\lambda(t) g \lambda(t)^{-1}) \to f(g_0) $ has a limit, 
and conclude that $ f(g) \in \bbP \cap \bbG^{-\sigma} = \bbP^{-\sigma} $.  
So $ f(g) $ is a desired $ (-\sigma) $-square root of $ g $.  

For the parabolic subgroup $ \bbQ $, the arguments are similar, 
and we note $ f(g)\in \bbK=\bbG^\theta $ holds because $ \theta $ is just a conjugation by $ I $.  

We did not use any particular properties of $ \bbP $ and $ \bbQ $, so the same technique also proves that 
$ \bbP_1 \cap \bbQ_1 $ admits $ (-\sigma) $-square roots.  
\end{proof}

\subsection{Embedding of double flag varieties into triple flag varieties}\label{subsec:Embedding.3FV}

Now let us consider the case of triple flag varieties.  

Let $ G $ be a connected reductive group and $ K = G^{\theta} $ a symmetric subgroup of $ G $ 
fixed by a certain involution $ \theta \in \Aut G $ as above.
We will embed a double flag variety 
$ \dblFV = K/Q \times G/P $ into a triple flag variety $ \bbdblFV $ in $ K $-equivariant way (cf.~\S~\ref{section:mult.FV}).  
For that purpose, choose parabolic subgroups $ P_1, P_2, P_3 $ of $ G $ in the following way.  

We take $ P_1 := P $ and $ P_2 := \theta(P) $.  
For the parabolic subgroup $ Q $ of $ K $, we can choose a $ \theta $-stable parabolic subgroup $ P_3 \subset G $ 
which cuts out $ Q $ from $ K $, i.e., $ P_3 = \theta(P_3) $ and $ Q = P_3 \cap K $.
Thus our triple flag variety is 
$ \bbdblFV = G/P_1 \times G/P_2 \times G/P_3 $ as in \S~\ref{section:mult.FV}.  

We consider the following setting:
\begin{align*}
\bbG &= G \times G, \quad
\bbP = P \times \theta(P) \subset \bbG
\\
&
\sigma, \tau \in \Aut \bbG \text{ : involutions}, \;
\makebox[0pt][l]{
$\begin{aligned}[t]
\sigma(g_1, g_2) &:= (\theta(g_2), \theta(g_1)) %%\text{ and } 
\\
\tau(g_1, g_2) &:= (g_2, g_1) 
\end{aligned}$
}
\\
\bbK &= \bbG^{\tau} = \{ (g, g) \mid g \in G \} \simeq G,  &
\bbQ &= \{ (g,g) \mid g \in P_3 \} \simeq P_3
\\
G &= \bbG^{\sigma} = \{ (g, \theta(g)) \mid g \in G \} \simeq G, &
K &= \bbG^{\sigma, \tau} = \{ (g, g) \mid g = \theta(g) \} \simeq G^{\theta}
\end{align*}
Here we use the involution $ \tau $ instead of $ \theta $, which is used in the former sections, because we want to keep the setting $ K = G^{\theta} $.  

Now notice 
\begin{equation}\label{eq:3FV.in.setting.Embedding.DFV}
\bbdblFV = \bbK/\bbQ \times \bbG/\bbP \simeq G/P_3 \times \bigl( G/P \times G/\theta(P) \bigr)
\end{equation}
is the triple flag variety.  

\begin{theorem}\label{thm:DFV.embedded.into.3FV}
Let $ G = \GL_n $ and consider $ \theta(g) = J \transpose{g}^{-1} J^{-1} $ for $ J $ which satisfies 
$ J^2 = \unitmatrix_n $.  Apply the setting explained above in this subsection.  
Then we get an orbit embedding 
$ \dblFV = K/Q \times G/P \hookrightarrow \bbK/\bbQ \times \bbG/\bbP = \bbdblFV $, 
where $ \bbdblFV $ is the triple flag variety given in \eqref{eq:3FV.in.setting.Embedding.DFV}.
In particular, if the triple flag variety $ \bbdblFV $ is of finite type, 
the double flag variety $ \dblFV = K/Q \times G/P $ is of finite type.  
\end{theorem}

\begin{proof}
We use Theorem~\ref{thm:embedding.orbits}.  
What we must prove is the condition (E).  
Let us see 
\begin{align*}
\bbP^{-\sigma} &= \{ (g, \theta(g^{-1})) \mid g \in P \} \simeq P, 
\\
\bbQ^{-\sigma} &= \{ (g, g) \mid g \in P_3, \; g^{-1} = \theta(g) \} \simeq P_3^{-\theta} .
\end{align*}
Since $ \theta(g) = J \transpose{g}^{-1} J^{-1} $, 
the same argument of the proof of Proposition~\ref{prop:condition.E} can be applied.  
It is literally the same but after exchanging $ \theta $ here and $ \sigma $ {there.}
\end{proof}

\begin{remark}
If we are interested in just the finiteness of orbits in a double flag variety, 
there is a simpler comparison with triple flag varieties.  
See Theorem 3.4 in \cite{NO.2011}.  
The above Theorem~\ref{thm:DFV.embedded.into.3FV} is a refinement of Theorem 3.1 in \cite{NO.2011} for the above mentioned involution $ \theta $. 
\end{remark}

In the following, we list the tables of double flag varieties of finite type which are obtained by embedding into  triple flag varieties of finite type as well as double flag varieties of type AIII.  
The readers are warned that the lists are far from complete, and they overlap the already given lists in \S~\ref{section:finite.DFV.P.or.Q=Borel}.
Note also that, in these lists, we omit the trivial cases for which $ G = P $ or $ Q = K $.

{Since type AII is logically simpler than type AI, let us begin with type AII.  We will discuss on type AI in the second subsection (\S~\ref{subsec:AI}).}

\subsection{Type {AII}}\label{subsec:AII}

We use Theorem~\ref{thm:DFV.embedded.into.3FV} above.  
Let us consider $ \GL_{2n} $ instead of $ \GL_n $ and 
the involution $ \theta(g) = J \transpose{g}^{-1} J^{-1} $ defined as follows: 
\begin{equation*}
J = \begin{pmatrix}
0 & - \Sigma_n \\
\Sigma_n & 0
\end{pmatrix}
, \quad
\Sigma_n = \begin{pmatrix}
 & & 1 \\[-1.5ex]
 & \makebox[4ex][c]{$\iddots$} & \\[-1.2ex]
1 & &
\end{pmatrix}
\end{equation*}
Then our symmetric pair becomes $ (\GL_{2n}, \Sp_{2n}) $.  
We compare the double flag variety $ \dblFV $ with the triple flag variety for type A, 
the table of finite type in \S~\ref{subsubsec:table.3FV.typeA} gives the following table of finite type double flag variety.

To indicate parabolic subgroups, we use the subset of simple roots 
$ J \subset \Pi $ as in \S~\ref{sec:finiteness.criterions}, and the numbering of simple roots is indicated below.
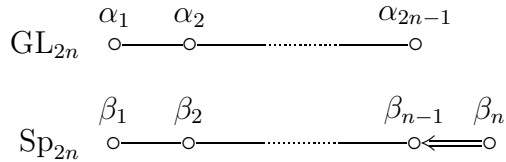
\begin{align*}
\GL_{2n} & 
\begin{xy}
\ar@{-} (0,0) *++!D{\alpha_1} *{\circ}="A"; (10,0) *++!D{\alpha_2} 
 *{\circ}="B"
\ar@{-} "B"; (20,0)="C" 
\ar@{.} "C"; (30,0)="D" 
\ar@{-} "D"; (40,0) *++!D{\alpha_{2 n-1}} *{\circ}="E"
\end{xy}
\\
\Sp_{2n} & 
\begin{xy}
\ar@{-} (0,0) *++!D{\beta_1} *{\circ}="A"; (10,0) *++!D{\beta_2} 
 *{\circ}="B"
\ar@{-} "B"; (20,0)="C" 
\ar@{.} "C"; (30,0)="D" 
\ar@{-} "D"; (40,0) *++!D{\beta_{n-1}} *{\circ}="E"
\ar@{<=} "E"; (50,0) *++!D{\beta_n} *{\circ}="F"
\end{xy}
\end{align*}

\begin{longtable}{c|c|c}
\hline
\rule[0pt]{0pt}{12pt}
$ (G,K) $
& $ \Pi_G \setminus J_G\ (P=P_{J_G}) $ & $\Pi_K \setminus J_K\ (Q=Q_{J_K})$
\\
\hline\hline
$ (\GL_{2n}, \Sp_{2n}) \smallvstrut $
& $ \{ \alpha_i \} \smallvstrut $ & any subset  \\
\cline{2-3}
& $ \{ \alpha_i , \alpha_j \} \smallvstrut $ & $ \{ \beta_n \} $  \\
\hline
\end{longtable}

\subsection{Type {AI}}\label{subsec:AI}

Again we use Theorem~\ref{thm:DFV.embedded.into.3FV} above.  
Let us consider the involution $ \theta(g) = \transpose{g}^{-1} $ so that $ J = \unitmatrix_n $ in this case.
Then our symmetric pair becomes $ (\SL_n, \SO_n) $ (we prefer $ \SL_n $ rather than $ \GL_n $).  
We compare $ \dblFV $ with the triple flag variety of type A as in the former subsection and get 
the following table of finite type double flag varieties.

%%To indicate parabolic subgroups, we use the subset of simple roots $ J \subset \Pi $ as in \S~\ref{sec:finiteness.criterions}, and 
The numbering of simple roots is indicated below.
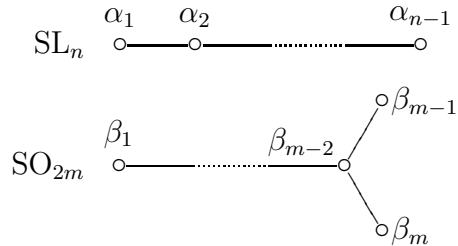
\begin{align*}
\SL_{n} & 
\begin{xy}
\ar@{-} (0,0) *++!D{\alpha_1} *{\circ}="A"; (10,0) *++!D{\alpha_2} 
 *{\circ}="B"
\ar@{-} "B"; (20,0)="C" 
\ar@{.} "C"; (30,0)="D" 
\ar@{-} "D"; (40,0) *++!D{\alpha_{n-1}} *{\circ}="E"
\end{xy}
\\
\SO_{2m} &
\begin{xy}
\ar@{-} (0,0) *++!D{\beta_1} *{\circ}="A"; (10,0)="C" 
\ar@{.} "C"; (20,0)="D" 
\ar@{-} "D"; (30,0) *+!DR{\beta_{m-2}} *{\circ}="E"
\ar@{-} "E"; (35,8.6)  *+!L{\beta_{m-1}} *{\circ}
\ar@{-} "E"; (35,-8.6)  *+!L{\beta_m} *{\circ}
\end{xy}
\skipover{
\\ & & 
\SO_{2m + 1} & 
\begin{xy}
\ar@{-} (0,0) *++!D{\alpha_1} *{\circ}="A"; (10,0)="C" 
\ar@{.} "C"; (20,0)="D" 
\ar@{-} "D"; (30,0) *++!D{\alpha_{m-1}} *{\circ}="E"
\ar@{=>} "E"; (40,0) *++!D{\alpha_m} *{\circ}="F"
\end{xy}
}
\end{align*}

\clearpage

\begin{longtable}{c|c|c|c}
\hline
\rule[0pt]{0pt}{12pt}
$ (G,K) $
& $ \Pi_G \setminus J_G\ (P=P_{J_G}) $ & $\Pi_K \setminus J_K\ (Q=Q_{J_K})$
\\
\hline\hline
$ (\SL_n, \SO_n) \smallvstrut $
& $ \{ \alpha_i \} \smallvstrut $ & any subset  \\
\cline{2-4}
& $ \{ \alpha_i , \alpha_j \} \smallvstrut $ & $ \{ \beta_{m - 1} \}, \{ \beta_m \} $  & if $ n = 2 m $ \\
\hline
\end{longtable}

\subsection{Type CI}\label{subsec:CI}

In the setting of \S~\ref{subsec:Embedding.AIII}, 
we take $ 2 n $ instead of $ n $, 
and choose
\begin{equation*}
J = \begin{pmatrix}
0 & - \Sigma_n \\
\Sigma_n & 0
\end{pmatrix}
, \quad
\Sigma_n = \begin{pmatrix}
 & & 1 \\[-1.5ex]
 & \makebox[4ex][c]{$\iddots$} & \\[-1.2ex]
1 & &
\end{pmatrix}, 
\qquad
\text{ and }
\quad
I = \begin{pmatrix}
\unitmatrix_n & \\
& - \unitmatrix_n
\end{pmatrix} .
\end{equation*}
Thus we get an embedding
\begin{equation*}
\begin{aligned}
\dblFV &= K/Q \times G/P = \GL_n/Q \times \Sp_{2n}/P \hookrightarrow 
\\
&
(\bbK_1/\bbQ_1 \times \bbK_2/\bbQ_2) \times \bbG/\bbP  =  (\GL_n/\bbQ_1 \times \GL_n/\bbQ_2) \times \GL_{2n}/\bbP = \bbdblFV  .
\end{aligned}
\end{equation*}
With this embedding, a 
$ \sigma $-stable parabolic subgroup $ \bbP $ of $ \bbG = \GL_{2n} $ corresponds to 
a composition $ (c_1, c_2, \dots, c_{\ell}) $ of $ 2n $ with 
$ c_i = c_{\ell - i + 1} $, i.e., central symmetric. {Moreover, since $ \bbQ_1 \times \bbQ_2 $ is $\sigma$-stable, the parabolic subgroups $ \bbQ_1 \subset \bbK_1 \cong \GL_n $ and $ \bbQ_2 \subset \bbK_2 \cong \GL_n $ are conjugated by $ \sigma $, and hence correspond to the same composition of $n$; then $Q = ( \bbQ_1\times \bbQ_2 )^\sigma \subset K \cong \GL_n$ also corresponds to this common composition of $n$.} 
Taking this into account, Table~\ref{table:AIII.finite.type} of the finite-type double flag varieties of type AIII tells us 
the following type CI double flag varieties of finite type.  
Each label of the following item indicates the labels in Table~\ref{table:AIII.finite.type}.
\begin{itemize}
\item[(P2)]
$ P = P_{(n,n)} $, and $ Q $ is any parabolic subgroup.
We call the parabolic subgroup $ P = P_{(n,n)} $ a Siegel parabolic subgroup, which stabilizes a Lagrangian subspace.
\item[(P3-1)]
$ P = P_{(1, 2 n - 2, 1)} $, and any $ Q $.  
We call the parabolic subgroup $ P = P_{(1, 2 n - 2, 1)} $ a mirabolic parabolic subgroup of $ \Sp_{2n} $, which stabilizes a line.
\item[(P3-2)]
$ P = P_{(2, 2 n - 4, 2)} $ or $ P_{(n - 1, 2, n - 1)} $, and $ Q $ is a maximal parabolic subgroup.  
\item[(P3-3)]
$ P $ is a maximal parabolic subgroup, which stabilizes an isotropic subspace, 
and $ Q = Q_{(1, n -1)} $ or $ Q_{(n - 1, 1)} $ is a mirabolic parabolic subgroup, which stabilizes a line (or a hyperplane).
\item[(P3-4)]
$ P $ is a maximal parabolic subgroup, and $ Q $ is a maximal parabolic subgroup.

\smallskip

\item
In the rest of the cases {(P4-1)--(P4-3)}, (P6), (PA-1)--(PA-3), 
{the only possibility for having $Q=(\bbQ_1\times \bbQ_2)\cap K$ is that $Q=K=\GL_n$, but we disregard this (trivial) situation in this section}.
\end{itemize}
From this, we get the following type CI double flag varieties of finite type.  
To indicate parabolic subgroups, we use the subset of simple roots 
$ J \subset \Pi $, and the numbering of simple roots is indicated below.
\begin{align*}
\Sp_{2n} & 
\begin{xy}
\ar@{-} (0,0) *++!D{\alpha_1} *{\circ}="A"; (10,0) *++!D{\alpha_2} 
 *{\circ}="B"
\ar@{-} "B"; (20,0)="C" 
\ar@{.} "C"; (30,0)="D" 
\ar@{-} "D"; (40,0) *++!D{\alpha_{n-1}} *{\circ}="E"
\ar@{<=} "E"; (50,0) *++!D{\alpha_n} *{\circ}="F"
\end{xy}
\\
\GL_n & 
\begin{xy}
\ar@{-} (0,0) *++!D{\beta_1} *{\circ}="A"; (10,0) *++!D{\beta_2} 
 *{\circ}="B"
\ar@{-} "B"; (20,0)="C" 
\ar@{.} "C"; (30,0)="D" 
\ar@{-} "D"; (40,0) *++!D{\beta_{n-1}} *{\circ}="E"
\end{xy}
\end{align*}

\begin{longtable}{c|c|c}
\hline
\rule[0pt]{0pt}{12pt}
$ (G,K) $
& $ \Pi_G \setminus J_G\ (P=P_{J_G}) $ & $\Pi_K \setminus J_K\ (Q=Q_{J_K})$
\\
\hline\hline
$ (\Sp_{2n}, \GL_n) \smallvstrut $
& $ \{ \alpha_1 \}, \{ \alpha_n \} \smallvstrut $ & any subset  \\
\cline{2-3}
& $ \{ \alpha_i \} \smallvstrut $ & $ \{ \beta_j \} $  \\
\hline
\end{longtable}

Note that, in the first case, $ Q $ can be a Borel subgroup $ B_K $ of $ K $, and it already appeared in \S~\ref{subsection:Table.Q=Borel}.

\subsection{Type CII}\label{subsec:CII}

In the setting of \S~\ref{subsec:Embedding.AIII}, 
we take $ 2 n $ instead of $ n $, 
$ 2 p $ (respectively $ 2 q $) instead of $ p $ (respectively $ q $), 
where $ n = p + q $.
Choose
\begin{equation*}
\begin{aligned}
J &= \begin{pmatrix}
0 & - \Sigma_n \\
\Sigma_n & 0
\end{pmatrix}
, \;\;
\Sigma_n = \begin{pmatrix}
 & & 1 \\[-1.5ex]
 & \makebox[4ex][c]{$\iddots$} & \\[-1.2ex]
1 & &
\end{pmatrix}, 
\quad
\\[1ex]
%%\text{ and }\quad
I &= \begin{pmatrix}
- I_{p,q} & \\
& I_{q,p}
\end{pmatrix} ,
\;\;
I_{p,q} = \begin{pmatrix}
\unitmatrix_p & \\
& - \unitmatrix_q
\end{pmatrix} .
\end{aligned}
\end{equation*}
Thus our embedding becomes 
\begin{equation*}
\begin{aligned}
\dblFV &= (K_1/Q_1 {\times} K_2/Q_2) {\times} G/P = (\Sp_{2p}/Q_1 {\times} \Sp_{2q}/Q_2) {\times} \Sp_{2n}/P 
\\
&\hookrightarrow 
(\bbK_1/\bbQ_1 {\times} \bbK_2/\bbQ_2) {\times} \bbG/\bbP  
\\
& \hspace*{.25\textwidth} =  (\GL_{2p}/\bbQ_1 {\times} \GL_{2q}/\bbQ_2) {\times} \GL_{2n}/\bbP = \bbdblFV  .
\end{aligned}
\end{equation*}
We do not repeat the consideration for each cases of finite type double flag varieties of type AIII, but only list the table of 
type CII double flag varieties of finite type.

\skipover{
Here is the numbering of simple roots for $ G = \Sp_{2n} $.  
\begin{align*}
\Sp_{2n} & 
\begin{xy}
\ar@{-} (0,0) *++!D{\alpha_1} *{\circ}="A"; (10,0) *++!D{\alpha_2} 
 *{\circ}="B"
\ar@{-} "B"; (20,0)="C" 
\ar@{.} "C"; (30,0)="D" 
\ar@{-} "D"; (40,0) *++!D{\alpha_{n-1}} *{\circ}="E"
\ar@{<=} "E"; (50,0) *++!D{\alpha_n} *{\circ}="F"
\end{xy}
\end{align*}
}
The numbering of simple roots for $ G = \Sp_{2n} $ is the same as type CI case, and 
those for $ K_1 = \Sp_{2p} $ and $ K_2 = \Sp_{2q} $ are also the same except for that we use 
$ \{ \beta_1, \dots, \beta_p \} $ for $ K_1 $ and 
$ \{ \gamma_1, \dots, \gamma_q \} $ for $ K_2 $ instead of $ \alpha_i $'s.

Note that the roles of $ K_1 $ and $ K_2 $ are interchangeable.  
In this table, {there are obtained 4 new cases other than those} which have been already listed in the case of $ P = B_G $ or $ Q = B_K $ (see \S\S~\ref{subsection:Table.P=Borel} and \ref{subsection:Table.Q=Borel}).

\clearpage

\begin{longtable}{c|c|c|c}
\hline
\rule[0pt]{0pt}{12pt}
%%$ (G,K) $ & $ \Pi_G \setminus J_G\ (P=P_{J_G}) $ & $\Pi_{K_1} \setminus J_{K_1}\ (Q_1=Q_{J_{K_1}})$ & $\Pi_{K_2} \setminus J_{K_2}\ (Q_2=Q_{J_{K_2}})$
$ (G,K) $
& $ \Pi_G \setminus J_G $ & $\Pi_{K_1} \setminus J_{K_1}$ & $\Pi_{K_2} \setminus J_{K_2}$
\\
& $ (P=P_{J_G}) $ & $(Q_1=Q_{J_{K_1}})$ & $ (Q_2=Q_{J_{K_2}})$
\\
\hline\hline
$ (\Sp_{2n}, \Sp_{2p} {\times} \Sp_{2q}) \smallvstrut $
& $ \{ \alpha_1 \}, \{ \alpha_n \} \smallvstrut $ & any subset & any subset
\\
\cline{2-4}
& $ \{ \alpha_2 \}, \{ \alpha_{n - 1} \} \smallvstrut $ & any subset & $ \{ \gamma_q \} $
\\
\cline{2-4}
& $ \{ \alpha_i \} $ & any subset & 
\begin{tabular}{l}
$ \emptyset \; (Q_2 = K_2)  \smallvstrut$
\\
$ \{ \gamma_1 \} $ if $ q=1 \smallvstrut$
\end{tabular}
\\
\cline{2-4}
& $ \{ \alpha_i \} \smallvstrut$ & $ \{ \beta_j \} , \{ \beta_j, \beta_p \} $  & $ \{ \gamma_q \} $ \\[1ex]
\cline{2-4}
& $ \{ \alpha_1, \alpha_n \} \smallvstrut $ & any subset & $ \emptyset \; (Q_2 = K_2) $ 
\\
\cline{2-4}
& $ \{ \alpha_i, \alpha_n \} \smallvstrut $ & $ \{ \beta_j \} , \{ \beta_j, \beta_p \} $ & $ \emptyset \; (Q_2 = K_2) $\\
\cline{2-4}
& {$ \{ \alpha_i, \alpha_n \} \smallvstrut $} & {any subset} & {$\emptyset$ if $q=1$} \\
\cline{2-4}
& $ \begin{array}{c} \{ \alpha_i, \alpha_j \}, \\ \{ \alpha_i, \alpha_j, \alpha_n \} \smallvstrut \end{array}  $ & $ \{ \beta_p \} $ & $ \emptyset \; (Q_2 = K_2) $
\\
\cline{2-4}
& any subset & $ \{ \beta_p \} \smallvstrut$ & $ \begin{array}{c} Q_2 = K_2 = \Sp_2 \\ ( q = 1) \end{array} $
\\
\cline{2-4}
& any subset & $ \{ \beta_1 \} \smallvstrut$ if $ p = 1 $ & $ \emptyset \; (Q_2 = K_2) $
\\
\hline
\end{longtable}

\subsection{Type BI (BDI)}\label{subsec:BI}

In the setting of \S~\ref{subsec:Embedding.AIII}, 
we take $ 2 n + 1 $ instead of $ n $, 
$ 2 p + 1 $ (respectively $ 2 q $) instead of $ p $ (respectively $ q $), 
where $ n = p + q $.
Choose
\begin{equation*}
J = \begin{pmatrix}
\Sigma_{2 p + 1} & 0 \\
0 & \Sigma_{2 q}  
\end{pmatrix}
, \;\;
\Sigma_m = \begin{pmatrix}
 & & 1 \\[-1.5ex]
 & \makebox[4ex][c]{$\iddots$} & \\[-1.2ex]
1 & &
\end{pmatrix}, 
\quad
%%\text{ and }\quad
I = \begin{pmatrix}
\unitmatrix_{2 p + 1} & \\
& - \unitmatrix_{2q}
\end{pmatrix} .
\end{equation*}
Thus our embedding becomes 
\begin{equation*}
\begin{aligned}
\dblFV &= (K_1/Q_1 {\times} K_2/Q_2) {\times} G/P = (\SO_{2p + 1}/Q_1 {\times} \SO_{2q}/Q_2) {\times} \SO_{2n + 1}/P 
\\
&\hookrightarrow 
(\bbK_1/\bbQ_1 {\times} \bbK_2/\bbQ_2) {\times} \bbG/\bbP  
\\
& \hspace*{.2\textwidth} =  (\GL_{2p + 1}/\bbQ_1 {\times} \GL_{2q}/\bbQ_2) {\times} \GL_{2n + 1}/\bbP = \bbdblFV  .
\end{aligned}
\end{equation*}
Below, we list the table of type BI (BDI) double flag varieties of finite type.

Here is the numbering of simple roots for $ G = \SO_{2n + 1} $ and $ K_2 = \SO_{2 q} $.  
The numbering for $ K_1 $ is similar as $ G $ and the difference is that 
we use $ \{ \beta_1, \dots, \beta_p \} $ instead of $ \alpha_i $'s. 
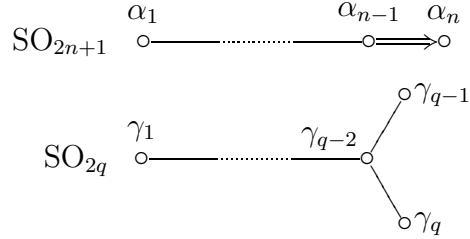
\begin{align*}
\SO_{2n + 1} & 
\begin{xy}
\ar@{-} (0,0) *++!D{\alpha_1} *{\circ}="A"; (10,0)="C" 
\ar@{.} "C"; (20,0)="D" 
\ar@{-} "D"; (30,0) *++!D{\alpha_{n-1}} *{\circ}="E"
\ar@{=>} "E"; (40,0) *++!D{\alpha_n} *{\circ}="F"
\end{xy}
\\
\SO_{2q} &
\begin{xy}
\ar@{-} (0,0) *++!D{\gamma_1} *{\circ}="A"; (10,0)="C" 
\ar@{.} "C"; (20,0)="D" 
\ar@{-} "D"; (30,0) *+!DR{\gamma_{q-2}} *{\circ}="E"
\ar@{-} "E"; (35,8.6)  *+!L{\gamma_{q-1}} *{\circ}
\ar@{-} "E"; (35,-8.6)  *+!L{\gamma_q} *{\circ}
\end{xy}
\end{align*}
\hfil
\resizebox{.99\linewidth}{!}{
%%\begin{longtable}{c|c|c|c}
\begin{tabular}{c|c|c|c}
\hline
\rule[0pt]{0pt}{12pt}
%%$ (G,K) $ & $ \Pi_G \setminus J_G\ (P=P_{J_G}) $ & $\Pi_{K_1} \setminus J_{K_1}\ (Q_1=Q_{J_{K_1}})$ & $\Pi_{K_2} \setminus J_{K_2}\ (Q_2=Q_{J_{K_2}})$
$ (G,K) $
& $ \Pi_G \setminus J_G $ & $\Pi_{K_1} \setminus J_{K_1}$ & $\Pi_{K_2} \setminus J_{K_2}$
\\
& $ (P=P_{J_G}) $ & $(Q_1=Q_{J_{K_1}})$ & $ (Q_2=Q_{J_{K_2}})$
\\
\hline\hline
$ (\SO_{2n +1}, \SO_{2p+1} \times \SO_{2q}) \smallvstrut $
& $ \{ \alpha_1 \}, \{ \alpha_n \} \smallvstrut$ & any subset  & any subset \\[1ex]
\cline{2-4}
& $ \{ \alpha_2 \} \smallvstrut$ & any subset  & $ \{ \gamma_{q - 1} \} , \{ \gamma_q \} $ \\[1ex]
\cline{2-4}
& $ \{ \alpha_i \} \smallvstrut$ & any subset  & \begin{tabular}{l}
$ \emptyset \; (Q_2 = K_2)  \smallvstrut$
\\
$ \{ \gamma_1 \} $ if $ q=1 \smallvstrut$
\end{tabular} \\[1ex]
\cline{2-4}
& $ \{ \alpha_i \} \smallvstrut$ & $ \emptyset \; (Q_1 = K_1) $ & any subset \\[1ex]
\cline{2-4}
& $ \{ \alpha_i \} \smallvstrut$ & $ \{ \beta_j \} $  & $ \{ \gamma_{q - 1} \} , \{ \gamma_q \} $ \\[1ex]
\cline{2-4}
& $ \{ \alpha_i, \alpha_j \} \smallvstrut$ & $ \emptyset \; (Q_1 = K_1) $  & $ \{ \gamma_{q - 1} \} , \{ \gamma_q \} $ \\[1ex]
\cline{2-4}
& any subset  & $ \emptyset \; (Q_1 = K_1) $ & $ \{\gamma_1\} $ if $ q = 1 $\\[1ex]
\cline{2-4}
& any subset  & $ Q_1 = \SO_1 \smallvstrut$ if $ p = 0 $ & any subset \\[1ex]
\hline
\end{tabular}
%%\end{longtable}
}
\hfil

\subsection{Type DI (BDI), even case}\label{subsec:DI-even}

In the setting of \S~\ref{subsec:Embedding.AIII}, 
we take $ 2 n $ instead of $ n $, 
$ 2 p $ (respectively $ 2 q $) instead of $ p $ (respectively $ q $), 
where $ n = p + q $.
Choose
\begin{equation*}
J = \begin{pmatrix}
\Sigma_{2 p} & 0 \\
0 & \Sigma_{2 q}  
\end{pmatrix}
, \;\;
\Sigma_m = \begin{pmatrix}
 & & 1 \\[-1.5ex]
 & \makebox[4ex][c]{$\iddots$} & \\[-1.2ex]
1 & &
\end{pmatrix}, 
\quad
%%\text{ and }\quad
I = \begin{pmatrix}
\unitmatrix_{2 p} & \\
& - \unitmatrix_{2q}
\end{pmatrix} .
\end{equation*}
Thus our embedding becomes 
\begin{equation*}
\begin{aligned}
\dblFV &= (K_1/Q_1 \times K_2/Q_2) \times G/P = (\SO_{2p}/Q_1 \times \SO_{2q}/Q_2) \times \SO_{2n}/P 
\\
&\hookrightarrow 
(\bbK_1/\bbQ_1 \times \bbK_2/\bbQ_2) \times \bbG/\bbP  
\\
& \hspace*{.2\textwidth} =  (\GL_{2p}/\bbQ_1 \times \GL_{2q}/\bbQ_2) \times \GL_{2n}/\bbP = \bbdblFV  .
\end{aligned}
\end{equation*}
Below, we list the table of type DI (BDI) double flag varieties of finite type (even case).

Here is the numbering of simple roots for $ G = \SO_{2n} $.  
The numbering for $ K_1 $ and $ K_2 $ is similar as $ G $ and the difference is that 
we use $ \{ \beta_1, \dots, \beta_p \} $ and $ \{ \gamma_1, \dots, \gamma_q \} $ instead of $ \alpha_i $'s. 
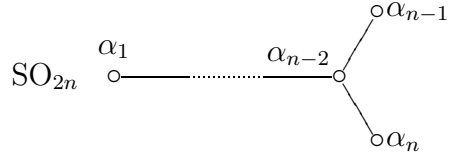
\begin{align*}
\SO_{2n} &
\begin{xy}
\ar@{-} (0,0) *++!D{\alpha_1} *{\circ}="A"; (10,0)="C" 
\ar@{.} "C"; (20,0)="D" 
\ar@{-} "D"; (30,0) *+!DR{\alpha_{n-2}} *{\circ}="E"
\ar@{-} "E"; (35,8.6)  *+!L{\alpha_{n-1}} *{\circ}
\ar@{-} "E"; (35,-8.6)  *+!L{\alpha_n} *{\circ}
\end{xy}
\end{align*}
\noindent
\resizebox{.98\linewidth}{!}{
%%\begin{longtable}{c|c|c|c}
\begin{tabular}{c|c|c|c}
\hline
\rule[0pt]{0pt}{12pt}
%%$ (G,K) $ & $ \Pi_G \setminus J_G\ (P=P_{J_G}) $ & $\Pi_{K_1} \setminus J_{K_1}\ (Q_1=Q_{J_{K_1}})$ & $\Pi_{K_2} \setminus J_{K_2}\ (Q_2=Q_{J_{K_2}})$
$ (G,K) $
& $ \Pi_G \setminus J_G $ & $\Pi_{K_1} \setminus J_{K_1}$ & $\Pi_{K_2} \setminus J_{K_2}$
\\
& $ (P=P_{J_G}) $ & $(Q_1=Q_{J_{K_1}})$ & $ (Q_2=Q_{J_{K_2}})$
\\
\hline\hline
%%$ (\SO_{2n}, \SO_{2p} {\times} \SO_{2q}) \smallvstrut $
$ (\SO_{2n}, $ \qquad\qquad
& $ \{ \alpha_1 \}, \{ \alpha_{n - 1} \} , \{ \alpha_n \} \smallvstrut$ & any subset  & any subset \\[1ex]
\cline{2-4}
$ \SO_{2p} {\times} \SO_{2q}) \smallvstrut $
& $ \{ \alpha_2 \} , {\{ \alpha_{n - 1}, \alpha_n \}} \smallvstrut$ & any subset & $ \{ \gamma_{q - 1} \} , \{ \gamma_q \} $ \\[1ex]
\cline{2-4}
& {$ \{ \alpha_i \} \smallvstrut$} & {any subset} & {$ \{ \gamma_1 \} $ if $q=1$} \\[1ex]
\cline{2-4}
& \begin{tabular}{c} 
$ \{ \alpha_i \}, \{ \alpha_1, \alpha_{n - 1} \} \smallvstrut$, \\
$ \{ \alpha_1, \alpha_n \}, \{ \alpha_{n - 1}, \alpha_n \}  \smallvstrut$ 
  \end{tabular}
& any subset  & $ \emptyset \; (Q_2 {=} K_2) $ \\[1ex]
\cline{2-4} 
& {$ \{ \alpha_i, \alpha_{n - 1} \}, 
\{ \alpha_i, \alpha_n \} \smallvstrut$}
& {any subset}  & {$ \emptyset \; (Q_2 {=} K_2) $ if $q{=}1$} \\[1ex]
\cline{2-4} %%
& $ \{ \alpha_i \} \smallvstrut$ & $ \begin{array}{c} \{ \beta_j \} , \smallvstrut \\ \{ \beta_j , \beta_{p - 1} \} , \{ \beta_j, \beta_p \} \smallvstrut \end{array} $  & $ \{ \gamma_{q - 1} \} , \{ \gamma_q \} $ \\[1ex]
\cline{2-4}
& $ \{ \alpha_i, \alpha_{n - 1} \} , \{ \alpha_i, \alpha_n \} \smallvstrut$  & 
$ \begin{array}{c}
\{ \beta_j \} , \smallvstrut\\ \{ \beta_j, \beta_{p - 1} \} , \{ \beta_j, \beta_p \} \smallvstrut \end{array} $  & $ \emptyset \; (Q_2 {=} K_2) $ \\[1ex]
\cline{2-4}
& \begin{tabular}{c}
$ \{ \alpha_i, \alpha_j\},  \{ \alpha_i, \alpha_j, \alpha_{n - 1} \} \smallvstrut $, \\
$ \{ \alpha_i, \alpha_j, \alpha_n \}, \{ \alpha_i, \alpha_{n-1}, \alpha_n \} \smallvstrut$  
  \end{tabular}
& $ \{ \beta_{p - 1} \} , \{ \beta_p \} $  & $ \emptyset \; (Q_2 {=} K_2) $ \\[1ex]
\cline{2-4}
& {any subset}  & {$ \emptyset \; (Q_1 {=} K_1) $}  & {$ \{ \gamma_1 \} $ if $q{=}1 \smallvstrut$} \\[1ex]
\cline{2-4}
& any subset  & $ \{ \beta_{p - 1} \} , \{ \beta_p \} \smallvstrut $  & $ \emptyset \; (Q_2 {=} K_2) $ if $q{=}1$ \\[1ex]
\hline
\end{tabular}
%%\end{longtable}
}

\subsection{Type DI (BDI), odd case}\label{subsec:DI-odd}

In the setting of \S~\ref{subsec:Embedding.AIII}, 
we take $ 2 n $ instead of $ n $, 
$ 2 p + 1$ (respectively $ 2 q + 1$) instead of $ p $ (respectively $ q $), 
where $ n = p + q + 1$.
Choose
\begin{equation*}
J = \begin{pmatrix}
\Sigma_{2 p+1} & 0 \\
0 & \Sigma_{2 q+1}  
\end{pmatrix}
, \;\;
\Sigma_m = \begin{pmatrix}
 & & 1 \\[-1.5ex]
 & \makebox[4ex][c]{$\iddots$} & \\[-1.2ex]
1 & &
\end{pmatrix}, 
\quad
%%\text{ and }\quad
I = \begin{pmatrix}
\unitmatrix_{2 p+1} & \\
& - \unitmatrix_{2q+1}
\end{pmatrix} .
\end{equation*}
Thus our embedding becomes 
\begin{equation*}
\begin{aligned}
\dblFV &= (K_1/Q_1 \times K_2/Q_2) \times G/P = (\SO_{2p+1}/Q_1 \times \SO_{2q+1}/Q_2) \times \SO_{2n}/P \hookrightarrow 
\\
&
(\bbK_1/\bbQ_1 \times \bbK_2/\bbQ_2) \times \bbG/\bbP  =  (\GL_{2p+1}/\bbQ_1 \times \GL_{2q+1}/\bbQ_2) \times \GL_{2n}/\bbP = \bbdblFV  .
\end{aligned}
\end{equation*}
Below, we list the table of type DI (BDI) double flag varieties of finite type (odd case).

\skipover{
Here is the numbering of simple roots for $ G = \SO_{2n} $ and $ K_1 = \SO_{2 p + 1} $.  
The numbering for $ K_2 $ is similar as $ K_1 $.  
\begin{align*}
\SO_{2n} &
\begin{xy}
\ar@{-} (0,0) *++!D{\alpha_1} *{\circ}="A"; (10,0)="C" 
\ar@{.} "C"; (20,0)="D" 
\ar@{-} "D"; (30,0) *+!DR{\alpha_{n-2}} *{\circ}="E"
\ar@{-} "E"; (35,8.6)  *+!L{\alpha_{n-1}} *{\circ}
\ar@{-} "E"; (35,-8.6)  *+!L{\alpha_n} *{\circ}
\end{xy}
&
\SO_{2p + 1} &
\begin{xy}
\ar@{-} (0,0) *++!D{\beta_1} *{\circ}="A"; (10,0)="C" 
\ar@{.} "C"; (20,0)="D" 
\ar@{-} "D"; (30,0) *++!D{\beta_{p-1}} *{\circ}="E"
\ar@{=>} "E"; (40,0) *++!D{\beta_p} *{\circ}="F"
\end{xy}
\end{align*}
}

The numbering of simple roots for $ G = \SO_{2n} $ and $ K_1 = \SO_{2 p + 1} $ 
are the same as that of type BI and DI (even case) except for the labeling Greek letters.

\noindent
\resizebox{\linewidth}{!}{
%%\begin{longtable}{c|c|c|c}
\begin{tabular}{c|c|c|c}
\hline
\rule[0pt]{0pt}{12pt}
%%$ (G,K) $ & $ \Pi_G \setminus J_G\ (P=P_{J_G}) $ & $\Pi_{K_1} \setminus J_{K_1}\ (Q_1=Q_{J_{K_1}})$ & $\Pi_{K_2} \setminus J_{K_2}\ (Q_2=Q_{J_{K_2}})$
$ (G,K) $
& $ \Pi_G \setminus J_G $ & $\Pi_{K_1} \setminus J_{K_1}$ & $\Pi_{K_2} \setminus J_{K_2}$
\\
& $ (P=P_{J_G}) $ & $(Q_1=Q_{J_{K_1}})$ & $ (Q_2=Q_{J_{K_2}})$
\\
\hline\hline
$ \begin{array}{l}(\SO_{2n}, \\ \SO_{2p + 1} \times \SO_{2q + 1}) \end{array} \smallvstrut $
& $ \{ \alpha_1 \}, \{ \alpha_{n - 1} \} , \{ \alpha_n \} \smallvstrut$ & any subset  & any subset \\[1ex]
\cline{2-4}
& \begin{tabular}{c} $ \{ \alpha_i \}, \{ \alpha_{n - 1}, \alpha_n \} \smallvstrut$ \\ {$\{\alpha_1,\alpha_{n-1}\},\{\alpha_1,\alpha_n\} \smallvstrut$}
\end{tabular}& any subset  & $ \emptyset \; (Q_2 = K_2) $ \\[1ex]
\cline{2-4}
& $ \{ \alpha_i, \alpha_{n - 1} \} , \{ \alpha_i, \alpha_n \} \smallvstrut$  & $ \{ \beta_j \} $  & $ \emptyset \; (Q_2 = K_2) $ \\[1ex]
\cline{2-4}
& any subset  & any subset  & $ K_2 = SO_1 \smallvstrut $ if $ q = 0 $ \\[1ex]
\hline
\end{tabular}
%%\end{longtable}
}

\subsection{\;\;Type DIII}\label{subsec:DIII}

In the setting of \S~\ref{subsec:Embedding.AIII}, 
we take $ 2 n $ instead of $ n $, 
$ n $ instead of $ p $ (respectively $ q $).  
Choose
\begin{equation*}
J = \Sigma_{2 n} 
= \begin{pmatrix}
 & & 1 \\[-1.5ex]
 & \makebox[4ex][c]{$\iddots$} & \\[-1.2ex]
1 & &
\end{pmatrix}, 
\quad
%%\text{ and }\quad
I = \begin{pmatrix}
\unitmatrix_{n} & \\
& - \unitmatrix_{n}
\end{pmatrix} .
\end{equation*}
Thus our embedding becomes 
\begin{equation*}
\begin{aligned}
\dblFV &= K/Q \times G/P = \GL_n/Q \times \SO_{2n}/P \hookrightarrow 
\\
&
(\bbK_1/\bbQ_1 \times \bbK_2/\bbQ_2) \times \bbG/\bbP  =  (\GL_n/\bbQ_1 \times \GL_n/\bbQ_2) \times \GL_{2n}/\bbP = \bbdblFV  .
\end{aligned}
\end{equation*}
Below, we list the table of type DIII double flag varieties of finite type.

\skipover{
Here is the numbering of simple roots for $ G = \SO_{2n} $ and $ K = \GL_n $.
\begin{align*}
\SO_{2n} &
\begin{xy}
\ar@{-} (0,0) *++!D{\alpha_1} *{\circ}="A"; (10,0)="C" 
\ar@{.} "C"; (20,0)="D" 
\ar@{-} "D"; (30,0) *+!DR{\alpha_{n-2}} *{\circ}="E"
\ar@{-} "E"; (35,8.6)  *+!L{\alpha_{n-1}} *{\circ}
\ar@{-} "E"; (35,-8.6)  *+!L{\alpha_n} *{\circ}
\end{xy}
&
\GL_n & 
\begin{xy}
\ar@{-} (0,0) *++!D{\beta_1} *{\circ}="A"; (10,0) *++!D{\beta_2} 
 *{\circ}="B"
\ar@{-} "B"; (20,0)="C" 
\ar@{.} "C"; (30,0)="D" 
\ar@{-} "D"; (40,0) *++!D{\beta_{n-1}} *{\circ}="E"
\end{xy}
\end{align*}
}
The numbering of simple roots for $ G = \SO_{2n} $ and $ K = \GL_n $ are the same as that of type DI and type CI respectively.
\begin{longtable}{c|c|cc}
\hline
\rule[0pt]{0pt}{12pt}
%%$ (G,K) $ & $ \Pi_G \setminus J_G\ (P=P_{J_G}) $ & $\Pi_{K_1} \setminus J_{K_1}\ (Q_1=Q_{J_{K_1}})$ & $\Pi_{K_2} \setminus J_{K_2}\ (Q_2=Q_{J_{K_2}})$
$ (G,K) $
& $ \Pi_G \setminus J_G \ (P=P_{J_G})$ & $\Pi_K \setminus J_K\ (Q=Q_{J_K})$
\\
\hline\hline
$ (\SO_{2n}, \GL_n) \smallvstrut $
& $ \{ \alpha_1 \}, \{ \alpha_{n - 1} \} , \{ \alpha_n \} \smallvstrut$ & any subset \\[1ex]
\cline{2-4}
& $ \{ \alpha_i \} \smallvstrut$ & $ \{ \beta_j \} $ \\[1ex]
\hline
\end{longtable}

{Here in this table, the embedding theory does not give the cases} where $ P $ is a Borel subgroup (hence can be any parabolic subgroup) and 
$ \Pi_K \setminus J_K = \{ \beta_1 \} $ or {$ \{ \beta_{n-1} \} $} (mirabolic case), 
which is of finite type and listed in \S~\ref{subsection:Table.P=Borel}.

\section{An example of embedding theory for type CI into type AIII}\label{section:embedding.typeCI.into.typeAIII}

In this section, we give an example of the embedding of the case of type CI into type AIII based on \cite{Fresse.N.2021}.  
As in \S~\ref{subsec:CI}, let us begin with 
$ \bbG = \GL_{2n} $ and commuting involutions $ \sigma, \theta \in \Aut \bbG $ defined by 
\begin{align*}
\theta(g) = I^{-1}\, g\, I & & I = I_{n,n} = \diag (\unitmatrix_n, -\unitmatrix_n) ,
\\
\sigma(g) = J^{-1}\,\transpose{g}^{-1} J & & J = J_n = \mattwo{}{-\unitmatrix_n}{\unitmatrix_n}{} .
\end{align*}
Note that we use a convention for $ J $ which is different from the one in \S~\ref{sec:embedding.theory.DFV}.  
We define a symplectic form on $ V = \C^{2n} $ by $ (u, v) = \transpose{u} J v $.  
Note that $ V^+ = \langle \eb_1, \eb_2, \dots, \eb_n \rangle \subset V $ is a Lagrangian subspace 
with respect to this symplectic form, 
where $ \eb_k $ denotes the $ k $-th elementary basis vector as usual.  
Then 
\begin{equation}\label{eq:bbK.G.K.typeCI.embedding}
\bbK = \bbG^{\theta} = \GL_n \times \GL_n , \quad 
G = \bbG^{\sigma} = \Sp_{2n} , \quad
K = G^{\theta} \simeq \GL_n
\end{equation}
are all connected, and $ (G, K) = (\Sp_{2n}, \GL_n) $ is a symmetric pair of type CI 
and $ (\bbG, \bbK) = (\GL_{2n}, \GL_n \times \GL_n) $ is that of type AIII.  

Let us denote by $ B_n^{\pm} $ the Borel subgroup of $ \GL_n $ consisting of upper/lower triangular matrices.
Define 
\begin{equation*}
\bbBK = B_n^+ \times B_n^- \subset \bbK = \GL_n \times \GL_n, \qquad \bbP = \Stab_{\bbG}(V^+) \subset \bbG = \GL_{2n} 
\end{equation*}
both of which are $ \sigma $-stable.  
%% Borel subgroup of $ \bbK $ and a $ \sigma $-stable parabolic subgroup of $ \bbG $ respectively.  
Then $ Q = B_K = \bbBK^{\sigma} $ is a Borel subgroup in $ K $
and $ P = P_S = \bbP^{\sigma} $ is the Siegel parabolic subgroup of $ G = \Sp_{2n} $ stabilizing 
a Lagrangian subspace $ V^+ $.  
%%Similarly we put $ V^- = \langle \eb_{n + 1}, \eb_{n + 2}, \dots, \eb_{2 n} \rangle $, and we get a polarization $ V = V^+ \oplus V^- $ stable under $ \bbK $. 
%
Thus our double flag varieties are 
\begin{equation}\label{eq:dblflag.Sp2n.GLn}
\dblFV =  K/B_K \times G/P = \GL_n/B_n^+ \times \Sp_{2n}/P_S 
\simeq \FlagVar_n \times \LGrass(\C^{2n})
\end{equation}
and 
\begin{equation}\label{eq:dblflag.GL2n.GLnGLn}
\begin{aligned}[t]
\bbdblFV &= \bbK/\bbBK \times \bbG/\bbP = \bigl( \GL_n/B_n^+ \times \GL_n/B_n^- \bigr) \times \GL_{2n}/P_{(n,n)} 
\\  &
\simeq \bigl( \FlagVar_n \times \FlagVar_n \bigr) \times \Grass_n(\C^{2n}),
\end{aligned}
\end{equation}
where $ \LGrass(V) $ denotes the Lagrangian Grassmannian, the variety of all the Lagrangian subspaces 
in $ V = \C^{2n} $, 
%% with the symplectic form defined by $ J $, 
and $ \FlagVar_n $ is the variety of complete flags in $ \C^n $.  
We use the notation $ P_{(n,n)} $ for the parabolic subgroup in $ \GL_{2n} $ determined by 
the partition $ (n, n) $ of $ 2n $.

Proposition~\ref{prop:condition.E} enables orbit embedding in this case.

\begin{theorem}\label{thm:abstract.embedding.typeCI.to.typeAIII}
Let $ \dblFV $ and $ \bbdblFV $ be the double flag varieties defined in 
\eqref{eq:dblflag.Sp2n.GLn} and 
\eqref{eq:dblflag.GL2n.GLnGLn} respectively.  
Then the orbit map $ \dblFV / K \to \bbdblFV / \bbK $ is injective, i.e., 
for any $ \bbK $-orbit $ \orbit $ in $ \bbdblFV $, 
the intersection $ X \cap \orbit $ is either empty or a single $ K $-orbit.
\end{theorem}

Let us describe the explicit embedding of orbits in this case. 
First, we clarify the explicit embedding map $ \dblFV \to \bbdblFV $ and the involutive automorphism on $ \bbdblFV $. %
Note the following isomorphism from 
%% the various identifications (see 
\S\S~\ref{subsec:orbit.by.quivers.typeAIII} and \ref{section:combi.quotient}.
\begin{equation}\label{eqn:bbdblFV.bbK.various.forms.typeAIII}
\begin{aligned}
\bbdblFV / \bbK &\simeq B_n^+ {\times} B_n^- \backslash \GL_{2n} / P_{(n,n)} 
\simeq B_n^+ {\times} B_n^- \backslash \Grass_n(\C^{2n}) 
\\  &
\simeq B_n^+ {\times} B_n^- \backslash \regMat_{2n,n} / \GL_n \simeq \regTnxTn / S_n ,
\end{aligned}
\end{equation}
where $ \regMat_{2n,n} $ denotes the set of $ 2n $ by $ n $ matrices of full rank (i.e., rank $ n $), 
and $ \ppermutations^{\circ} = \regTnxTn $ is the set of pairs of partial permutation matrices which are full rank 
(see \S~\ref{subsec:orbit.by.quivers.typeAIII}).  
Note that we put the subscript $ n $.  

Let us explain the above isomorphism more precisely.  
Take $ \omega = \vectwo{\tau_1}{\tau_2} \in \regTnxTn $.  
Then we can find matrices $ \xi_1, \xi_2 \in \Mat_n $ such that 
\begin{equation*}
\bbg = \mattwo{\tau_1}{\xi_1}{\tau_2}{-\xi_2} \in \GL_{2n} = \bbG
\end{equation*}
is a regular matrix.  
Note that the coset $ \bbg \bbP $ does not depend on the choice of $ \xi_1, \xi_2 $ 
but only depending on $ \omega $.  
The $ \bbK $-orbit $ \orbit $ which goes through 
$ (\bbBK, \bbg \bbP) \in \bbdblFV $ is denoted by $ \orbit_{\omega} $.  

There is a great possibility for the choice of $ \xi_1, \xi_2 $.  
In fact, we can choose $ \xi_1, \xi_2 $ in such a way that 
$ \vectwo{\xi_1}{\xi_2} \in \regTnxTn  $ is also a pair of partial permutations of full rank and 
it satisfies
\begin{equation} 
\transpose{\xi_1} \tau_1 - \transpose{\xi_2} \tau_2 = 0 .
\end{equation}
For this, see \cite[Lemma~4.4]{Fresse.N.2021}.  
In the following, we always assume this.

We define an involution on $ \regTnxTn / S_n \simeq \bbdblFV / \bbK $ by 
$ \sigma(\orbit_{\omega}) = \orbit_{\sigma(\omega)} \; (\omega \in \regTnxTn) $, 
which is denoted by the same letter $ \sigma $ by abuse of notation.  
Note that $ \sigma(\omega) $ is determined modulo the right multiplication by $ S_n $.  

\begin{proposition}\label{prop:explicit.sigma.on.bbdblFV.bbK}
For $ \omega = \vectwo{\tau_1}{\tau_2} \in \regTnxTn $ and 
$ \bbg = \mattwo{\tau_1}{\xi_1}{\tau_2}{-\xi_2} $ chosen above,  
let $ \orbit_{\omega} $ be the $ \bbK $-orbit in $ \bbdblFV $ through the point 
$ (\bbBK, \bbg \bbP) $.  
Then $ \sigma(\orbit_{\omega}) = \orbit_{\sigma(\omega)} $ is given by 
\begin{equation*}
\sigma(\omega) = \sigma\Bigl( \vectwo{\tau_1}{\tau_2} \Bigr)
= \vectwo{\xi_2}{\xi_1} \pmod{S_n}.
\end{equation*}
\end{proposition}

\begin{proof}
We have $ \sigma\bigl( (\bbBK, \bbg \bbP) \bigr) = (\bbBK, \sigma(\bbg) \bbP) $ and 
$ \sigma(\bbg) = J_n^{-1} \transpose{\bbg}^{-1} J_n $. 
Let us compute $ \transpose{\bbg}^{-1} $.  We get 
\begin{align*}
\transpose{\bbg} \bbg 
&= \mattwo{\transpose{\tau_1}}{\transpose{\tau_2}}{\transpose{\xi_1}}{-\transpose{\xi_2}} 
  \mattwo{\tau_1}{\xi_1}{\tau_2}{-\xi_2} 
\\
&= \mattwo{\transpose{\tau_1} \tau_1 + \transpose{\tau_2} \tau_2}{\transpose{\tau_1} \xi_1 - \transpose{\tau_2} \xi_2}{\transpose{\xi_1} \tau_1 - \transpose{\xi_2} \tau_2}{\transpose{\xi_1} \xi_1 + \transpose{\xi_2} \xi_2}
=: \mattwo{d_1}{0}{0}{d_2} , 
\end{align*}
by the property $ \transpose{\xi_1} \tau_1 - \transpose{\xi_2} \tau_2 = 0 $.  
An easy calculation tells that $ d_1 = \transpose{\tau_1} \tau_1 + \transpose{\tau_2} \tau_2 $ 
is a diagonal matrix with diagonal entries $ 1 $ or $ 2 $, 
and so is $ d_2 = \transpose{\xi_1} \xi_1 + \transpose{\xi_2} \xi_2 $.  
Thus we get 
$ \transpose{\bbg}^{-1} = \bbg \bbd^{-1} $ with $ \bbd = \diag (d_1, d_2) $.  
From this, we compute 
\begin{equation*}
\sigma(\bbg) = J_n^{-1} \transpose{\bbg}^{-1} J_n 
= J_n^{-1} \bbg \bbd^{-1} J_n 
= \mattwo{\xi_2}{\tau_2}{\xi_1}{-\tau_1}
\cdot \diag(-d_2^{-1}, -d_1^{-1}) , 
\end{equation*}
and get $ \sigma(\bbg) \bbP = \mattwo{\xi_2}{\tau_2}{\xi_1}{-\tau_1} \bbP $.  
This implies $ \sigma\Bigl( \vectwo{\tau_1}{\tau_2} \Bigr)
= \vectwo{\xi_2}{\xi_1} $.
\end{proof}

\begin{theorem}\label{thm:explicit.embedding.typeCI.to.typeAIII}
For $ \omega = \vectwo{\tau_1}{\tau_2} \in \regTnxTn / S_n $, 
let $ \orbit_{\omega} $ be the corresponding $ \bbK $-orbit in $ \bbdblFV $.  
Then the following {\upshape(1)--(4)} are all equivalent.
\begin{penumerate}
\item
$ \orbit_{\omega} \cap \dblFV \neq \emptyset $, hence it is a single $ K $-orbit.
\item
$ \sigma( \orbit_{\omega} ) = \orbit_{\omega} $, i.e., the $ \bbK $-orbit is $ \sigma $-stable.
\item
$ \transpose{\tau_1} \tau_2 \in \Sym_n $.
\item
$ \transpose{\tau_2} \tau_1 \in \Sym_n $.
\end{penumerate}

\smallskip

In particular, the set of $ K $-orbits in the double flag variety 
$ \dblFV = \GL_n / B_n^+ \times \Sp_{2n} / P_S $ 
of type {\upshape{}CI} is parametrized by $ \regCnxCn / S_n $, where 
\begin{equation*}
\regCnxCn := 
\bigl\{ \omega = \vectwo{\tau_1}{\tau_2} \in \regTnxTn \bigm| 
\transpose{\tau_1} \tau_2 = \transpose{\tau_2} \tau_1 \in \Sym_n \bigr\} .
\end{equation*}
\end{theorem}

\begin{proof}
{For (2)$\Rightarrow$(3) we use that $ \sigma(\omega) = \omega \mod{S_n} $, which means} 
\begin{equation*}
\sigma(\omega) = \vectwo{\xi_2}{\xi_1} = \vectwo{\tau_1 s}{\tau_2 s} 
\end{equation*}
for some $ s \in S_n $.
{The other implications are straightforward.}
\end{proof}

As in the case of $ \bbdblFV / \bbK $ (see \eqref{eqn:bbdblFV.bbK.various.forms.typeAIII}), 
there are natural bijections
\begin{equation}\label{eqn:dblFV.K.various.forms.typeCI}
\begin{aligned}
\dblFV / K &\simeq (\FlagVar_n \times \LGrass(\C^{2n})) / \GL_n 
\simeq B_n^+ \backslash \Sp_{2n} / P_S \simeq B_n^+ \backslash \LGrass(\C^{2n}) 
\\
&
\simeq B_n^+ \backslash \regCMat_{2n,n} / \GL_n \simeq \regCnxCn / S_n,
\end{aligned}
\end{equation}
where 
\begin{equation}
\begin{aligned}
\regCMat_{2n,n} 
&= \{ A \in \regMat_{2n,n} \mid \transpose{A} J_n A = 0 \} 
\\
&= \{ A = \vectwo{A_1}{A_2} \in \regMat_{2n,n} \mid \transpose{A_1} A_2 = \transpose{A_2} A_1 
\in \Sym_n \} . 
\end{aligned}
\end{equation}
Note that the actions of $ b \in B_n^+ $ are all defined by the left multiplications by 
$ \mattwo{b}{0}{0}{\transpose{b}^{-1}} $.  
The above theorem tells that all these coset spaces are in bijection with 
$ \regCnxCn / S_n $.  

%%[2023/08/04 17:50:58 JST]

\section*{Acknowledgments}

The authors thank the referee for his/her precise reading, and useful suggestions.  
The second author (K.~N.) is supported by JSPS KAKENHI Grant Number \#{21K03184}.

\bibliographystyle{amsalpha}
%%\bibliography{bib_DFV.bib}
%%\bibliography{bib_exotic_nullcone,/Users/kyo/home/math/mypaper/bib_nishiyama.bib}

\def\cprime{$'$} \def\Dbar{\leavevmode\lower.6ex\hbox to 0pt{\hskip-.23ex
  \accent"16\hss}D}
\providecommand{\bysame}{\leavevmode\hbox to3em{\hrulefill}\thinspace}
\providecommand{\MR}{\relax\ifhmode\unskip\space\fi MR }
\providecommand{\MRhref}[2]{%
  \href{http://www.ams.org/mathscinet-getitem?mr=#1}{#2}
}
\providecommand{\href}[2]{#2}

\renewcommand{\MR}[1]{}

%% for biblatex
%%\printbibliography

\end{document}